\documentclass[11pt,reqno]{amsart} 

\usepackage[margin=1in]{geometry}  
\usepackage[utf8]{inputenc}  
\usepackage{amsmath,amsthm,amsfonts,amssymb}
\numberwithin{equation}{section}
\usepackage{times}
\usepackage{enumitem}
\setlist[enumerate]{label=(\roman*), ref=\roman*}
\usepackage{mathrsfs}
\usepackage{comment}
\usepackage{graphicx}
\graphicspath{{./}{paper/}}
\usepackage{algpseudocode}
\usepackage{algorithm}
\usepackage{algorithmicx}

\usepackage{caption}
\usepackage{subcaption}
\usepackage{makecell}
\usepackage{xcolor}
\usepackage{fancyvrb}

\usepackage{tikz}
\usetikzlibrary{arrows.meta, decorations.pathmorphing, decorations.markings, patterns,calc,arrows}
\usepackage{booktabs,longtable}

\usepackage[nosort]{cite}
\usepackage[colorlinks=true, pdfstartview=FitV, linkcolor=blue,citecolor=blue, urlcolor=blue]{hyperref}
\usepackage[nameinlink,capitalise]{cleveref}
\hypersetup{pdftitle={explosions}, pdfauthor={Jiajie Chen, Giorgio Cialdea, Steve Shkoller, Vlad Vicol}}

\newtheorem{theorem}{Theorem}[section]
\newtheorem{lemma}[theorem]{Lemma}
\newtheorem{proposition}[theorem]{Proposition}
\newtheorem{corollary}[theorem]{Corollary}
\newtheorem{definition}[theorem]{Definition}

\newtheorem{remark}[theorem]{Remark}

\newcommand{\p}{\partial}

\newcommand{\uu}{\mathbf{u}}

\newcommand{\s}{\mathsf{s}}
\newcommand{\ds}{\dot{\mathsf{s}}}

\newcommand{\usp}{u{\scriptstyle|_{\s}^+}}

\newcommand{\rsp}{\rho{\scriptstyle|_{\s}^+}}

\newcommand{\rsm}{\rho{\scriptstyle|_{\s}^-}}

\newcommand{\usm}{u{\scriptstyle|_{\s}^-}}

\newcommand{\ssm}{\sigma{\scriptstyle|_{\s}^-}}
\newcommand{\ssp}{\sigma{\scriptstyle|_{\s}^+}}

\newcommand{\crr}{\mathsf{c}_\mathsf{r}}

\newcommand{\crrs}{\crr^*}

\newcommand{\cbb}{\mathsf{c}_\mathsf{b}}

\newcommand{\cbbs}{\cbb^*}

\newcommand{\rr}{\mathsf{r}}
\newcommand{\PP}{\mathsf{P}}
\newcommand{\cp}{\mathsf{p}}

\newcommand{\QQ}{\mathsf{Q}}
\newcommand{\cq}{\mathsf{q}}

\newcommand{\BB}{\mathfrak{b}}
\newcommand{\NNN}{\mathsf{N}}
\newcommand{\xiRH}{\xi_{\mathrm{RH}}}
\newcommand{\VR}{V_{\mathrm{RH}}}
\newcommand{\QR}{Q_{\mathrm{RH}}}

\newcommand{\mm}{\mathsf{m}}
\newcommand{\mmp}{\mathsf{m}_+}
\newcommand{\mmm}{\mathsf{m}_-}

\newcommand{\sage}{\textsc{SageMath}}
\newcommand{\RBF}{\ensuremath{\mathrm{RBF}}}

\newcommand{\pp}{\partial}

\newcommand{\RR}{\mathbb{R}}

\def\jump#1{{[\hspace{-1.5pt}[#1]\hspace{-1.5pt}]}}

\title{A new class of Euler explosions}

\author{Jiajie Chen}
\address{Department of Mathematics, University of Chicago, Chicago, IL 60637.}
\email{\href{jiajiechen@uchicago.edu}{jiajiechen@uchicago.edu}}

\author{Giorgio Cialdea}
\address{Courant Institute, Department of Mathematics, New York University, New York, NY 10012.}
\email{\href{giorgio.cialdea@cims.nyu.edu}{giorgio.cialdea@cims.nyu.edu}}

\author{Steve Shkoller}
\address{Department of Mathematics, University of California Davis, Davis, CA 95616.}
\email{\href{shkoller@math.ucdavis.edu}{shkoller@math.ucdavis.edu}}

\author{Vlad Vicol}
\address{Courant Institute, Department of Mathematics, New York University, New York, NY 10012.}
\email{\href{vicol@cims.nyu.edu}{vicol@cims.nyu.edu}}

\date{\today}

\begin{document}

\begin{abstract}
We study the \emph{global-in-time continuation}, past the singularity, of the smooth, non-isentropic, radially symmetric imploding solutions of the compressible Euler equations recently constructed by Chen, Shkoller, and Vicol \cite{CSV26}.
In three space dimensions, for all physically relevant adiabatic exponents $\gamma>1$, we consider the Euler solution that evolves smoothly until an implosion singularity forms at the origin at time $t=0$. We then prove that this solution can be uniquely continued for $t>0$ as a \emph{reflected outward-propagating shock}, sometimes called a \emph{reflected blast wave}.
For $t>0$,  the continuation is a globally \emph{forward self-similar} weak solution of the Euler equations, selected by the Rankine--Hugoniot conditions and the Lax entropy inequality; it is smooth away from the expanding shock sphere and the spatial origin.
The structure at the center of symmetry distinguishes these explosions from the classical Guderley reflected shock. In Guderley's continuation, the reflected blast wave leaves a \emph{point vacuum} at the origin, where the density vanishes. The solutions constructed here exhibit the \emph{opposite} behavior: for every fixed $t>0$ the density is \emph{unbounded} at $r=0$ (though it remains locally integrable), while the pressure stays bounded and the temperature vanishes there.
\end{abstract}

\maketitle

\setcounter{tocdepth}{1}
\tableofcontents


\section{Introduction} 
\label{introduction}
Implosions form a remarkable class of singularities in compressible fluid motion, in which at least one of the primary thermodynamic flow variables (density, pressure, temperature) becomes unbounded at a single point in spacetime.
A classical problem in fluid dynamics is to describe the evolution of such solutions beyond the singularity time; see~\cite{Gud42,Sed59, Sta60,Laz81,Chi98,Bar96}.
Describing the solution after the singularity has formed, the \emph{life after death of the smooth flow}, is at the forefront of modern PDE theory, and is most demanding precisely when this continuation is global-in-time.\footnote{The most thoroughly studied instance of this problem is shock formation, in which an initially smooth compressible flow develops a gradient catastrophe in finite time: the solution itself remains bounded while its first derivatives blow up. The structure of the spacetime beyond this first singularity, which forms the boundary of the maximal globally hyperbolic development of the smooth data, is the subject of the maximal development problem~\cite{ShVi24}.  By contrast, an implosion is a stronger form of singularity, at which the flow variables themselves, rather than only their derivatives, become unbounded.}

For implosions this continuation takes a concrete form: it leads to the study of explosions, the outward-propagating blast waves generated by the singular state produced at the collapse.\footnote{The most famous self-similar explosion solution of the Euler equations is the Sedov--von Neumann--Taylor blast wave~\cite{Sed46,vN47,Tay50a}. It is currently \emph{not known} whether this solution arises as the continuation of an Euler implosion.}
The question has recently returned to the forefront, motivated by the rigorous constructions of smooth, radially symmetric imploding solutions to the compressible Euler equations~\cite{MRRS22a,CSV26}. Once such singularities form from smooth initial data, it is natural to ask whether the solution admits a global-in-time continuation past the collapse, and what structure this solution has at the center of symmetry, where the implosion was concentrated.

In this paper we analyze implosions and explosions in the context of the compressible Euler equations. For the unknown density $\rho$, momentum $\rho \uu$, and total energy $E$, these equations are given by the system of conservation laws
\begin{subequations} \label{eq:euler:cl}
 \begin{align}  
 \partial_t \rho +  \operatorname{div}(\rho   \uu) & = 0  \,, \\ 
 \partial_t (\rho \uu) + \operatorname{div}(\rho \uu \otimes \uu  +  p I) & = 0 \,,  \\ 
 \partial_t E  + \operatorname{div} ((p+E) \uu) &=0 \,,  
 \end{align}
 \end{subequations}  
 where the pressure $p$ is given by 
 \begin{equation}  \label{def:pres}
 p := (\gamma-1) \bigl( E - \tfrac 12 \rho |\uu|^2 \bigr),
 \end{equation} 
and $\gamma>1$ is the adiabatic exponent.
We also set $\alpha:=\frac{\gamma-1}{2}$ and introduce the square root of the pseudo-entropy\footnote{Usually, the quantity $\frac{\gamma p}{\rho^\gamma}$ is denoted as $e^S$, where $S$ is the specific entropy.}
 \begin{equation} \label{def:b}
b:= \sqrt{\tfrac{\gamma p}{\rho^\gamma}}
\end{equation} 
and the rescaled sound speed
\begin{equation} \label{def:sigma}
\sigma := \tfrac{1}{\alpha} \rho^\alpha b \,.
\end{equation}

\subsection{Self-similar ansatz}
All known implosions and consequent explosions for the Euler equations arise as examples of globally self-similar solutions in radial/spherical symmetry, where the $(d+1)$-dimensional PDE reduces to a system of ODEs. 
By making a radial/spherical ansatz
\begin{equation*} 
 \sigma(x,t) =  \sigma(|x|,t), \quad b(x,t) =  b(|x|,t), \quad \uu(x,t) = \tfrac{x}{|x|} u^r(|x|,t),
\end{equation*}
where $|x|=r$, and using that $\operatorname{div}\uu= \pp_r u^r + \tfrac{d-1}{r} u^r$, we may write \eqref{eq:euler:cl} as
\begin{subequations}\label{eq:euler:rad}
     \begin{align}
\partial_t \sigma + u^r \, \partial_r \sigma + \alpha \sigma \left( \partial_r u^r + \tfrac{d-1}{r}u^r \right) &= 0
\,, \\
\partial_t u^r + u^r \, \partial_r u^r + \alpha \sigma \partial_r \sigma - \tfrac{\alpha \sigma^2}{\gamma b} \partial_r b &= 0 \,, \\
\partial_t b + u^r \, \partial_r b &= 0 \,.
\end{align}\end{subequations}
We now introduce the self-similar variables used throughout the paper. Following Lazarus~\cite{Laz81}, one could introduce the global self-similar variable $x=\frac{t}{r^{\lambda}}$, where $x<0$ corresponds to the implosion phase and $x>0$ corresponds to the explosion phase. In order to remain consistent with~\cite{CSV26}, we instead use two~\emph{separate self-similar variables}: 
\begin{align*}
&R>0 \qquad \mbox{is used in the implosion phase and is defined in }~\eqref{radial:self:similar},\\ 
&\xi>0  \qquad \mbox{is used for the explosion phase and is defined in }~\eqref{def:xi}.
\end{align*}

\subsubsection{Globally self-similar implosions: backward self-similar ansatz}

Given the similarity exponents $\crr>0$ and $\cbb \in \RR$, we say a solution $(\bar u, \bar \sigma, \bar b)$ is a globally self-similar implosion for $t<0$ if it is of the form\footnote{In the notation of Lazarus~\cite{Laz81}, see also~\cite{JK18,JT23,Jen25,JLS25a}, the exponent $\lambda$ corresponds to $\frac{1}{\crr}$, while $\kappa$ corresponds to $\frac{\crr-1-\cbb}{\alpha\crr}$.}
\begin{subequations}\label{bw:ss}
\begin{align} 
\bar u(r,t) &:= (-t)^{\crr -1} \bar U\left( \tfrac{r}{(-t)^{\crr}} \right), \\
 \bar \sigma(r,t) &:= (-t)^{\crr-1} \bar \Sigma\left( \tfrac{r}{(-t)^{\crr}} \right), \\ 
  \bar b(r,t) &:= (-t)^{\cbb} \bar B\left( \tfrac{r}{(-t)^{\crr}} \right).
\end{align}
\end{subequations}
By introducing the (backward) self-similar radial variable
\begin{equation} \label{radial:self:similar}
R := \tfrac{r}{(-t)^{\crr}}, 
\end{equation}
the system \eqref{eq:euler:rad} reduces to a system of ODEs for $(\bar U, \bar \Sigma, \bar B)$ in the variable $R$
\begin{subequations} \label{bw:ss:euler}
\begin{align}
 (1-\crr) \bar \Sigma + \crr R \pp_R \bar \Sigma + \bar U \pp_R \bar \Sigma + \alpha \bar \Sigma ( \pp_R \bar U + \tfrac{d-1}{R} \bar U) &=0,  \\
(1- \crr) \bar U + \crr R \pp_R \bar U + \bar U \pp_R \bar U + \alpha \bar \Sigma \pp_R \bar \Sigma - \tfrac{\alpha}{\gamma} \bar \Sigma^2 \tfrac{\pp_R \bar B}{\bar B}&=0, \\
- \cbb \bar B + \crr R \pp_R \bar B + \bar U \pp_R \bar B &=0. \label{bw:ss:euler:B}
\end{align}\end{subequations}

\subsubsection{Globally self-similar explosions: forward self-similar ansatz}

Given the similarity exponents $\crr>0$ and $\cbb \in \RR$, we say a solution $( u,  \sigma,  b)$ is a globally self-similar solution to \eqref{eq:euler:rad} for $t>0$ if it is of the form 
\begin{subequations}
 \label{forward:profiles}
\begin{align}
u(r,t) &:= t^{\crr-1} U\left( \tfrac{r}{t^{\crr}}\right) 
\,,
\\
\sigma(r,t) &:= t^{\crr-1} \Sigma \left( \tfrac{r}{t^{\crr}}\right)
\,,
\\
b(r,t) &:= t^{\cbb} B\left( \tfrac{r}{t^{\crr}}\right)
\,.
\end{align}
\end{subequations}
Introduce the (forward) self-similar radial variable
\begin{equation} \label{def:xi}
\xi :=\tfrac{r}{t^{\crr}} \,.
\end{equation}
Then, similarly to \eqref{bw:ss:euler},
the system of PDEs \eqref{eq:euler:rad} reduces to a system of ODEs for the functions $( U, \Sigma, B)$ in the self-similar variable $\xi$
\begin{subequations} \label{fw:ss:eq}
\begin{align}
(U(\xi) - \crr \xi) \pp_\xi \Sigma(\xi) + (\crr-1) \Sigma(\xi) + \alpha \Sigma(\xi) \left( \pp_\xi U(\xi) + \tfrac{d-1}{\xi}U(\xi) \right) & = 0
\,,\\
(U(\xi) - \crr \xi) \pp_\xi U(\xi) + (\crr-1) U(\xi) + \alpha \Sigma(\xi) \pp_\xi \Sigma(\xi) - \tfrac{\alpha \Sigma(\xi)^2}{\gamma B(\xi)} \pp_\xi B(\xi) & = 0
\,, \\
(U(\xi) - \crr \xi) \pp_\xi B(\xi) + \cbb B(\xi) &= 0\,.
\end{align}
\end{subequations}

\begin{remark} \label{remark:time:rev}
     Since the Euler equations are time-reversible for smooth solutions under the transformation
     \begin{equation*}
(\uu(x,t), \sigma(x,t), b(x,t)) \mapsto (-\uu(x,-t), \sigma(x,-t), b(x,-t))
\end{equation*}
the backward self-similar system \eqref{bw:ss:euler} and the forward self-similar system \eqref{fw:ss:eq} are equivalent in the class of smooth solutions.
It is important to note that the two systems are \emph{not equivalent} when considering solutions that contain shocks, due to the arrow of time induced by the entropy inequality.
\end{remark}

\subsection{Brief summary of known results}
All known examples of implosions and consequent explosions arise as exact radially self-similar solutions, constructed by solving \eqref{bw:ss:euler} and then \eqref{fw:ss:eq}. 

We summarize here the main features of the known solutions of the Euler equations that exhibit an implosion followed by an explosion.\footnote{The smooth isentropic implosions from~\cite{MRRS22a,MRRS22b} have not been shown to admit a continuation past the implosion time, but we expect that techniques similar to the ones used here can be used to prove a continuation result for those implosions.} All these solutions are constructed as exact solutions to the ODEs \eqref{bw:ss:euler} and \eqref{fw:ss:eq}; no rigorous nonlinear stability analysis, for either the implosion phase or the explosion phase, is available for any of them.
\begin{itemize}[leftmargin=2em]
\item
\textbf{The Guderley solution.} The classical imploding scenario is given by Guderley's 1942 self-similar solution~\cite{Gud42}, re-discovered by Landau and Stanyukovich~\cite{LS45}, and rigorously constructed by Jang--Liu--Schrecker in~\cite{JLS25b}.\footnote{In~\cite{CSV25} Cialdea-Shkoller-Vicol have proven that Guderley's imploding shock arises dynamically as a solution of the Euler equations with classical (shock-free) initial data; it cannot however arise from initial data that is globally $C^2$-smooth.} It describes an inward-propagating radial/spherical shock in a quiescent medium (constant density, zero velocity and zero temperature). At the time of the implosion, both the velocity and the temperature diverge while the density remains bounded. This solution applies to the full Euler system (not isentropic). The continuation past the time of the blowup has also been the subject of classical studies~\cite{Gud42,LS45}, but has only been rigorously constructed by Jang--Jiu--Schrecker in~\cite{JLS25a}. After the implosion, a shock of finite strength arises and propagates outward. Moreover, at the center of symmetry a point-vacuum singularity arises (that is, the density vanishes at the origin). The size of the jump of the velocity and sound speed decreases as $t \to + \infty$ (this is equivalent to the fact that, in the Guderley implosion, $\crr<1$), while the jump size of the density remains bounded as $t \to +\infty$.

\item
\textbf{Collapsing cavities and the reflection wave~\cite{Hun60,Laz81}.} The collapsing-cavity problem is another classical self-similar scenario for the isentropic Euler equations. Initiated by Hunter~\cite{Hun60} and later treated systematically by Lazarus~\cite{Laz81}, it describes the collapse of an empty spherical cavity in a compressible fluid. At the initial time the density vanishes in a spherical cavity; as $t \to 0^-$ the vacuum region shrinks and both the density and velocity blow up. The continuation beyond collapse produces an outward-propagating reflection wave, where the density is bounded and bounded away from zero everywhere.

\item
\textbf{Continuous implosions followed by shockless explosions~\cite{Jen25}.} In~\cite{Jen25}, Jenssen constructs a continuous imploding solution to the isentropic Euler equations. At the time of the implosion $t=0$, both the velocity and the density blow up. The author also discusses the subsequent explosion, which, unlike the solutions described earlier and the ones in this manuscript, contains no shock. The resulting self-similar flow is shock-free and continuous away from the collapse point; for each finite time after collapse, the density at the center remains finite and strictly positive.

\item
\textbf{Isentropic self-similar shock waves~\cite{SWWZ26}.} In~\cite{SWWZ26}, Shao--Wang--Wei--Zhang construct a continuous solution to the isentropic compressible Euler equations that has a point-vacuum singularity at the origin prior to the singular time, and whose gradient blows up at the center of symmetry. At the time of the implosion, both the velocity and sound speed remain locally bounded. The authors are then able to continue the solution for $t>0$ as a shock of finite strength that propagates outward. For $t>0$ there is no vacuum singularity at the origin, and the size of the shock increases as $t \to +\infty$ (this is equivalent to $\crr>1$)\footnote{
    Using the notation of the authors in~\cite{SWWZ26}, the solutions we just described correspond to the case $\operatorname{Ma} \le \tilde{l}$. The authors construct other solutions to the compressible isentropic Euler equations for $\operatorname{Ma} > \tilde{l}$ that have a shock for $t<0$. While these solutions are weak solutions of the isentropic Euler equations, they violate the Lax geometric inequality and hence are not entropy solutions.
}. 
\item \textbf{Other results.} In~\cite{JT23}, Jenssen and Tsikkou construct an imploding solution to the non-isentropic Euler system and discuss the continuation after the singularity time. Although no proof is given
there, the authors present numerical evidence for the existence of a shock solution past the implosion time.
\end{itemize}
 
\subsection{Main results}
Recently, in~\cite{CSV26}, Chen--Shkoller--Vicol constructed a new class of smooth, non-isentropic, globally self-similar implosions. In the present paper, we continue the ground state of this class of implosions to all times $t>0$. We are not aware of a previous continuation result for implosions of this type (that is, non-isentropic and $C^\infty$ smooth). The continuation we obtain has a feature at the center of symmetry that the previously known solutions do not: for each fixed $t>0$ the density is unbounded at the origin $r=0$, while the pressure stays bounded and the temperature vanishes there. This is to be contrasted with the Guderley continuation~\cite{Gud42,LS45,JLS25a}, in which the reflected wave leaves a point vacuum at the center (that is, the density vanishes at $r=0$). Next, we recall the existence result established in~\cite{CSV26}:

\begin{theorem}[\textbf{Theorem 1.7 in~\cite{CSV26}}] \label{th:csv}
Let $d \in \{1, 2, 3\}$, and $1 < \gamma \le 2 d +1$ be arbitrary. For each integer $\NNN \ge 1$, use the closed form expressions \eqref{cr:formula:g}--\eqref{cb:formula:g} to define the similarity exponents $\crr:= \crrs(d, \gamma, \NNN)$ and $\cbb:= \cbbs(d, \gamma, \NNN)$. Then, there exists a smooth analytic globally self-similar solution $( \bar u,\bar \sigma, \bar b)(r,t)$ to the radially symmetric Euler  equations \eqref{eq:euler:rad}, given in terms of the profiles $(\bar U, \bar \Sigma, \bar B)$ via the ansatz \eqref{bw:ss}, which solve the system~\eqref{bw:ss:euler}. Moreover, there exist constants $\underline v_1\in\RR$, $\underline q_1>0$, and $\underline h_1>0$ such that as $R \to +\infty$ the profiles satisfy the asymptotic power law behavior
\begin{equation} \label{eq:lim:infty}
\lim_{R \to + \infty} R^{\frac{1}{\crr}-1} \bar \Sigma(R) = \underline q_1, \quad \lim_{R \to + \infty} R^{\frac{1}{\crr}-1} \bar U(R) = \underline v_1, \quad \lim_{R \to + \infty} R^{-\frac{\cbb}{\crr}} \bar B(R) = \underline h_1.
\end{equation}
\end{theorem}
For the ground state $\NNN=1$ used below, Proposition A.1 in~\cite{CSV26} gives $\underline v_1<0$.
Using the transformation \eqref{bw:ss}, the profiles $(\bar U, \bar \Sigma, \bar B)$ naturally define solutions of the Euler equations in terms of the variables $(\bar u, \bar \sigma,\bar b)$. For each fixed $t<0$, the density $\bar \rho(\cdot,t)$ is bounded and strictly positive on $\RR^d$, but it becomes unbounded as $t \to 0^-$. Similarly, the velocity vector $\bar \uu$ is smooth on $\RR^d \times [-1,0)$. The pressure $\bar p$ is likewise smooth on $\RR^d \times [-1,0)$, and vanishes at the origin.
From the asymptotic power law behavior of the profiles $(\bar U, \bar \Sigma, \bar B)$,  for any $r>0$ we have the pointwise limits
\begin{subequations} \label{pw:limits}
\begin{align}
&\lim_{t\to 0^-} \bigl(\bar u,\bar \sigma,\bar b\bigr)(r,t) = \bigl( \underline{v}_1 r^{1-\frac{1}{\crr}}, \underline{q}_1 r^{1-\frac{1}{\crr}}, \underline{h}_1 r^{\frac{\cbb}{\crr}}\bigr),
\label{eq:fields:at:time:of:implosion:a} 
\end{align}
which implies, from the definitions \eqref{def:pres}, \eqref{def:b} and \eqref{def:sigma}, that
\begin{align}
&\lim_{t\to 0^-} \bar \rho(r,t) = \underline \rho_1 r^{\frac{\crr - 1 - \cbb}{\alpha \crr}} 
= \underline \rho_1  r^{-\frac{d}{(1+\alpha d) \crr}}
\label{eq:fields:at:time:of:implosion:b}
, \\
&\lim_{t\to 0^-} \bar p(r,t) = \underline p_1   r^{\frac{\gamma (\crr - 1) - \cbb}{\alpha \crr}},
\label{eq:fields:at:time:of:implosion:d}\\
&\lim_{t\to 0^-} \bar E(r,t) = \underline e_1 r^{\frac{\gamma (\crr - 1) - \cbb}{\alpha \crr}} 
\label{eq:fields:at:time:of:implosion:f},
\end{align}
\end{subequations}
where
\begin{equation*}
\underline{\rho}_1 = \left( \tfrac{\alpha \underline q_1}{\underline h_1} \right)^{\frac{1}{\alpha}},
 \quad \underline p_1= \tfrac{1}{\gamma} (\alpha \underline q_1)^{\frac{\gamma}{\alpha}} \underline h_1^{-\frac{1}{\alpha}}, 
 \quad \underline e_1= \tfrac{1}{2} (\tfrac{\alpha}{\gamma}\underline q_1^2 +\underline v_1^2 )(\alpha \underline q_1)^{\frac{1}{\alpha}} \underline h_1^{-\frac{1}{\alpha}}.
\end{equation*}

In particular, at the time of the implosion, the density becomes unbounded. Whether the pressure or the sound speed blows up depends on the particular values of $\crr$ and $\cbb$. We refer the interested reader to the discussion in~\cite{CSV26}.
Nevertheless, as observed in~\cite{CSV26}, at the time of implosion \textit{the mass, energy, and momentum are locally finite}.\footnote{For globally self-similar profiles it is not reasonable to expect  finite \emph{global} mass, momentum, or energy, due to the natural tail behavior at infinity in self-similar coordinates.} For the mass, this follows from the inequality $\crr > \frac{1}{1+\alpha d}$, which follows directly from~\eqref{cr:formula:g}, and implies
\begin{subequations} \label{loc:fin}

\begin{equation} \label{loc:fin:m}
    \lim_{t \to 0^-} \int_0^1 \bar \rho(r,t) r^{d-1} dr < +\infty.
\end{equation}
 Similarly, using $d > \tfrac{\cbb - \gamma (\crr-1)}{\alpha \crr} \iff \crr > \tfrac{d+2 + 2 \alpha d}{(d+2)(1+\alpha d)}$ (see Lemma 2.8 in~\cite{CSV26}), we have 
\begin{equation} \label{loc:fin:e}
\lim_{t \to 0^-}\int_0^1 \bar E(r,t) r^{d-1} dr < + \infty. 
\end{equation} 
 Lastly, since $d > \tfrac{d}{(1+\alpha d)\crr}+\tfrac{1}{\crr}-1 \iff \crr> \tfrac{d+1+\alpha d}{(d+1)(1+\alpha d)}$ (see Lemma 2.8 in~\cite{CSV26}), we deduce that
 \begin{equation} \label{loc:fin:mom}
  \lim_{t \to 0^-} \int_0^1 |\bar \rho \bar u^r| r^{d-1} dr<+\infty.
 \end{equation}
     
\end{subequations}
 The goal of this paper is to construct the continuation of these solutions past the time of the implosion singularity.
The pointwise limits \eqref{pw:limits} and the local integrability conditions \eqref{loc:fin} are the starting point of the analysis for the explosion phase $t>0$. We will construct globally forward self-similar solutions $(u, \sigma, b)$ to the Euler equations, via the ansatz \eqref{forward:profiles}, by solving the system \eqref{fw:ss:eq}.

We state our main result. For simplicity we will consider the case $d=3$, $\NNN=1$, and $\gamma \in (1, 3+ 2 \sqrt{2}]$. We expect that similar arguments apply to the remaining parameter regimes.

\begin{theorem}[\textbf{Main Theorem}] \label{th:main}
Fix $d=3$, $1<\gamma \le 3+ 2\sqrt{2}$, and $\NNN=1$. Define $\crr=\crrs(3, \gamma, 1)$ and $\cbb=\cbbs(3, \gamma, 1)$ using the expressions \eqref{cr:formula} and \eqref{cb:formula}. Fix the constants $\underline v_1$, $\underline q_1$, and $\underline h_1$ from~\eqref{eq:lim:infty}. Then, there exists a unique globally forward self-similar shock solution $(u, \sigma , b)(r,t)$ to the radially symmetric Euler equations \eqref{eq:euler:rad}, given in terms of unique piecewise smooth profiles $(U, \Sigma, B)$ such that
\begin{enumerate}[leftmargin=2em]
\item Piecewise smoothness: there exists $\xiRH$ such that $(U,\Sigma, B)$ are $C^\infty$ smooth on $(0, \xiRH)$ and $(\xiRH, +\infty)$,
\item Rankine--Hugoniot jump conditions: across the jump at $\xi=\xiRH$,  $(U, \Sigma, B)$ satisfy the jump conditions \eqref{RH:jump} and the Lax entropy inequality \eqref{lax:ineq},
\item Regularity at $0$: there exist constants $\overline v_0$ and $\overline q_0,\overline h_0>0$ such that as $\xi\to0^+$ the solution has the following local behavior
\begin{equation} \label{asy:origin}
U(\xi) =\overline v_0 \xi + O(\xi^{\frac{9\gamma}{2+3\gamma}}), \quad \Sigma(\xi)=\overline q_0 \xi^{\frac{3}{2+3\gamma}} + O(\xi^{\frac{1+6\gamma}{2+3\gamma}}), \quad B(\xi) =\overline h_0 \xi^{\frac{3\gamma}{2+3\gamma}}  + O(\xi^{\frac{9\gamma-2}{2+3\gamma}}),
\end{equation}
\item Positivity: for $\xi>0$ we have $\Sigma(\xi)>0$ and $B(\xi)>0$,
\item Matching at $\xi=+\infty$: the asymptotic behavior of $(U,\Sigma, B)$ matches the asymptotic behavior \eqref{eq:lim:infty} coming from the implosion
\begin{equation} \label{lim:infty:ex}
\lim_{\xi \to +\infty} \xi^{\frac{1}{\crr}-1}U(\xi) = \underline v_1, \quad \lim_{\xi \to +\infty}  \xi^{\frac{1}{\crr}-1} \Sigma (\xi) =\underline q_1, \quad \lim_{\xi \to +\infty} \xi^{-\frac{\cbb}{\crr}} B(\xi) = \underline h_1.
\end{equation}
The matching condition at $\xi=+\infty$ will allow us to prove, in Corollary~\ref{cor:main}, that these solutions are indeed the continuation of the imploding solutions constructed in Theorem~\ref{th:csv}.
\end{enumerate}

\end{theorem}
\begin{figure}[htbp]
\centering
\begin{subfigure}{0.9\textwidth}
\centering
\includegraphics[width=0.9\textwidth]{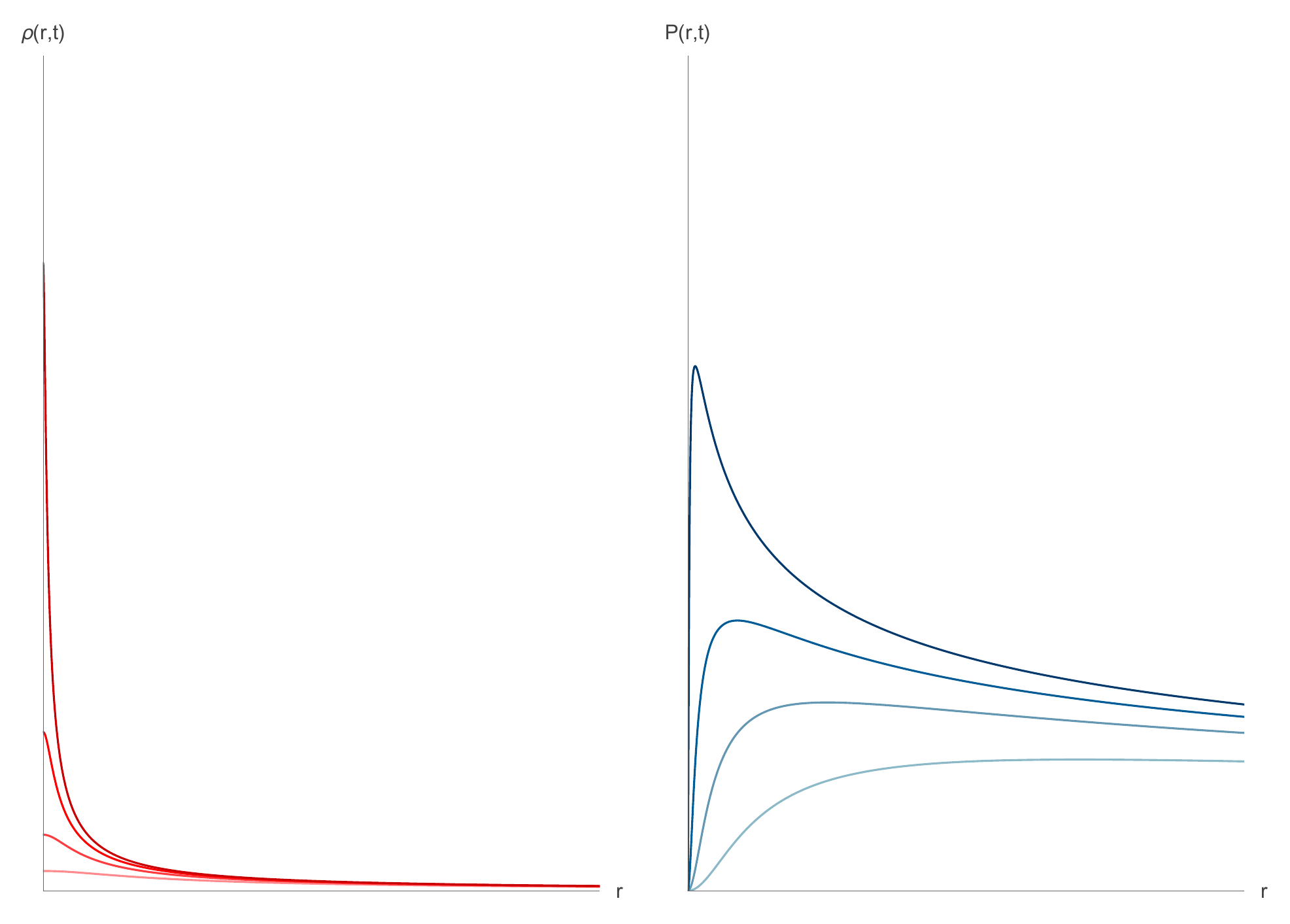}
\caption{Implosion.}
\end{subfigure}
\medskip
\begin{subfigure}{0.9\textwidth}
\centering
\includegraphics[width=0.9\textwidth]{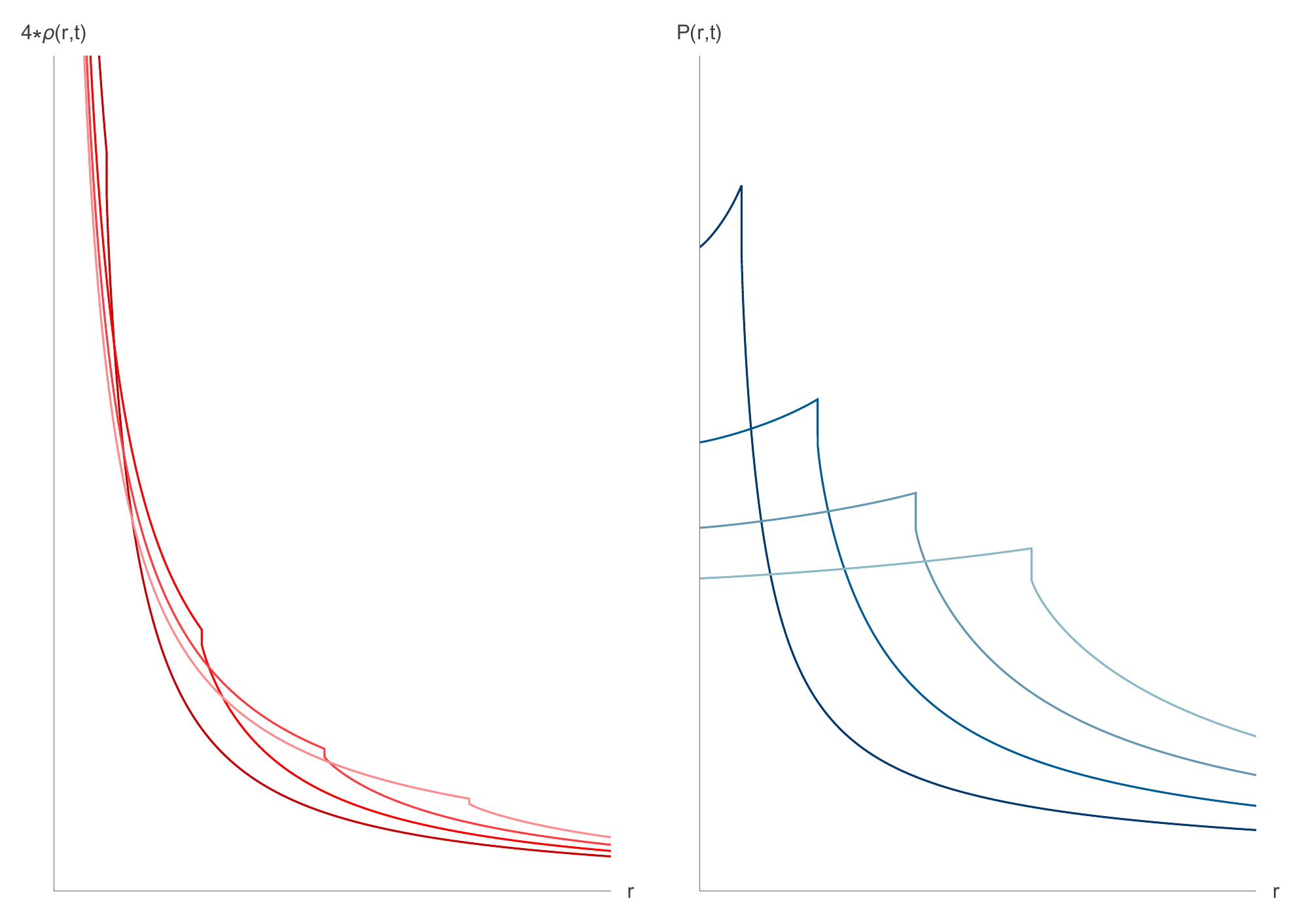}
\caption{Explosion.}
\end{subfigure}
\caption{\textbf{Implosion and explosion snapshots for $d=3$, $\gamma=\frac{5}{3}$, and $\NNN=1$.} \emph{Top}: radial density $\rho(r,t)$ (red, left) and pressure $p(r,t)$ (blue, right) for the ground-state implosion profile of~\cite{CSV26}, shown at four times, with darker shades indicating later times, closer to $t=0$. \emph{Bottom}: corresponding snapshots for the forward self-similar explosion constructed here (with darker shades indicating earlier times, closer to $t=0$); the shock propagates outward, the density remains singular at $r=0$, and the pressure stays bounded near the origin.}
\label{fig:implosion:explosion:snapshots}
\end{figure}
\begin{remark}[Behavior of the profiles at $\xi=0$ in the monatomic and diatomic cases]
The asymptotics \eqref{asy:origin} give explicit fractional power laws at the origin.
In the monatomic case $\gamma=\frac{5}{3}$, they reduce to
\begin{equation*}
U(\xi) \sim \xi, \qquad \Sigma(\xi) \sim \xi^{\frac{3}{7}}, \qquad B(\xi) \sim \xi^{\frac{5}{7}}.
\end{equation*}
In the diatomic case $\gamma=\frac{7}{5}$, they reduce to
\begin{equation*}
U(\xi) \sim \xi, \qquad \Sigma(\xi) \sim \xi^{\frac{15}{31}}, \qquad B(\xi) \sim \xi^{\frac{21}{31}}.
\end{equation*}
\end{remark}
\subsubsection{Description of the explosion in the original variables}
The profiles $(U, \Sigma, B)$, through the transformation \eqref{forward:profiles}, define natural solutions to the Euler equations, with a shock at the location 
\begin{equation*}
\s(t) =\xiRH t^{\crr}.
\end{equation*} 
The original flow variables are smooth on either side of the shock $\s(t)$ for $t>0$ and $r>0$. At the origin $r=0$ the velocity field $\uu$ is $C^1$, the pressure $p$ is bounded, while the density $\rho$ is unbounded (but locally $L^1$). If we introduce the temperature variable $\theta =\frac{p}{\rho}$, then \eqref{asy:origin}, together with $\cbb=\crr-\frac{2}{3\gamma-1}$ (see~\eqref{eq:cb:value}), gives, for each fixed $t>0$, as $r\to0^+$, the asymptotic behavior
\begin{align*}
u(r,t) \sim t^{-1} r, \quad
\rho(r,t) \sim  t^{\frac{6\crr}{2+3\gamma}-\frac{6}{3\gamma-1}} r^{-\frac{6}{2+3\gamma}},\quad
p(r,t) \sim t^{2\crr-2-\frac{6}{3\gamma-1}},\quad
\theta(r,t) \sim t^{\frac{2(3\gamma-1)\crr}{2+3\gamma}-2} r^{\frac{6}{2+3\gamma}}.
\end{align*}
Thus the pressure remains bounded and the temperature vanishes at $r=0$.
As $t \to +\infty$, the velocity jump  $\jump{u}$ and temperature jump $\jump{\theta}$ across the shock surface grow (in time), while the density jump $\jump{\rho}$ decays. Whether the pressure jump $\jump{p}$ grows or not depends on the particular value of $\crr$:
\begin{equation*}
\jump{u}(t) \sim t^{\crr-1}, \quad \jump{\rho}(t) \sim t^{-\frac{6}{3\gamma-1}}, \quad \jump{p}(t) \sim t^{2\crr-2-\frac{6}{3\gamma-1}}, \quad \jump{\theta}(t) \sim t^{2\crr-2}.
\end{equation*}

\subsubsection{The explosion is the continuation of the implosion}
Using the transformation \eqref{forward:profiles}, the profiles $(U, \Sigma, B)$ naturally define globally self-similar shock solutions for the Euler equations. Moreover, the matching condition \eqref{lim:infty:ex} implies that, for each fixed $r>0$,
\begin{equation*}
\lim_{t\to 0^+} \bigl( u, \rho, \theta \bigr)(r,t) = \left( \underline{v}_1 r^{1-\frac{1}{\crr}}, \underline{\rho}_1 r^{-\frac{6}{(3\gamma-1)\crr}}, \tfrac{(\gamma-1)^2\underline q_1^2}{4\gamma} r^{2-\frac{2}{\crr}}\right),
\end{equation*}
which we use to show that such explosions arise naturally as a continuation of the implosion solutions of Theorem~\ref{th:csv}. We have the following corollary.
\begin{corollary} \label{cor:main}
    Fix $d=3$, $1<\gamma \le 3+2 \sqrt{2}$. There exist smooth initial radial data $(u_\mathrm{in}, \sigma_\mathrm{in}, b_\mathrm{in} )$ at the initial time $t=-1$ that generate a weak global Euler solution, in the sense of Definition~\ref{def:radial:weak:solution}, on the time interval $[-1, +\infty)$. The evolution unfolds in the following stages:
    \begin{enumerate}[leftmargin=2em]
\item Smooth evolution: for $t \in [-1, 0)$ the solution $(u, \sigma, b)$ evolves smoothly, 
\item Implosion: at $t=0$ the evolution culminates in an implosion singularity at the origin $r=0$, where the density $\rho \to +\infty$,
\item Reflected blast wave: for $t>0$ the solution reflects as a strong, outgoing shock wave.
\end{enumerate}

\end{corollary}

\subsubsection{Examples of smooth explosions}
As a byproduct of our analysis, we obtain the existence of smooth, globally forward self-similar solutions to the Euler equations. These solutions are time reversals of the smooth implosions from Theorem~\ref{th:csv}, but they do not satisfy the matching condition at $t=0$ required to continue those implosions.
\begin{theorem}[\bf Smooth explosions] \label{thm:2}
Fix $d\in \{1,2,3\}$, $1<\gamma \le 2d+1$. For each $\NNN \ge 1$, use the explicit expressions \eqref{cr:formula:g} and \eqref{cb:formula:g} to define $\crr=\crrs(d, \gamma, \NNN)$ and $\cbb=\cbbs(d, \gamma, \NNN)$. Then, there exists an exact globally forward self-similar solution $(u, \sigma, b)(r,t)$ to the radially symmetric Euler equations \eqref{eq:euler:rad}, given in terms of smooth profiles $(U, \Sigma, B)$. The profiles $(U, \Sigma, B)$ solve the system \eqref{fw:ss:eq}. \end{theorem}
These exploding solutions are \textit{exactly} the time reversal of the implosions from Theorem~\ref{th:csv}.
Using the transformation \eqref{forward:profiles}, we obtain smooth solutions $u, \sigma, b$ to the Euler equations for all $t>0$. Both $\sigma(r,t)$ and $b(r,t)$ vanish at $r=0$, while the density is strictly positive for all $t>0$.
\subsection{Comparison with previous results}

The continuation constructed in Theorem~\ref{th:main} is close in spirit to the classical Guderley continuation~\cite{Gud42,LS45,JLS25a}. In the explosion phase the flow is selected by a self-similar phase portrait and by the Rankine--Hugoniot matching condition. The structure at the center $r=0$, however, is different. In the Guderley continuation, the reflected wave leaves a point vacuum at the origin. On the other hand, the solutions of Theorem~\ref{th:main} have, for any fixed $t>0$, unbounded density at $r=0$, while the pressure $p$ remains bounded and the temperature $\theta$ vanishes there.

At the level of the proof, the construction follows the same general mechanism as the classical self-similar shock constructions. One constructs a branch determined by the matching condition at $\xi=+\infty$ (see Proposition~\ref{prop:sub:sonic}), controls its continuation until it reaches the sonic line, and constructs a second branch at $\xi=0$ (see Proposition~\ref{prop:infty}). These two branches lie on opposite sides of the sonic line, one in the subsonic region and one in the supersonic region. The shock location $\xiRH$ is then selected by the Rankine--Hugoniot jump conditions together with the Lax entropy inequality.

\subsection{Strategy of the proof}
Proving Theorem~\ref{th:main} amounts to finding a solution to the system of ODEs \eqref{fw:ss:eq} such that the boundary conditions \eqref{lim:infty:ex} hold. 
It will be convenient to introduce the renormalized profiles
\begin{equation*}
 U(\xi) :=\xi  V (\xi), \quad  \Sigma(\xi) :=\xi  Q(\xi), \quad  B(\xi):=\xi  H(\xi).t
\end{equation*}
In these variables, the system \eqref{fw:ss:eq} becomes \eqref{eq:main}. From \eqref{eq:H:main}, we can isolate
\begin{equation*}
\tfrac{\xi \pp_\xi H}{H}= -\tfrac{V-\crr+\cbb}{V-\crr}.
\end{equation*}
By substituting this expression into \eqref{eq:main}, the system of three ODEs then reduces to a $2\times 2$ closed system for $(V,Q)$, which can be solved for $\xi\partial_\xi V$ and $\xi\partial_\xi Q$ to obtain the autonomous system \eqref{eq:euler:ODE}, where $\xi \pp_\xi$ acts as the time variable. After solving the system \eqref{eq:euler:ODE},  $H$ is recovered through $V$ by a direct integration and by imposing the boundary condition \eqref{lim:infty:ex} at $\xi=+\infty$. In the renormalized variables, the conditions \eqref{lim:infty:ex} become \begin{equation} \label{boundary:infty}
\lim_{\xi \to +\infty} \xi^{\frac{1}{\crr}}V(\xi) = \underline v_1, \quad \lim_{\xi \to +\infty} \xi^{\frac{1}{\crr}}Q(\xi) = \underline q_1, \quad \lim_{\xi \to +\infty} \xi^{1-\frac{\cbb}{\crr}}H(\xi) = \underline h_1.
\end{equation}
The first task is then local, to construct a solution $(V,Q)$ to \eqref{eq:euler:ODE} near $\xi=+\infty$ that satisfies the boundary conditions \eqref{boundary:infty}. 
We then proceed to prove that such a solution can be continued up to a time $\xi=\xi_s$, and that at such a time $(V,Q)$ hits the sonic line $V+\alpha Q=\crr$; crucially, \emph{this intersection lies to the left of the sonic point} $P_2=(V_2, Q_2)$ (see \eqref{def:V2}). To prove this we first need to establish sharper bounds on the local Mach number at infinity $\frac{\underline v_1}{\underline q_1}$. In~\cite{CSV26} it was shown (see Proposition A.1 there) that $\frac{\underline v_1}{\underline q_1}<0$, but for the purpose of showing that the solution $(V,Q)$ misses the sonic point this bound turns out to be insufficient. In Proposition~\ref{prop:ub:imp}, we establish the new improved bound $\frac{\underline v_1}{\underline q_1} < -\tfrac{3}{5}\sqrt{\tfrac{2\alpha}{3\gamma}}$.

To show this bound we construct, \textit{in the implosion phase portrait}, a polynomial upper barrier $(\mathcal{W}(q),q)$ for the imploding trajectory $(\bar V,\bar Q)$. Establishing that $(\mathcal{W}(q),q)$ is an upper barrier reduces to checking the sign of the scalar product of the tangent to the curve against the vector field generating the ODE \eqref{2x2:imp}. After some algebra, we are left with checking the sign of a polynomial $\PP(t)$ (defined in \eqref{p:def:I}) of degree $6$ in $t$ on $[0,1]$, where the coefficients are explicit functions of $\alpha$. By writing the polynomial $\PP$ in a Bernstein basis of degree $8$, we reduce the problem to checking the sign of $10$ explicit coefficients $\beta_j(\alpha)$ (see Lemma~\ref{lemma:Bernstein}). 

Once the bound $\frac{\underline v_1}{\underline q_1}<-\tfrac{3}{5}\sqrt{\tfrac{2\alpha}{3\gamma}}$ is established, we use this information to construct a polynomial lower barrier \textit{in the explosion phase portrait} $(\mathcal{W}_E(q), q)$, in order to show that the solution $(V,Q)$ does not hit the sonic point $P_2$. As before, proving that $(\mathcal{W}_E(q), q)$ is a lower barrier reduces to checking the sign of a polynomial $\QQ(t)$ of degree $5$ in $t$. Again, we rewrite this polynomial in a Bernstein basis of degree $5$ with coefficients $\BB_j(\alpha)$ (see Lemma~\ref{lemma:Bernstein:q}). For each fixed rational value of $\alpha$, the sign of the coefficients $\BB_j(\alpha)$ can be verified elementarily, since each of them reduces to determining the signs of six constant algebraic expressions. Since extending this argument uniformly to the entire interval $\alpha \in (0,1+\sqrt{2}]$ would not provide any additional insight, we instead give two computer-assisted proofs; see Appendix~\ref{app:rational-cover-bernstein-q} and Appendix~\ref{app:CA} for the details.

Once we know that the solution $(V,Q)$ does not hit the sonic point $P_2$, we construct a trajectory $(V^\infty,\allowbreak Q^\infty)$ connecting the point $P_2$ to the point $P_6$ at $Q=+\infty$ in Proposition~\ref{prop:infty}. We then consider the \textit{Hugoniot locus} of the trajectory $(V,Q)$, defined as $\mathrm{RH}(V(\xi), Q(\xi))$ (see \eqref{RH:jump}). An application of the intermediate value theorem first gives an intersection between the Hugoniot image of the outer branch and the inner trajectory $(V^\infty,Q^\infty)$; rescaling the inner trajectory in $\xi$, we arrange that this intersection occurs at the same value $\xi=\xiRH$. The proof is then concluded after redefining
\begin{equation*}
(V,Q)(\xi)= \begin{cases} 
     (V,Q)(\xi) \quad & \xiRH < \xi < +\infty, \\
(V^\infty,Q^\infty)(\xi) \quad & 0 < \xi <\xiRH, 
\end{cases}
\end{equation*}
 and observing that such a solution naturally induces a shock solution for the Euler equations \eqref{eq:euler:cl} via the transformation \eqref{forward:profiles}.


\section{A new bound for the imploding self-similar profiles} 
\label{nb:imp}
The goal of this section is to show upper bounds for $\frac{\underline v_1}{\underline q_1}$ for the imploding solutions $(\bar U, \bar \Sigma, \bar B)$ of Theorem~\ref{th:csv}, where $\underline v_1$ and $\underline q_1$ are defined in \eqref{eq:lim:infty}. This section is devoted to constructing a suitable barrier for the system of ODEs \eqref{bw:ss:euler} that will allow us to prove the following bound.
\begin{proposition} \label{prop:ub:imp}
Let $d=3$, $\alpha \in (0,1+\sqrt{2}]$, $\gamma=2\alpha+1$, and $\NNN = 1$. Let $(\bar U, \bar \Sigma, \bar B)$ be the global-in-$R$ imploding solution from Theorem~\ref{th:csv}, and let $\underline v_1$, $\underline q_1$ be the leading coefficients as $R \to +\infty$ in \eqref{eq:lim:infty}. Then,
\begin{equation} \label{upp:bound:imp}
- \sqrt{\tfrac{2 \alpha}{3 \gamma}} \le \tfrac{\underline v_1}{\underline q_1} < \rr=\rr(\alpha)  := - \tfrac{3}{5} \sqrt{\tfrac{2\alpha}{3\gamma}} <0.
\end{equation}
\end{proposition}
\begin{remark}[\textbf{ $\underline v_1 <0$ is not sufficient}]
    In Proposition A.1 in~\cite{CSV26}, the authors proved $\underline v_1 <0$. This bound is not sufficient to describe the maximal development of the exploding trajectory $(V,Q)$ in Section~\ref{exp:pp}. In particular, in order to prove Proposition~\ref{prop:sub:sonic}, it is essential to use the stronger upper bound from \eqref{upp:bound:imp}.
\end{remark}
\subsection[A closed system for (bar V, bar Q)]{A closed system for \texorpdfstring{$(\bar V,\bar Q)$}{(bar V, bar Q)}}
We start by recalling some details of the construction carried out in~\cite{CSV26} to prove Theorem~\ref{th:csv}. We introduce the constants
\begin{equation}
V_0 = \tfrac{1}{1+3 \alpha }, \quad Q_0 = \tfrac{1}{1+3 \alpha }\sqrt{\tfrac{3\gamma}{2\alpha}},
\end{equation}
and fix
\begin{equation} \label{eq:cb:value}
\cbb= \crr-\tfrac{1}{1+3\alpha} = \crr-V_0.
\end{equation}
Moreover, we introduce the following renormalization for the profiles $(\bar U, \bar \Sigma, \bar B)$
\begin{equation} \label{renorm:imp}
\bar U(R) =: R \bar V(R), \quad \bar \Sigma(R) =: R \bar Q(R), \quad \bar B(R)=: R \bar H(R).
\end{equation}
For $d=3$, the system \eqref{bw:ss:euler} transforms into
\begin{subequations} \label{eq:ss:eul:v}
\begin{align}
(1+(1+3 \alpha)\bar V)\bar Q+(\crr+\bar V) R \pp_R \bar Q + \alpha \bar Q R \pp_R \bar V&=0, \\
(\bar V + \bar V^2 + \tfrac{2\alpha^2}{\gamma}\bar Q^2)+ (\crr + \bar V)R \pp_R \bar V+ \alpha \bar Q R \pp_R \bar Q - \tfrac{\alpha}{\gamma} \bar Q^2\tfrac{R \pp_R \bar H}{\bar H}&=0, \label{eq:Q:i} \\
(\crr-\cbb+ \bar V)\bar H+ (\crr+ \bar V)R\pp_R \bar H&=0. \label{eq:H:i}
\end{align} 
\end{subequations}
If $\bar V \neq -\crr$ and $\bar H \neq 0$, we may use \eqref{eq:H:i} to express  $\frac{R\pp_R \bar H}{\bar H} = - \frac{\crr-\cbb+\bar V}{\crr+\bar V}$, plug it into \eqref{eq:Q:i} and reduce the system \eqref{eq:ss:eul:v} to a closed system for the unknowns $\bar V$ and $\bar Q$. On the set where $\bar \Delta[\bar V,\bar Q]\neq0$, solving the resulting $2\times2$ linear system for $R\partial_R\bar V$ and $R\partial_R\bar Q$, we obtain the autonomous ODE
\begin{subequations} \label{2x2:imp} \begin{align}
R \pp_R \bar V = \frac{ \Delta_{\bar V}[\bar V, \bar Q]}{\bar \Delta[\bar V, \bar Q]},  \\
R \pp_R \bar Q = \frac{ \Delta_{\bar Q}[\bar V, \bar Q]}{\bar \Delta[\bar V, \bar Q]},
  \end{align}
 \end{subequations}
 where\footnote{In~\cite{CSV26}, the authors use $\Delta$ to denote the denominator of~\eqref{2x2:imp}. Here we use $\bar \Delta$ for the same object, since we reserved the notation $\Delta$ for the explosion phase, see \S~\ref{P:exp}.}
\begin{subequations} \label{deltas:imp}\begin{align}
 \Delta_{\bar V}[\bar V, \bar Q]&:= \tfrac{\alpha^2 (2+3 \gamma)}{\gamma} (\crr +\bar V) \bar Q^2 (\bar V+V_0) - (\crr+\bar V)^2 (\bar V+\bar V^2 + \tfrac{2 \alpha^2}{\gamma} \bar Q^2),\\
  \Delta_{\bar Q}[\bar V, \bar Q] &:= \alpha \bar Q  (\crr+\bar V)(\bar V+\bar V^2 + \tfrac{2 \alpha^2}{\gamma} \bar Q^2) + \tfrac{\alpha^2}{\gamma} \bar Q^3 ( \bar V + V_0) - (1+3 \alpha )(\crr+ \bar V)^2 \bar Q (\bar V+V_0), \\
  \bar \Delta[\bar V, \bar Q] &:= (\crr+\bar V)(\crr+ \bar V - \alpha \bar Q) ( \crr + \bar V + \alpha \bar Q).
 \end{align}\end{subequations}

\subsection[Properties of (bar V, bar Q)]{Properties of \texorpdfstring{$(\bar V, \bar Q)$}{(bar V, bar Q)}}

For an integer $n \ge 1$, introduce the expression $\mathsf{E}_n:=\mathsf{E}_n(\gamma, d)$
\begin{equation*}
\mathsf{E}_n:= \tfrac{\alpha \gamma d (d+2)}{4 n} + \tfrac{\alpha d (1+\alpha d)}{2 n^2},
\end{equation*}
and define the similarity exponents
\begin{equation} \label{cr:formula:g}
\crrs(d, \gamma, \NNN):= \tfrac{1}{1+\alpha d} \left(1 + \sqrt{\tfrac{\alpha d \gamma}{2} + \mathsf{E}_\NNN + \tfrac{(1-\alpha d)^2}{16 \NNN^2}} + \tfrac{1-\alpha d}{4 \NNN} \right),
\end{equation}
and
\begin{equation} \label{cb:formula:g}
\cbbs(d, \gamma, \NNN) := \crrs(d, \gamma, \NNN)- \tfrac{1}{1+\alpha d} = \tfrac{1}{1+\alpha d} \left( \sqrt{\tfrac{\alpha d \gamma}{2} + \mathsf{E}_\NNN + \tfrac{(1-\alpha d)^2}{16 \NNN^2}} + \tfrac{1-\alpha d}{4 \NNN} \right).
\end{equation}
For most of the paper, we will focus on the case $d=3$ and $\NNN=1$, in which case the expressions above reduce to
\begin{equation}\label{En:exp}
\mathsf{E}_1= \tfrac{15\alpha \gamma}{4 } + \tfrac{3 \alpha (1+3 \alpha)}{2 },
\end{equation}
\begin{equation} \label{cr:formula}
\crrs(3, 1+2\alpha, 1):= \tfrac{1}{1+3 \alpha} \left(1 + \sqrt{\tfrac{3 \alpha \gamma}{2} + \mathsf{E}_1 + \tfrac{(1-3 \alpha)^2}{16 }} + \tfrac{1-3 \alpha}{4 } \right),
\end{equation}
and
\begin{equation} \label{cb:formula}
\cbbs(3, 1+2\alpha, 1) := \crrs(3, 1+2\alpha, 1)- \tfrac{1}{1+3 \alpha} = \tfrac{1}{1+3 \alpha} \left( \sqrt{\tfrac{3 \alpha \gamma}{2} + \mathsf{E}_1 + \tfrac{(1-3 \alpha)^2}{16 }} + \tfrac{1-3 \alpha}{4 } \right).
\end{equation}
Set $d=3$, $\alpha \in (0,1+\sqrt{2}]$, $\crr=\crrs(3, \gamma, 1)$, and $\cbb=\cbbs(3, \gamma, 1)$. Then, $(\bar V(R), \bar Q(R))$ are global, smooth solutions (for $R \in(0, +\infty)$) to the system \eqref{2x2:imp}. As $R \to 0^+$, we have
\begin{equation*}
\lim_{R \to 0^+} \bar V(R) = -V_0, \quad \lim_{R \to 0^+} \bar Q(R) = Q_0,
\end{equation*}
while, from \eqref{eq:lim:infty}, as $R \to +\infty$ we have
\begin{equation} \label{eq:infty:bar:V}
\lim_{R \to \infty} R^{\frac{1}{\crr}} \bar V = \underline v_1, \quad \lim_{R \to \infty} R^{\frac{1}{\crr}} \bar Q = \underline q_1.
\end{equation}
Moreover, it was shown in Corollary 2.16 in~\cite{CSV26} that there exists a map $\mathcal{V}: (0, Q_0) \to \mathbb{R}$ such that $\mathcal{V}(q) = \bar V ( \bar Q^{-1}(q))$\footnote{This follows from the fact that $\bar Q(R)$ is a strictly decreasing function in $R$.}. By taking limits as $R \to 0^+$ and $R \to +\infty$, $\mathcal{V}$ can be extended continuously to all $[0,Q_0]$, with $\mathcal{V}(0)=0$ and $\mathcal{V}(Q_0)=-V_0$. For $q \in (0, Q_0)$, the chain rule and \eqref{2x2:imp} give
\begin{equation} \label{eq:mc:V:i}
\frac{d \mathcal{V}}{d q} = \tfrac{\Delta_{\bar V}}{\Delta_{\bar Q}} [ \mathcal{V}(q), q],
\end{equation}
and as a consequence of \eqref{eq:infty:bar:V}, we have
\begin{equation*}
\lim_{q \to 0^+} \tfrac{ \mathcal{V}(q)}{ q} = \partial_q^+\mathcal{V}(0) = \tfrac{\underline v_1}{\underline q_1}.
\end{equation*}

\subsection{An explicit upper barrier}
We seek a $C^1$ function $\mathcal{W} \colon [0,Q_0] \to [-V_0, 0]$ such that $\mathcal{W}$ serves as an upper barrier for the function $\mathcal{V}$. It is convenient to introduce the variable $t:=\tfrac{q}{Q_0}$, and a function $G \colon [0,1] \to [-V_0, 0]$ such that
\begin{equation*}
\mathcal{W}(q) := G(t).
\end{equation*}
To guarantee that $\mathcal{W}$ is an upper barrier for $\mathcal{V}$ we must ensure the following
\begin{enumerate}[leftmargin=2em]
\item \label{cond:i:bar} Endpoint conditions: $G(0)=0$ and $G(1)\ge -V_0$. This ensures $\mathcal{W}(0)=0$ and $\mathcal{W}(Q_0)\ge\mathcal{V}(Q_0)=-V_0$.
\item  \label{cond:ii:bar} Derivative upper bound:  $ G'(0) \le -\tfrac{3}{5+15\alpha}$. This condition is equivalent to  $\mathcal{W}'(0) \le \rr$. 
\item \label{cond:iii:bar}  Barrier condition: we must ensure that for all $t \in (0, 1)$ we have  
\begin{equation} \label{eq:upper:barrier:cond}
\tfrac{\Delta_{\bar V}}{\Delta_{\bar Q}}[G(t), Q_0 t] \ge \tfrac{G'(t)}{Q_0}.
\end{equation}
Equation \eqref{eq:upper:barrier:cond} is a restatement of the inequality $\tfrac{\Delta_{\bar V}}{\Delta_{\bar Q}}[\mathcal{W}(q), q] \ge \mathcal{W}'(q)$.
\end{enumerate}
We now choose $G$ to be an explicit polynomial of degree $2$ in $t$.
\begin{proposition} \label{prop:barrier:imp} Let $d=3$, $\alpha \in (0,1+\sqrt{2}]$ and $\NNN=1$. Then, the function
\begin{equation} \label{def:G:imp}
G(t) := - \tfrac{3}{10} \tfrac{1}{1+3 \alpha} t^2 -\tfrac{3}{5+15\alpha}t = - \tfrac{3}{10}V_0 t^2 - \tfrac{3}{5} V_0 t
\end{equation}
satisfies the conditions \eqref{cond:i:bar}--\eqref{cond:iii:bar}.
\end{proposition}

We observe, from the definition of $G$ in \eqref{def:G:imp}, that $(G(t), Q_0 t) \in (-V_0, 0) \times (0, Q_0)$ for every $t \in (0,1)$. In particular, by Corollary 2.16 in~\cite{CSV26} we have
\begin{equation*}
\Delta_{\bar Q}[G(t), Q_0 t] < 0 \quad \forall t \in (0, 1).
\end{equation*}
This implies that the inequality \eqref{eq:upper:barrier:cond} can be rearranged as 
\begin{equation} \label{ineq:upper:im}
\mathcal{B}^I(t):=\Delta_{\bar V}[G(t), Q_0 t] - \tfrac{G'(t)}{Q_0}  \Delta_{\bar Q}[G(t), Q_0 t] \le 0 \quad \forall t \in [0, 1].
\end{equation}
A direct computation shows that $\mathcal{B}^I(t)$ is a polynomial of degree $8$ in $t$; but we observe that $\mathcal{B}^I(t)$ is divisible by $t^2$. 
\begin{lemma} \label{pol:division}
Let $\mathcal{B}^I(t)$ be the polynomial defined in \eqref{ineq:upper:im}. Then,  $\mathcal{B}^I$ is divisible by $t^2$. In particular, 
\begin{equation} \label{p:def:I}
\PP(t):=  t^{-2} \mathcal{B}^I(t),
\end{equation}
is a polynomial of degree at most $6$ in $t$.
\end{lemma} 
\begin{proof}[Proof of Lemma~\ref{pol:division}]
Since $\mathcal{B}^I(t)$ is a polynomial, it is sufficient to prove that $t=0$ is a root of multiplicity at least two, namely
\begin{equation} \label{const:div}
\mathcal{B}^I(0) = 0, \quad \frac{d}{dt}\mathcal{B}^I (0) =0.
\end{equation}
By a direct computation we obtain
\begin{equation*}
\mathcal{B}^I(0) = \Delta_{\bar V}[0,0] - \tfrac{G'(0)}{Q_0} \Delta_{\bar Q}[0,0] =0.
\end{equation*}
Another direct computation gives 
\begin{equation*}
( \partial_V \Delta_{\bar V}[V,Q], \partial_Q \Delta_{\bar V}[V,Q] )[0,0]= (-\crr^2, 0), \quad ( \partial_V \Delta_{\bar Q}[V,Q], \partial_Q \Delta_{\bar Q}[V,Q] )[0,0]= (0, -\crr^2),
\end{equation*}
which yields
\begin{align*}
\tfrac{d}{dt} \mathcal{B}^I(0) &=  \left.\tfrac{d}{dt}\Delta_{\bar V}[G(t), Q_0 t]\right|_{t=0}  -  \left.\tfrac{d}{dt}\Delta_{\bar Q}[G(t), Q_0 t]\right|_{t=0} \tfrac{G'(0)}{Q_0} - \Delta_{\bar Q}[0,0]  \tfrac{G''(0)}{Q_0}   \\
&= -  G'(0) \crr^2 + G'(0) \crr^2 =0.
\end{align*}
\end{proof}
\subsection{Reduction to verifying the sign of a polynomial}
By Lemma~\ref{pol:division}, it is enough to prove
\begin{equation} \label{eq:P:sign}
\PP(t) \le 0 \quad \forall t \in [0, 1].
\end{equation}
We introduce the following notation for the coefficients of $\PP$
\begin{equation} \label{polynomial:coeff}
\PP(t) := \cp_0(\alpha)+\cp_1(\alpha) t + \cp_2(\alpha) t^2 + \cp_3(\alpha) t^3 + \cp_4(\alpha) t^4 + \cp_5(\alpha) t^5 + \cp_6(\alpha) t^6.
\end{equation}
Using \eqref{cr:formula} and \eqref{cb:formula}, specialized to $\NNN=1$ and with $\gamma=1+2\alpha$, we obtain
\begin{equation} \label{formula:NNN:1}
\crr = \crrs(3, 1+2\alpha, 1)=\tfrac{5-3\alpha+\sqrt{1+102\alpha+249\alpha^2}}{4+12\alpha},  \quad \cbb = \cbbs(3, 1+2\alpha, 1)=\tfrac{1-3\alpha+\sqrt{1+102\alpha+249\alpha^2}}{4+12\alpha}.
\end{equation}
Using \eqref{formula:NNN:1}, we can explicitly compute the coefficients $\cp_j$ solely in terms of $\alpha$.
The explicit expressions of these coefficients are recorded in \eqref{coeff:PP}, in Appendix~\ref{app:coef:p}.

To show that $\PP$ is negative on $[0,1]$, we introduce the degree $8$ Bernstein coefficients associated with $\PP$.  We prove that all but possibly one of these coefficients are negative, and that the remaining one is controlled by its neighboring coefficients.  The argument is elementary, but it requires checking the signs of $20$ cubics in $\alpha$ on $(0,1+\sqrt{2}]$.

\begin{lemma} \label{lemma:Bernstein}
We introduce the (rescaled) Bernstein coefficients of degree 8 associated to the polynomial \eqref{polynomial:coeff}. For any integer $0 \le j \le 8$ we define
\begin{equation} \label{Bernstein:coeff:imp}
\beta_j=\beta_j(\alpha) := \tfrac{20000(1+3\alpha)^4}{3} \sum_{i=0}^8 \cp_i(\alpha) \binom{j}{i}\binom{8}{i}^{-1},
\end{equation}
with the convention that $\cp_i(\alpha)=0$ for every $i \ge 7$.
For every $\alpha \in (0,1+\sqrt{2}]$ we have
\begin{subequations} \label{eq:lemma:pr}
\begin{align}
&\beta_j(\alpha) < 0, 
\quad \mbox{for all} \quad j \neq 6,\\
&\beta_5(\alpha) + \tfrac{2}{3} \beta_6(\alpha) <0, \\
&\beta_7(\alpha) + \tfrac{2}{3} \beta_6(\alpha) <0.
\end{align}
\end{subequations}
In particular, for $\alpha \in (0,1+\sqrt{2}]$ and $t \in [0,1]$ we have
\begin{equation} \label{PP:negative}
\PP(t) < 0.
\end{equation}
\end{lemma}
\begin{proof}[Proof of Lemma~\ref{lemma:Bernstein}]
The explicit expressions for the coefficients $\beta_j$  are recorded in Appendix~\ref{app:coef:p}, see \eqref{bernstein:12}.  The proof of the inequalities \eqref{eq:lemma:pr} is elementary but lengthy; so we defer it to \S~\ref{app:proof:PP}.

To prove \eqref{PP:negative}, we first observe that for $t \in [0,1]$
\begin{equation} \label{interp:bernstein:coeff}
\binom{8}{6}t^6(1-t)^{2}\le \frac{2}{3}\left(\binom{8}{5}t^5(1-t)^{3}+\binom{8}{7}t^7(1-t)\right).
\end{equation}
This follows from a direct computation, 
\begin{equation*}
\frac23\left(\binom{8}{5}t^5(1-t)^{3}+\binom{8}{7}t^7(1-t)\right)- \binom{8}{6}t^6(1-t)^{2} =\tfrac43 t^5(1-t)(53t^2-77t+28) \ge 0,
\end{equation*}
since the discriminant of the quadratic factor is equal to $-7<0$ and its leading coefficient is positive. 

Using the inequalities \eqref{eq:lemma:pr} and \eqref{interp:bernstein:coeff}, for $t \in [0,1]$ and $\alpha \in (0, 1+\sqrt{2}]$ we have
\begin{equation*}
\PP(t) = \tfrac{3}{20000(1+3\alpha)^4}\sum_{j=0}^8 \beta_j \binom{8}{j} t^j (1-t)^{8-j}  <\tfrac{3}{20000(1+3\alpha)^4}  \sum_{j=5}^7 \beta_j \binom{8}{j} t^j (1-t)^{8-j} \le 0.
\end{equation*}
This completes the proof, modulo verifying the bounds \eqref{eq:lemma:pr} in \S~\ref{app:proof:PP}.
\end{proof}

\begin{figure}[h]
\centering
\includegraphics[width=0.4\textwidth]{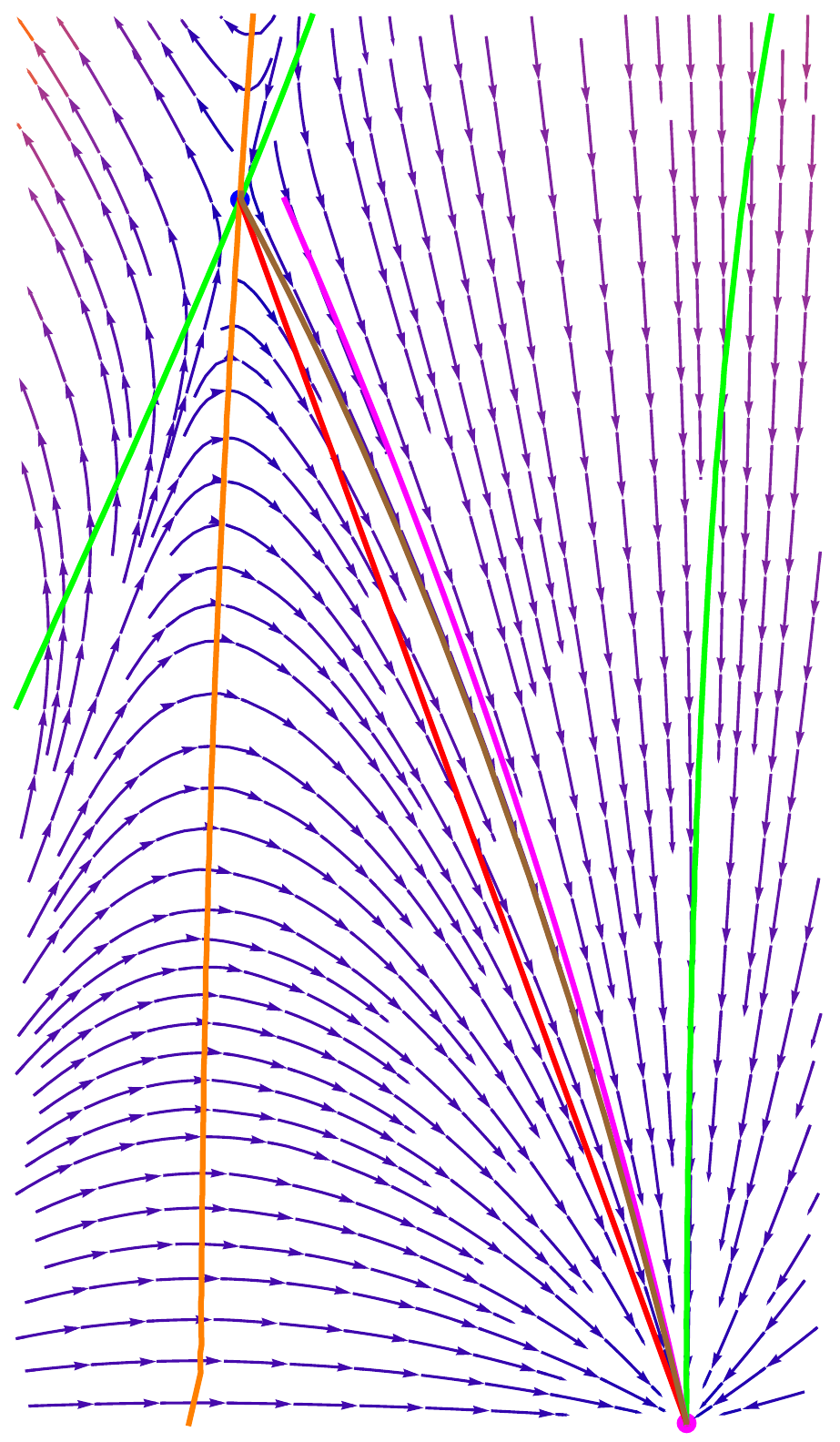}
\caption{\textbf{The implosion phase portrait for $\gamma=\frac{5}{3}$,  $d=3$ and $\NNN=1$.} The vertical axis represents $\bar Q$, while the horizontal axis represents $\bar V$. The pink point represents $(0,0)$, while the blue one represents $(-V_0, Q_0)$. The green curve represents $\{ \Delta_{\bar V}[\bar V, \bar Q]=0\}$, while the orange one represents $\{ \Delta_{\bar Q}[\bar V, \bar Q]=0\}$. In brown we have the trajectory $(\bar V(R), \bar Q(R))$. In red we have the straight line connecting $(0,0)$ to $(-V_0, Q_0)$, which serves as a lower barrier for $(\bar V(R), \bar Q(R))$. In magenta we have the upper barrier $(\mathcal{W}(q),q)$.}
 \label{fig:impl}
\end{figure}

 \begin{proof}[Proof of Proposition~\ref{prop:barrier:imp}]
From the definition \eqref{def:G:imp}, we have $G(0)=0$ and
\begin{equation*}
G(1)=-\tfrac{9}{10(1+3\alpha)}= -\tfrac{9}{10} V_0 \ge -V_0.
\end{equation*}
Thus \eqref{cond:i:bar} is satisfied. Moreover,
\begin{equation*}
G'(0)=-\tfrac{3}{5+15\alpha},
\end{equation*}
 and hence \eqref{cond:ii:bar} holds.

It remains to verify the barrier condition \eqref{cond:iii:bar}. By \eqref{ineq:upper:im}, it is enough to prove $\mathcal{B}^I(t)\le0$ on $[0,1]$. Since Lemma~\ref{pol:division} gives $\mathcal{B}^I(t)=t^2\PP(t)$, this reduces to proving
\begin{equation*}
\PP(t)\le 0 \quad \forall t\in[0,1],
\end{equation*}
which we have shown in Lemma \ref{lemma:Bernstein}, see  \eqref{PP:negative}.
\end{proof}

\subsection{Proof of Proposition~\ref{prop:ub:imp}}
\begin{proof}
The lower bound in \eqref{upp:bound:imp} has already been proved in~\cite{CSV26} (see Proposition A.1 there) by showing $\mathcal{V}(q) \ge -\sqrt{\tfrac{2\alpha}{3\gamma}} q$ for $q \in [0, Q_0]$. Hence, we focus here on showing the upper bound. 

In light of \eqref{eq:mc:V:i}, which gives $\tfrac{\Delta_{\bar V}}{\Delta_{\bar Q}}[\mathcal{V}(q),q] = \mathcal{V}'(q)$, the conditions \eqref{cond:i:bar} and \eqref{cond:iii:bar} guarantee $\mathcal{W}(q) \ge \mathcal{V}(q)$ for all $q \in [0, Q_0]$ (see also Figure \ref{fig:impl}). Therefore, using \eqref{cond:ii:bar}, we have
\begin{equation} \label{ineq:msfc}
 \rr \ge \lim_{q \to 0^+} \tfrac{\mathcal{W}(q)}{q} \ge \lim_{q \to 0^+} \tfrac{\mathcal{V}(q)}{q} = \tfrac{\underline v_1}{\underline q_1}.
\end{equation} 
To conclude the strict inequality, choose a solution $(\bar V_{\mathsf{c}},\bar Q_{\mathsf{c}})$ of \eqref{2x2:imp} with $\bar Q_{\mathsf{c}}(1)=Q_0$ and $-V_0<\bar V_{\mathsf{c}}(1)<\mathcal{W}(Q_0)<0$. By the same arguments in Proposition 2.15 and Corollary 2.16 in~\cite{CSV26}, this solution satisfies $\lim_{R \to +\infty}(\bar V_{\mathsf{c}}(R),\bar Q_{\mathsf{c}}(R))=(0,0)$, and $\bar Q_{\mathsf{c}}(R)$ is monotone decreasing. Hence we may define $\mathcal{V}_{\mathsf{c}}(q):=\bar V_{\mathsf{c}}(\bar Q_{\mathsf{c}}^{-1}(q))$ on $(0,Q_0]$.

By the same arguments as above, $\mathcal{W}$ is an upper barrier for $\mathcal{V}_{\mathsf{c}}(q)$, while $\mathcal{V}$ is a lower barrier. For a given ratio $\mathsf{R}$, there is at most one possible solution $(\bar V,\bar Q)$ to~\eqref{2x2:imp}, up to rescaling in $R$, such that $\lim_{R \to +\infty} \frac{\bar V(R)}{\bar Q(R)}=\mathsf{R}$ (since $(0,0)$ is a ``star'' for the ODE, see also Proposition 2.18 in~\cite{CSV26}). This implies that $\lim_{q \to 0^+} \tfrac{\mathcal{V}_{\mathsf{c}}(q)}{q} > \lim_{q \to 0^+} \tfrac{\mathcal{V}(q)}{q}$; and since the inequality~\eqref{ineq:msfc} applies with $\mathcal{V}_{\mathsf{c}}$ in place of $\mathcal{V}$ we obtain
\begin{equation*}
\rr \ge \lim_{q \to 0^+} \tfrac{\mathcal{V}_{\mathsf{c}}(q)}{q} > \lim_{q \to 0^+} \tfrac{\mathcal{V}(q)}{q} = \tfrac{\underline v_1}{\underline q_1}.
\end{equation*}
\end{proof}


\section{The explosion phase portrait} 
\label{exp:pp}
The goal of this section is to construct the profiles $(U, \Sigma, B)$ from Theorem~\ref{th:main} that arise as a natural continuation of the imploding solutions $(\bar U, \bar \Sigma, \bar B)$ from Theorem~\ref{th:csv}. That is, for $d=3$ and $\alpha \in (0, 1+\sqrt{2}]$, we search for solutions of \eqref{fw:ss:eq} satisfying the boundary conditions \eqref{lim:infty:ex} at $\xi=+\infty$, with $\crr = \crrs(3, 1+2\alpha, 1)$ and $\cbb=\cbbs(3, 1+2\alpha, 1)$. The existence and main properties of such profiles are stated in Theorem~\ref{th:main:sec}.

\subsection{Setup}
The analysis will be carried out in the renormalized variables, analogous to the variables in~\eqref{renorm:imp} used in the implosion phase portrait:
\begin{equation} \label{renormalized}
U(\xi) =:\xi  V (\xi), \quad  \Sigma(\xi) =:\xi  Q(\xi), \quad  B(\xi)=:\xi  H(\xi).
\end{equation}
For $d=3$, substituting \eqref{renormalized} into \eqref{fw:ss:eq}, dividing by $\xi>0$, and using $\gamma=1+2\alpha$, the equations \eqref{fw:ss:eq} become
\begin{subequations} \label{eq:main}
\begin{align}
    (V - \crr)\xi \pp_\xi Q + (\alpha Q) \xi \pp_\xi V + Q\big(V-1 + 3 \alpha  V \big)
    &= 0
    \,
    \label{eq:Q:main}
    \\
    (V - \crr)\xi \pp_\xi V + (\alpha Q)\xi \pp_\xi Q - \left(\tfrac{\alpha Q^2}{\gamma H}\right) \xi \pp_\xi H + V(V-1) + \tfrac{2\alpha^2}{\gamma} Q^2
    &= 0
    \,
    \label{eq:V:main}
    \\
   (V - \crr) \xi\pp_\xi H + (V - \crr + \cbb) H &= 0
    \,.
    \label{eq:H:main}
\end{align}
\end{subequations}
\subsection[A coupled system for V and Q alone]{A coupled system for \texorpdfstring{$V$ and $Q$}{V and Q} alone}
On any interval where $V-\crr$ and $H$ do not vanish, \eqref{eq:H:main} gives
\begin{equation} \label{eq:H}
\frac{\xi \pp_\xi H}{H} = - \frac{V - \crr + \cbb}{V - \crr}
\,.
\end{equation}
\begin{subequations} \label{eq:new}
Substituting \eqref{eq:H} into~\eqref{eq:V:main}, multiplying by $V-\crr$, and using $\gamma=1+2\alpha$, we obtain
\begin{equation}
\xi (V-\crr)^2 \pp_\xi V+ \xi \alpha Q (V-\crr)\pp_\xi Q+ V(V-1)(V-\crr)+ \alpha Q^2 (V-\crr)+ \tfrac{\alpha \cbb Q^2}{\gamma} = 0
\,.
\label{eq:V:new}
\end{equation}
This ODE is coupled with~\eqref{eq:Q:main}, which we recall is
\begin{equation}
    \xi(V - \crr) \pp_\xi Q + \xi(\alpha Q) \pp_\xi V + Q\big(V-1 + 3 \alpha V\big) = 0
    \,.
    \label{eq:Q:new}
\end{equation}
\end{subequations}
On the set where $\Delta[V,Q]\neq0$, solving the resulting $2\times2$ linear system for $\xi\partial_\xi V$ and $\xi\partial_\xi Q$, whose determinant is $\Delta[V,Q]$, we obtain
\begin{equation}
    \xi \frac{dV}{d\xi} = \frac{P_v[V,Q]}{\Delta[V,Q]}
    \,,
    \qquad
    \xi \frac{dQ}{d\xi} = \frac{P_q[V,Q]}{\Delta[V,Q]}
    \,,
    \label{eq:euler:ODE}
\end{equation}
where
\begin{subequations} \label{P:exp}
    \begin{align}
    P_q[V,Q] &:=
     Q\left[-(V-\crr)^2(V (1+ 3 \alpha )-1) + \alpha  (V-\crr) \left[ V(V-1) + \alpha Q^2 \right] + \tfrac{\alpha^2 \cbb}{\gamma} Q^2\right]
    \,
    \\
    P_v[V,Q] &:=
    (V-\crr) \left[ - V(V-1)(V-\crr) + \alpha Q^2 (\crr - \tfrac{\cbb}{\gamma} - 1 + 3 \alpha  V) \right]
    \,,
    \\
    \Delta[V,Q] &:= (V - \crr)(V-\crr -\alpha Q)(V-\crr + \alpha Q)
    \,.
\end{align}
\end{subequations}

We can then study the phase portrait associated to the autonomous system \eqref{eq:euler:ODE} (with respect to the derivative $\xi \pp_\xi$), depicted in Figure~\ref{fig:phase:portrait:1} below.
\begin{figure}[ht] 
\centering
\begin{subfigure}{0.49\textwidth}
\centering
\includegraphics[width=\textwidth]{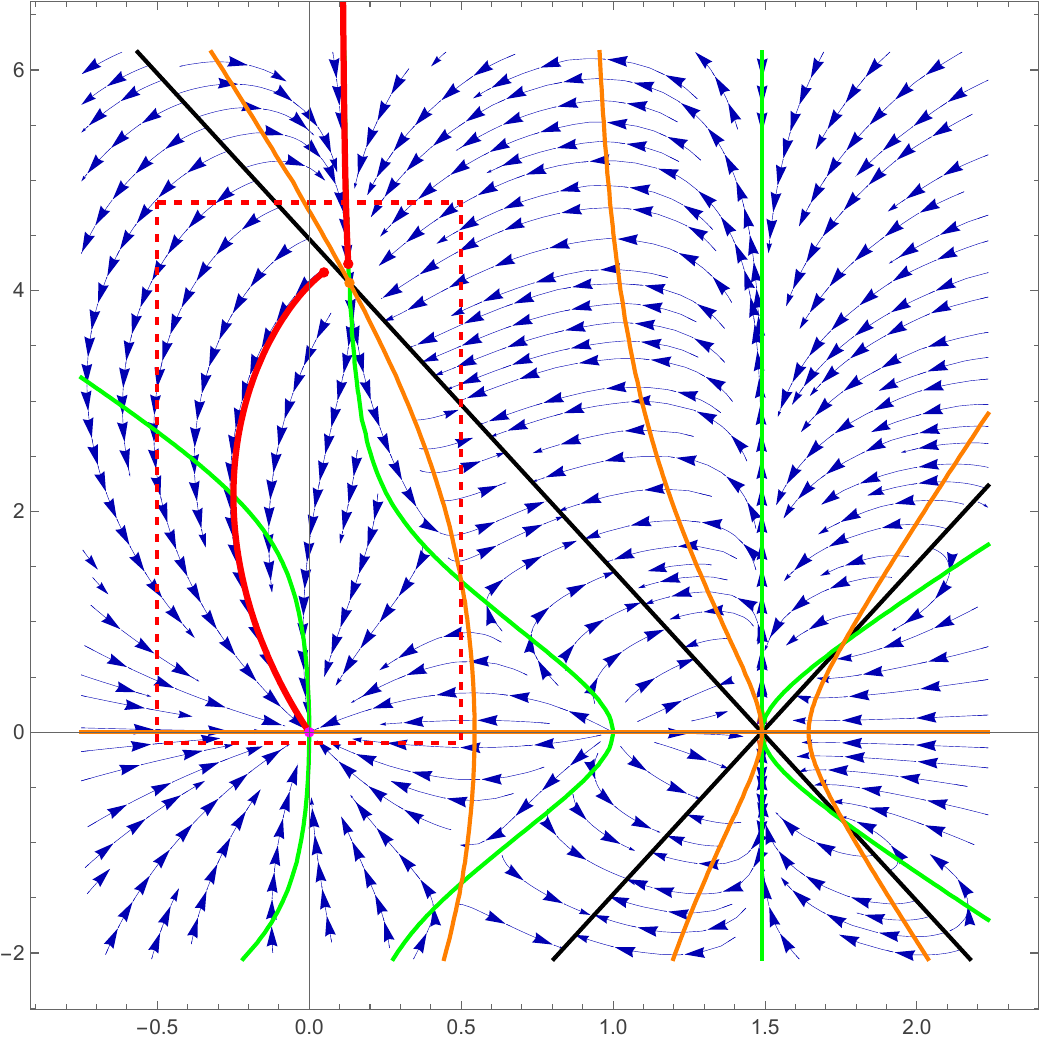}
\end{subfigure}
\hfill
\begin{subfigure}{0.49\textwidth}
\centering
\includegraphics[width=\textwidth,height=\textwidth]{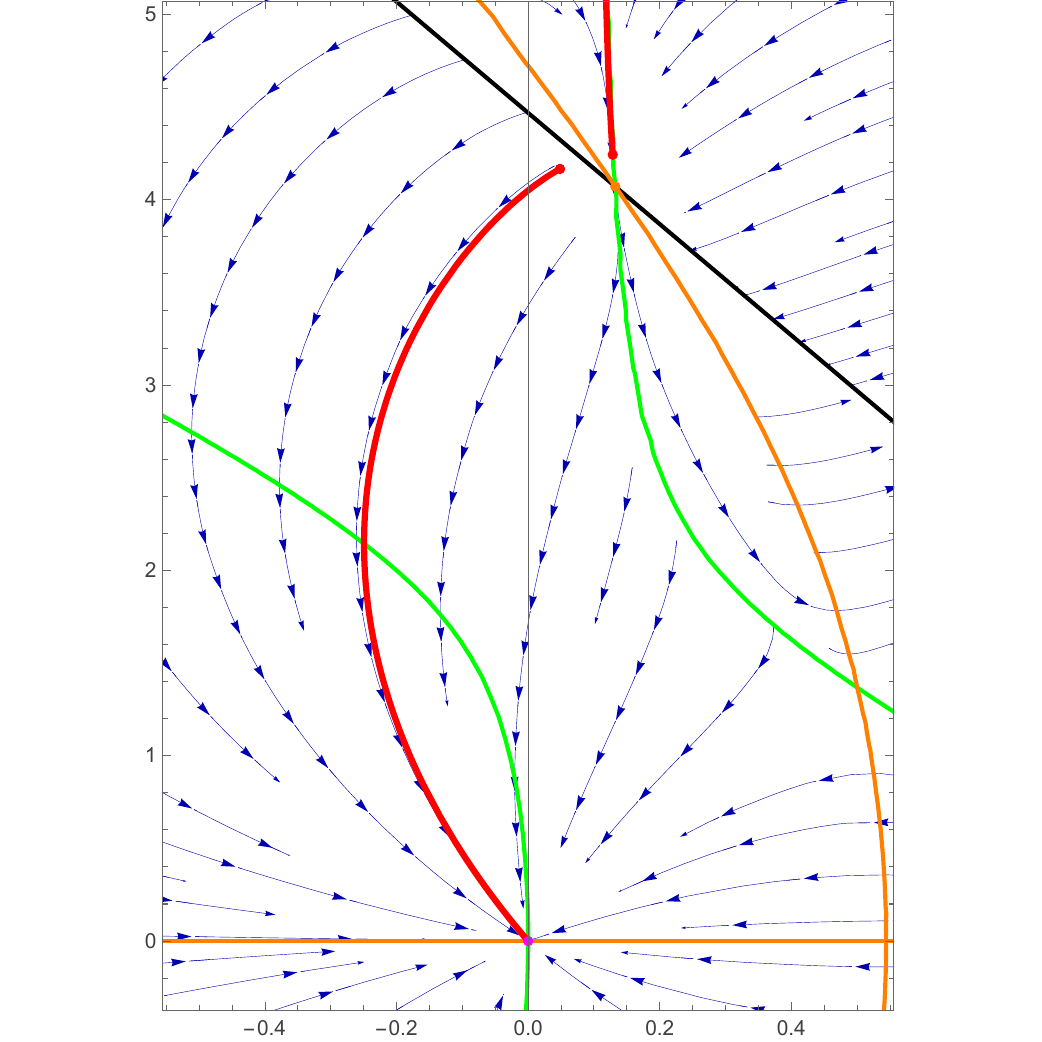}
\end{subfigure}
\caption{\textbf{The explosion phase portrait for $\gamma=\frac{5}{3}$, $d=3$ and $\NNN=1$.}
The vertical axis represents $Q$, while the horizontal axis represents $V$. The pink point represents $(0,0)$. The orange point represents the triple point $P_2$ defined in \eqref{def:V2} below. The green curve represents $\{ P_{v}[ V,  Q]=0\}$, the orange one $\{ P_{q}[ V,  Q]=0\}$ and the black one $\{ \Delta[V,Q]=0\}$. The two red curves represent the piecewise smooth profile $(V, Q)$ from Theorem~\ref{th:main:sec}.
The second figure shows a zoom of the dotted red rectangle, where most of the construction below takes place.}
\label{fig:phase:portrait:1}

\end{figure}
\subsection[A lower bound for crr]{A lower bound for \texorpdfstring{$\crr$}{crr}}
We remind the reader that the values $\crr= \crrs(3, 2\alpha+1,  1)$ and $\cbb:= \cbbs(3, 2\alpha+1, 1) =\crr-\tfrac{1}{1+3\alpha}$ have been fixed. In particular, using \eqref{cr:formula} and \eqref{cb:formula} we have
\begin{equation} \label{eq:cr:1}
\crr=\tfrac{5-3\alpha+\sqrt{1+102\alpha+249\alpha^2}}{4+12\alpha}, \quad \cbb=\tfrac{1-3\alpha+\sqrt{1+102\alpha+249\alpha^2}}{4+12\alpha}.
\end{equation}
We will use that $\crr>1$. 
Since $1+3\alpha>0$, from~\eqref{eq:cr:1} $\crr>1$ is equivalent to
\[
1+\tfrac14\sqrt{1+102\alpha+249\alpha^2}+\tfrac{1-3\alpha}{4}>1+3\alpha,
\]
or, equivalently,
\begin{equation} \label{crr:lb}
\sqrt{1+102\alpha+249\alpha^2}>15\alpha-1.
\end{equation}
If $15\alpha-1\le0$, this is immediate. If $15\alpha-1>0$, we can square both sides and obtain
\[
1+102\alpha+249\alpha^2-(15\alpha-1)^2=24\alpha^2+132\alpha>0,
\]
which proves the claim. 
We will also use that
\begin{equation} \label{eq:crr:lb}
\crr > \tfrac{3+2\alpha}{2(1+3\alpha)}.
\end{equation}
From \eqref{eq:cr:1}, using $\alpha>0$, we have
\begin{align*}
 \crr > \tfrac{3+2\alpha}{2(1+3\alpha)}\iff \sqrt{1+102 \alpha + 249\alpha^2}> 1+7\alpha \iff 88\alpha+200\alpha^2>0.
\end{align*}
\subsection{General properties of the phase portrait} \label{subsec:general:phase:portrait}
The goal of this subsection is to describe some basic features of the phase portrait for the autonomous system
\begin{equation}\label{main:ODE}
\xi \p_\xi V = \tfrac{P_v[V, Q]}{\Delta[V,Q]}, \quad \xi \p_\xi Q=\tfrac{P_q[V, Q]}{\Delta[V,Q]}.
\end{equation}
We introduce the notation
\begin{equation} 
V_0:= \tfrac{1}{1+3\alpha}, \quad Q_0 := \tfrac{1}{1+3\alpha}\sqrt{\tfrac{3 \gamma}{2\alpha}}\,.
\end{equation}
and define
\begin{equation} \label{Vinfty:def}
V_\infty:= \tfrac{-\gamma(\crr-1)+\cbb}{3\alpha\gamma} = \tfrac{1}{3 \gamma} \left( \tfrac{3}{1+3\alpha} +2 - 2 \crr\right).
\end{equation}
We will use that $V_\infty \ge0$.  Using \eqref{eq:cr:1}, and $\alpha\in(0, 1+\sqrt{2}]$, we get
\begin{align} 
 V_\infty \ge0 \iff \crr \le 1+ \tfrac{3}{2+6\alpha} \iff \tfrac{5+15 \alpha}{6+30\alpha+36\alpha^2} \ge \tfrac{\sqrt{1+102\alpha+249\alpha^2}}{6+30\alpha+36\alpha^2} \iff \notag \\
(1+2\alpha-\alpha^2)\ge 0 \iff 0<\alpha \le 1+\sqrt{2}. \label{pos:vinfty:range}
\end{align}
We will also use that $V_\infty< \tfrac{2 V_0}{3}<V_0<1$, which follows from \eqref{eq:crr:lb} and the identity
\begin{equation} \label{eq:v0:ub}
V_\infty - \tfrac{2}{3}V_0
= \tfrac{2}{3\gamma}
\left(\tfrac{3+2\alpha}{2(1+3\alpha)}-\crr\right).
\end{equation}
Indeed, since $\gamma>0$, \eqref{eq:crr:lb} shows that the right-hand side is negative, and hence
$V_\infty<\tfrac{2}{3}V_0$. The inequalities $\tfrac{2}{3}V_0<V_0<1$ follow directly from
$V_0=\tfrac{1}{1+3\alpha}$ and $\alpha>0$.

\subsubsection{The triple points}
A triple point is a point $P_i=(V_i,Q_i)$ such that
\[
 P_v[V_i,Q_i]=P_q[V_i,Q_i]=\Delta[V_i,Q_i]=0.
 \]
 In the region $\{Q\ge0\}$, there are three such points
 \[
 P_1:=(\crr,0),\qquad
 P_2:=(V_2,Q_2),\qquad
 P_3:=(V_3,Q_3).
\]
It is immediate to verify, from the definitions \eqref{P:exp}, that $P_1$ is a triple point. If we assume $V\neq \crr$, the condition $\Delta[V,Q]=0$ is equivalent to $(\crr-V)^2=\alpha^2 Q^2$. Using this identity, $V\neq \crr$, and $Q \neq 0$, the equations $P_v[V,Q]=0$ and $P_q[V,Q]=0$ collapse into the same quadratic polynomial
\begin{equation} \label{triple:eq}
F_T(V):=V(V-1)-3(V-\crr)(V-V_\infty)=-2V^2+(3\crr+3V_\infty-1)V-3\crr V_\infty=0.
\end{equation}
To check that $F_T$ admits two real solutions it is enough to show that its discriminant is positive. A direct computation gives that
\[
(3\crr+3V_\infty-1)^2-24\crr V_\infty
=(3\crr-3V_\infty-1)^2+12V_\infty(\crr-1)>0,
\]
since $V_\infty \ge 0$ and $\crr>1$, and if $V_\infty=0$ then $(3\crr+3V_\infty-1)^2-24\crr V_\infty=(3\crr-1)^2>0$. We use $V_2$ and $V_3$, with $V_2<V_3$, to indicate the two distinct solutions of $F_T(V)=0$.  Moreover, we have $F_T(\crr)=\crr(\crr-1)>0$; since the leading coefficient of the quadratic $F_T$ is negative, this implies that $V_2 < \crr < V_3$.

Note that $V_2$ has the explicit form
 \begin{subequations}  \label{def:V2}
\begin{equation}
V_2:= \tfrac{1}{4}\left( -1+3 \crr + 3 V_\infty - \sqrt{(1- 3 \crr-3 V_\infty)^2 - 24 \crr V_\infty} \right),
\end{equation}
and the corresponding value of $Q_2$, since $V_2 < \crr$, is given by
\begin{equation} \label{Q2:def}
Q_2:=\tfrac{1}{\alpha}(\crr-V_2).
\end{equation}
\end{subequations}
The coordinate $V_3$ has a similar explicit form
\begin{equation*}
V_3:= \tfrac{1}{4}\left(
-1+3\crr+3 V_\infty
+\sqrt{(1-3\crr-3V_\infty)^2-24\crr V_\infty}
\right),
\end{equation*}
and $Q_3$, since $\crr<V_3$, is given by
\begin{equation*}
Q_3:= \tfrac{1}{\alpha} (V_3-\crr).
\end{equation*}
We observe that the point $P_3$ will not play any role in our analysis. 

We will use the following bounds on $V_2$
\begin{equation} \label{bound:V2}
V_\infty \le V_2 < V_0.
\end{equation}
From \eqref{eq:v0:ub}, we know that $V_\infty<V_0<1<\crr<V_3$. Since  the leading coefficient of the quadratic $F_T$ is negative,  $F_T$ is positive on $(V_2,V_3)$, and $\crr\in(V_2,V_3)$ while $V_\infty<1<\crr$, the inequality $F_T(V_\infty)= V_\infty(V_\infty-1)\le 0$ implies $V_\infty \le V_2$. To prove the upper bound $V_2 < V_0$, by a similar argument, it is enough to show $F_T(V_0)>0$. Using $V_0<1<\crr$, we have
\begin{equation*}
F_T(V_0)> 0 \iff V_0(V_0-1) > 3 (V_0 - \crr)(V_0-V_\infty)  \iff V_\infty < V_0 - \tfrac{1}{3} \tfrac{1-V_0}{\crr-V_0} V_0,
\end{equation*}
and the last inequality follows from \eqref{eq:v0:ub} and $\crr>1$
\begin{equation*}
V_\infty \le \tfrac{2}{3} V_0 < V_0 - \tfrac{1}{3}\tfrac{1-V_0}{\crr-V_0} V_0.
\end{equation*}
We also record the inequality
\begin{equation} \label{V2:crr:2}
V_2 \le \tfrac{\crr}{2}.
\end{equation}
Similarly to what we have done previously, it will be enough to show $F_T(\tfrac{\crr}{2})>0$. By \eqref{triple:eq}, we have $ F_T\left(\frac{\crr}{2}\right)=\frac{\crr}{2}\left(2\crr-1-3V_\infty\right).$ 
Using $V_\infty<\frac23 V_0$, which follows from \eqref{eq:v0:ub}, it is enough to prove $ \crr>V_0+\frac12.$
But, from \eqref{eq:cr:1} and $V_0=\frac{1}{1+3\alpha}$,
\[
\crr-V_0-\tfrac12
=
\tfrac{\sqrt{1+102\alpha+249\alpha^2}-1-9\alpha}{4(1+3\alpha)}>0,
\]
since
\[
1+102\alpha+249\alpha^2-(1+9\alpha)^2
=84\alpha(1+2\alpha)>0.
\]
\subsubsection{Stationary points}
We observe that in the region $\{ Q \ge 0\}$ there are three stationary points, that is, those points such that $P_v[V_i, Q_i]=P_q[V_i, Q_i] =0$ but $ \Delta[V_i, Q_i]  \neq 0$. Those are
\begin{equation*}
P_0 := (0,0), \quad P_4 := ( V_0, Q_0), \quad P_5:=(1, 0).
\end{equation*}
These three points will not play a role in the analysis of this section, but $P_4$ will be relevant in Section~\ref{smooth:explosion}.
There are also two stationary points at infinity (when $Q=+\infty$):
\begin{equation*}
P_6 := ( V_\infty, + \infty), \quad P_7 := (\crr, +\infty).
\end{equation*}
In particular, $P_6$ will play an important role in this section. For the profiles $(V,Q)$ from Theorem~\ref{th:main:sec}, in Proposition~\ref{prop:infty} we will show that $\lim_{\xi \to 0^+} (V(\xi), Q(\xi)) = (V_\infty, +\infty)$.

\subsubsection{Branches of $P_q[V,Q]=0$}
We now study the branches of $P_q[V,Q]=0$ in the region $\{V<\crr, Q\ge0\}$. The set $\{ Q=0\}$ is one branch. For the remaining branches with $Q>0$, we divide $P_q[V,Q]=0$ by $Q$ and study the polynomial equation
\begin{equation} \label{Pq:div:Q}
-(V-\crr)^2(V (1+3\alpha)-1) + \alpha  (V-\crr) \left[ V(V-1) + \alpha Q^2 \right] + \tfrac{\alpha^2 \cbb}{\gamma} Q^2=0.
\end{equation}
Solving for $Q$ and taking the nonnegative square root because $Q \ge 0$, the solutions are given by
\begin{equation} \label{eq:Q:lvset}
Q =f_Q(V) :=  \tfrac{1}{\alpha}\sqrt{\tfrac{(V-\crr)P_{2,Q}[V] }{V-1+3\alpha V_\infty}},
\end{equation}
where we have used the definition of $V_\infty$ \eqref{Vinfty:def} to write $V-\crr+\tfrac{\cbb}{\gamma}=V-1+3\alpha V_\infty$, and where
\[
P_{2,Q}[V]:=(V-\crr)\bigl(V(1+3\alpha)-1\bigr)-\alpha V(V-1)
=(1+2\alpha)V^2+(\alpha-1-\crr(1+3\alpha))V+\crr.
\]

We are left to determine the sign of the argument of the square root. A straightforward computation gives the roots
\begin{equation*}
V_{\pm}^Q :=\tfrac{1}{2(1+2\alpha)}
\Bigl(
\crr(1+3\alpha)+1-\alpha \pm
\sqrt{\bigl(\crr(1+3\alpha)+1-\alpha\bigr)^2-4\crr(1+2\alpha)}
\Bigr).
\end{equation*}
We also introduce the pole of the radicand
\begin{equation} \label{def:V*Q}
V_*^Q:=1-3\alpha V_\infty=\crr-\tfrac{\cbb}{\gamma}
= \left( 1-\tfrac{1}{\gamma} \right) \cbb + \tfrac{1}{1+3\alpha}.
\end{equation}
Since $\crr>1$, we have
\begin{equation*}
P_{2,Q}[0]=\crr>0, \quad P_{2,Q}[\crr]=-\alpha \crr (\crr-1)<0,
\end{equation*}
and since the coefficient of the quadratic term of $P_{2,Q}$ is positive, we deduce that $0<V_{-}^Q < \crr < V_{+}^Q$.
By \eqref{def:V*Q} and $\cbb>0$,
 \begin{equation*}
 \tfrac{1}{1+3\alpha}<V_*^Q < \crr.
\end{equation*}
For $\crr=\crrs(3,\gamma,1)$, using the explicit formula \eqref{eq:cr:1}, we compute
\begin{equation*}
P_{2,Q}[V_*^Q]
=\tfrac{\alpha\left((9\alpha^2-1)\sqrt{1+102\alpha+249\alpha^2}-(147\alpha^3+141\alpha^2+39\alpha+1)\right)}
{4(18\alpha^3+21\alpha^2+8\alpha+1)}.
\end{equation*}
If $0<\alpha\le \frac13$, this is negative. If $\alpha>\frac13$, then
\begin{align*}
&\left(147\alpha^3+141\alpha^2+39\alpha+1\right)^2
-(9\alpha^2-1)^2(1+102\alpha+249\alpha^2)\\
&\qquad=\alpha(1440\alpha^5+33192\alpha^4+35748\alpha^3+13128\alpha^2+1572\alpha-24)>0,
\end{align*}
and hence $P_{2,Q}[V_*^Q]<0$ also in this case. Therefore $V_-^Q<V_*^Q<\crr$.
Since $V-\crr<0$ in the region under consideration, the radicand is nonnegative exactly for $V\in(-\infty,V_-^Q]$ and $V\in(V_*^Q,\crr)$. To conclude, the branches in $\{V<\crr, Q\ge0\}$ are
\begin{align*}
\Gamma^Q_1&:= \{ (V, f_Q(V)), V \in (-\infty, V_-^Q] \},\\
\Gamma^Q_2&:= \{ (V, f_Q(V)), V \in (V_*^Q, \crr) \},\\
\Gamma^Q_3&:= \{ (V, 0), V \in (-\infty, \crr) \}.
\end{align*}
Since $V_0<V_*^Q$, $Q_0>0$, and $P_q[V_0,Q_0]=0$, the point $(V_0,Q_0)$ lies on the first nonzero branch $\Gamma_1^Q$. Hence $V_0<V_-^Q$. Combining this with \eqref{bound:V2} and $V_-^Q<V_*^Q<V_+^Q$, we obtain
\begin{equation} \label{ineq:series}
V_2<V_0 < V_-^Q<V_*^Q <V_+^Q.
\end{equation}
We now prove the following lemma that will be useful later.
\begin{lemma} \label{lemma:pos:Pq}
    Let $\NNN=1$, $d=3$ and $\alpha \in (0, 1+\sqrt{2}]$. 
Assume $V<V_2$, $Q>0$ and $V+\alpha Q <\crr$, then
\begin{equation*}
P_q[V,Q]>0.
\end{equation*}

\end{lemma}
\begin{proof}[\it Proof of Lemma~\ref{lemma:pos:Pq}]
From \eqref{Pq:div:Q}, we have
\begin{equation*}
\tfrac{P_q[V,Q]}{Q}= \alpha^2 (V-V_*^Q) Q^2 - P_{2,Q}[V] (V-\crr).
\end{equation*}
For fixed $V<V_2$, the coefficient of $Q^2$ in $\frac{P_q[V,Q]}{Q}$ is negative from \eqref{ineq:series}, so $\frac{P_q[V,Q]}{Q}$ is concave on $0\le Q\le(\crr-V)/\alpha$. It is therefore enough to check positivity at the two endpoints of this interval.

Since $V<V_2<V_-^Q<\crr$, we have $P_{2,Q}[V]>0$ and $V-\crr<0$, hence
\[
\lim_{Q \to 0^+}\tfrac{P_q[V,Q]}{Q}
= - P_{2,Q}[V] (V-\crr)>0.
\]
By substituting $Q=\frac{\crr-V}{\alpha}$, for $V<V_2$ we obtain
\[
\tfrac{P_q[V,\tfrac{\crr-V}{\alpha}]}{\crr-V}
= (V-\crr)F_T(V)>0,
\]
 because $V-\crr<0$ and, for $V<V_2$, the quadratic $F_T$ is negative outside the interval between its two roots.

\end{proof}

\subsection[The triple point (V2, Q2)]{The triple point \texorpdfstring{$(V_2, Q_2)$}{(V2, Q2)}}
To study the triple point $(V_2, Q_2)$, it will be useful to consider the ODE \eqref{main:ODE} under the desingularized rescaling defined by $\partial_\psi:=\xi\Delta\,\partial_\xi$. The autonomous ODE \eqref{main:ODE} transforms into
\begin{equation} \label{eq:tran:ODE}
\pp_\psi V = P_v[V,Q], \quad \pp_\psi Q = P_q[V,Q].
\end{equation}
Since $P_v[V_2,Q_2]=P_q[V_2,Q_2]=\Delta[V_2,Q_2]=0$, the rational vector field \eqref{main:ODE} is singular at $(V_2,Q_2)$. After the desingularization \eqref{eq:tran:ODE}, this point becomes an equilibrium; the computation below shows that it is a sink. We compute here the eigenvalues $\lambda_\pm$ and eigenvectors $\eta_{\pm}$ associated to the sink $(V_2,Q_2)$. Since $V_2\neq\crr$, $Q_2>0$, and $P_v[V_2,Q_2]=P_q[V_2,Q_2]=0$, the bracketed factors in $P_v$ and $\frac{P_q}{Q}$ vanish at $(V_2,Q_2)$. Using \eqref{triple:eq} and the definition of $Q_2$ in \eqref{Q2:def}, a direct computation gives
\begin{subequations} \label{jacobian:P2}
\begin{equation}
\begin{aligned}\label{jacobian:P2:11}
J_{11}
&:= \pp_V P_v[V_2,Q_2] \\
&= (V_2-\crr) \pp_V \left[
- V(V-1)(V-\crr)
+ \alpha Q^2 \left(\crr - \tfrac{\cbb}{\gamma} - 1 + 3\alpha V\right)
\right]|_{V=V_2, Q=Q_2} \\
&= - (V_2 - \crr)^2 (2 V_2 -1 + 3(\crr-V_\infty)).
\end{aligned}
\end{equation}
\begin{equation}
\begin{aligned}
J_{12} \label{jacobian:P2:12}
&:= \pp_Q P_v[V_2,Q_2]  \\
&= (V_2-\crr) \pp_Q \left[
- V(V-1)(V-\crr)
+ \alpha Q^2 \left(\crr - \tfrac{\cbb}{\gamma} - 1 + 3\alpha V\right)
\right]|_{V=V_2, Q=Q_2} \\
&=  -6 \alpha (V_2 - \crr)^2 (V_2 - V_\infty).
\end{aligned}
\end{equation}
\begin{equation}
\begin{aligned} \label{jacobian:P2:21}
J_{21}
&:=\pp_V P_q[V_2,Q_2] \\
&=Q_2 \pp_V \left[
-(V-\crr)^2(V (1+3\alpha)-1)
+ \alpha  (V-\crr) \left[ V(V-1) + \alpha Q^2 \right]
+ \tfrac{\alpha^2 \cbb}{\gamma} Q^2
\right]_{V=V_2, Q=Q_2}  \\
&= -\tfrac{(V_2-\crr)^2}{\alpha}
[2-\alpha+3\alpha(\crr-V_\infty)-2(1+2\alpha)V_2].
\end{aligned}
\end{equation}
\begin{equation}
\begin{aligned} \label{jacobian:P2:22}
J_{22}
&:=\pp_Q P_q[V_2,Q_2] \\
&=Q_2 \pp_Q \left[
-(V-\crr)^2(V (1+3\alpha)-1)
+ \alpha  (V-\crr) \left[ V(V-1) + \alpha Q^2 \right]
+ \tfrac{\alpha^2 \cbb}{\gamma} Q^2
\right]_{V=V_2, Q=Q_2}  \\
&=2 (\crr-V_2)^2 (V_2-1+3\alpha V_\infty).
\end{aligned}
\end{equation}
\end{subequations}
The eigenvalues of the matrix
\begin{equation*}
J := \begin{pmatrix}
    J_{11} & J_{12} \\
    J_{21} & J_{22}
\end{pmatrix}
\end{equation*}
are the zeros of the characteristic polynomial
\begin{equation*}
\lambda^2 - (J_{11}+J_{22}) \lambda + (J_{11}J_{22}-J_{12}J_{21}).
\end{equation*}
In particular, we have
\begin{equation}\label{ev:j}
\lambda_{\pm} = \tfrac{J_{11}+J_{22}\pm \sqrt{(J_{11}-J_{22})^2+4 J_{12} J_{21}}}{2}.
\end{equation}
To prove the eigenvalues $\lambda_{\pm}$ are real, we will show $J_{12}\le0$ and $J_{21}<0$.
For $J_{12}$, from the lower bound \eqref{bound:V2}, we have 
\begin{equation} \label{j12:b}
J_{12}= -6\alpha (V_2-\crr)^2(V_2-V_\infty) \le 0.
\end{equation}
Showing that $J_{21}$ is negative is equivalent to proving
\begin{equation} \label{eq:ineq:V2}
2-\alpha+3\alpha(\crr-V_\infty)-2(1+2\alpha)V_2>0,
\end{equation}
and this inequality follows from  \eqref{bound:V2}, \eqref{eq:v0:ub} and $\crr>1$.
Indeed, we observe that \eqref{eq:v0:ub} and $\crr>1$ imply the inequality
\[
2-\alpha+3\alpha(\crr-V_\infty)>2-\alpha+3\alpha(1-\tfrac{2}{3}V_0)
=2+2\alpha-\tfrac{2\alpha}{1+3\alpha}.
\]
Then,  $J_{21}<0$ is a simple consequence of the upper bound \eqref{bound:V2}
\begin{align} \label{ineq:tp:V2}
V_0>V_2
& \implies V_0+\tfrac{\alpha}{1+2\alpha} > V_2 \implies \tfrac{1}{1+3\alpha}+\tfrac{\alpha}{1+2\alpha} > V_2 \notag \\
& \implies 2+2\alpha - \tfrac{2 \alpha}{1+3\alpha} > 2(1+2\alpha)V_2  \implies 2-\alpha+3\alpha(\crr-V_\infty)>2(1+2\alpha)V_2 \notag .
\end{align}
We now turn to proving that $\lambda_{\pm}<0$. We start by observing that
\[\operatorname{Tr}J=J_{11}+J_{22}=-(V_2-\crr)^2(1+3\crr-3\gamma V_\infty).\]
This trace is negative. Indeed, from \eqref{eq:v0:ub} and $\crr>1$ we have
\begin{equation*}
3\gamma V_\infty < \tfrac{2\gamma}{1+3\alpha}
=\tfrac{2(1+2\alpha)}{1+3\alpha}<2<1+3\crr.
\end{equation*}
We now prove that $\operatorname{det}J = J_{11}J_{22}-J_{12}J_{21} = - 2(\crr-V_2)^4 ((1+3\alpha)V_2-1)(3 V_\infty-4 V_2 + 3 \crr -1)>0$.
From the definition of $V_2$ \eqref{def:V2} we have
\begin{equation} \label{eq:mag}
3 V_\infty-4 V_2 + 3 \crr -1
=\sqrt{(1-3\crr-3V_\infty)^2-24\crr V_\infty}>0.
\end{equation}
Moreover, $V_2<V_0$ and hence $((1+3\alpha)V_2-1) <0$. This implies $\operatorname{det}J >0$; combined with $\operatorname{Tr} J <0$, this implies
$\lambda_- \le \lambda_+<0$.
Using $\crr>1$,  \eqref{bound:V2} and $V_\infty \ge 0$, we have
\begin{equation} \label{ineq:uff}
J_{11}-J_{22}=(\crr-V_2)^2 (3(1-\crr)-V_2-3(V_2-V_\infty)-6\alpha V_\infty) <0,
\end{equation}
which in particular implies the strict inequalities
\begin{equation} \label{sign:lambda}
\lambda_- < \lambda_+<0.
\end{equation}

The eigenvectors of $J$ have the expression
\begin{equation} \label{eq:evec:J}
\eta_\pm= \begin{pmatrix}
    \tfrac{J_{11}-J_{22}\pm \sqrt{(J_{11}-J_{22})^2 + 4 J_{12} J_{21}}}{2} \\
    J_{21}
\end{pmatrix}.
\end{equation}
We use $\mmp$ and $\mmm$ to denote the $\frac{dV}{dQ}$ slopes associated to $\eta_+$ and $\eta_-$, respectively:
\begin{equation} \label{eq:slope:ev}
    \mmp = \tfrac{J_{11}-J_{22}+ \sqrt{(J_{11}-J_{22})^2 + 4 J_{12} J_{21}}}{2 J_{21}},
    \quad
    \mmm = \tfrac{J_{11}-J_{22}- \sqrt{(J_{11}-J_{22})^2 + 4 J_{12} J_{21}}}{2 J_{21}}.
\end{equation}

We now prove the ordering
\begin{equation} \label{eq:slope:ordering}
    \mmm > \mmp>-\alpha.
\end{equation}

From \eqref{j12:b} and \eqref{eq:ineq:V2}, we have $J_{12}J_{21}\ge0$. Combining this with \eqref{ineq:uff}, we deduce $\mmm>0\ge\mmp$.
We observe the slopes are the two roots of
\begin{equation} \label{eq:slope:quadratic}
    F_{\mm}(\mm):=J_{21}\mm^2+(J_{22}-J_{11})\mm-J_{12}=0.
\end{equation}
Since $\mmm>0\ge\mmp$ and $-\alpha<0$, the point $-\alpha$ lies to the left of $\mmm$. Since $J_{21}<0$, the quadratic $F_{\mm}$ is negative outside the interval with endpoints $\mmp$ and $\mmm$. Therefore, to prove $\mmp>-\alpha$, it is enough to show $F_{\mm}(-\alpha)<0$.  Using the explicit expressions for $J_{ij}$, we have
\[
\begin{aligned}
F_{\mm}(-\alpha)
&=\alpha^2J_{21}-\alpha(J_{22}-J_{11})-J_{12} \\
&=-\alpha(\crr-V_2)^2(1+\alpha)(3\crr+3V_\infty-1-4V_2)<0,
\end{aligned}
\]
where the last inequality follows from \eqref{eq:mag}.

\subsection{Main result: existence of self-similar shock profiles}
We now state the main result of this section, and provide a short roadmap to its proof.
\begin{theorem} \label{th:main:sec}
Fix $d=3$, $1<\gamma \le 3+ 2\sqrt{2}$, and $\NNN=1$, and define $\crr=\crrs(3, \gamma, 1)$ and $\cbb=\cbbs(3, \gamma, 1)$ using the expressions \eqref{cr:formula} and \eqref{cb:formula}. Fix also the constants $\underline v_1$ and $\underline q_1>0$ defined in \eqref{eq:lim:infty}. Then, there exists a unique renormalized shock profile $(V,Q)$ for the forward self-similar system \eqref{eq:euler:ODE} such that
\begin{enumerate}[leftmargin=2em]
\item Piecewise smoothness: there exists $0<\xiRH<+\infty$ such that both $V$ and $Q$ are $C^\infty$ on $(0, \xiRH)$ and $(\xiRH, +\infty)$,
\item  $(V,Q)$ solves \eqref{eq:euler:ODE} on $(0, \xiRH)$ and $(\xiRH, +\infty)$,
\item Rankine--Hugoniot jump conditions: across the jump at $\xi=\xiRH$, the traces $(V,Q)(\xiRH^\pm)$ satisfy the $(V,Q)$-components of the Rankine--Hugoniot relations \eqref{RH:jump} and the Lax entropy inequality \eqref{lax:ineq},
\item Asymptotics at $\xi=0$: there exists a constant $\overline q_0>0$, with $\overline v_0=V_\infty$, such that as $\xi\to0^+$ the solution has the following local behavior
\begin{equation*}
V(\xi) =\overline v_0 + O(\xi^{\frac{2(2+3(\gamma-1))}{2+3\gamma}}), \quad Q(\xi)=\overline q_0 \xi^{-\frac{2+3(\gamma-1)}{2+3\gamma}} + O(\xi^{\frac{2+3(\gamma-1)}{2+3\gamma}}),
\end{equation*}
\item Positivity: for $\xi >0$ we have $Q(\xi)>0$,
\item Matching at $\xi=+\infty$: the asymptotic behavior of $(V,Q)$ matches the renormalized form of the implosion asymptotics \eqref{eq:lim:infty}
\begin{equation} \label{lim:infty:ex:ode}
\lim_{\xi \to +\infty} \xi^{\frac{1}{\crr}}V(\xi) = \underline v_1, \quad \lim_{\xi \to +\infty}  \xi^{\frac{1}{\crr}} Q(\xi) =\underline q_1.
\end{equation}
\end{enumerate}

\end{theorem}
The proof will consist of three steps:
\begin{itemize}[leftmargin=2em]
    \item In Proposition~\ref{prop:power:implosion}, given $\underline v_1$ and $\underline q_1$, we construct a local solution near $\xi=+\infty$ to the autonomous system \eqref{eq:euler:ODE} such that $V$ and $Q$ satisfy \eqref{lim:infty:ex:ode}. In Proposition~\ref{prop:sub:sonic}, we then show that such a solution $(V,Q)$ can be extended up to $\xi_s$, and that at such point the solution hits the sonic line $V+\alpha Q=\crr$ \textit{on the left} of the sonic point $(V_2, Q_2)$. This is achieved by constructing a suitable, explicit lower barrier (see Proposition~\ref{prop:barrier:ex} and Figure \ref{fig:supersonic:region}),
    \item We construct a trajectory $(V^\infty, Q^\infty)$ in Proposition~\ref{prop:infty} that connects $P_2$ to the stationary point at infinity $P_6$,
    \item Using the Rankine--Hugoniot jump conditions \eqref{RH:jump}, we identify a point $\xiRH$ such that the Hugoniot image of the outer state, $\mathrm{RH}(V(\xiRH^+),Q(\xiRH^+))$, lies on the trajectory $(V^\infty,Q^\infty)$ (see Figure \ref{fig:RH:locus}).
\end{itemize}
\subsection[The subsonic region V + alpha Q < crr]{The subsonic region \texorpdfstring{$V+\alpha Q < \crr$}{V + alpha Q < crr}}
The goal of this section is to construct a solution $(V,Q)$ of the reduced forward self-similar system \eqref{eq:euler:ODE} near $\xi=+\infty$; $H$ will be recovered later from \eqref{eq:H}. The terminal data from the backward problem impose only the leading asymptotics of $V$ and $Q$ at $\xi=+\infty$, so the first task is to show that these leading coefficients determine a local branch of solutions to this reduced system. In this section we carry this out by constructing a convergent power series at $\xi=+\infty$.

\begin{figure}[h]
\centering
\includegraphics[width=0.7\textwidth]{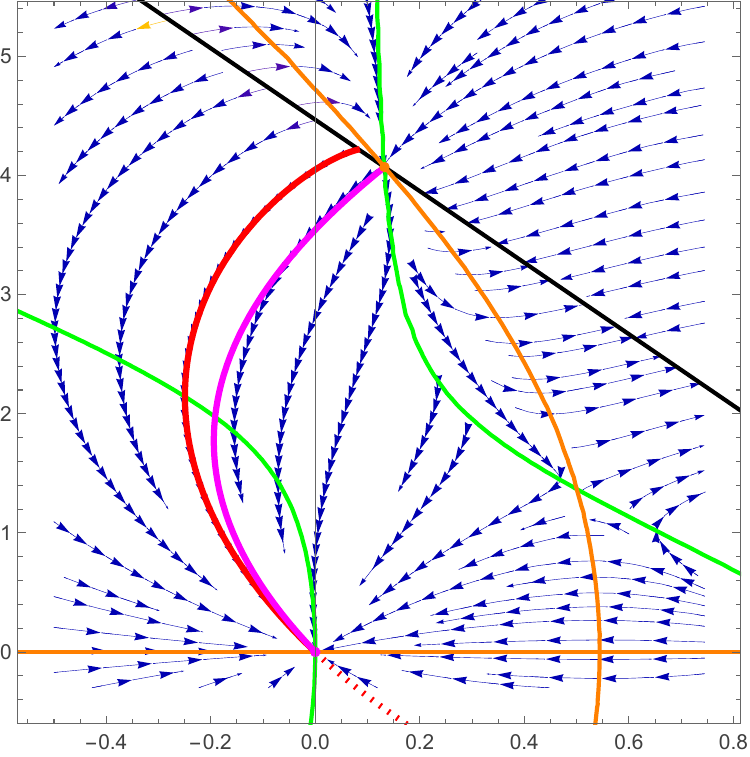}
\caption{\textbf{The subsonic region for $\gamma=\frac{5}{3}$, $d=3$ and $\NNN=1$.} The vertical axis represents $Q$, while the horizontal axis represents $V$. The pink point represents $(0,0)$. The orange point represents the triple point $P_2$. The green curve represents $\{ P_{v}[ V,  Q]=0\}$, the orange one $\{ P_{q}[ V,  Q]=0\}$ and the black one $\{ \Delta[V,Q]=0\}$. In red we have the trajectory $(V,Q)$ from Proposition~\ref{prop:sub:sonic}, extended until it hits the sonic line at $(V_s, Q_s)$. The magenta curve represents the barrier $\mathcal{W}_E$.}
\label{fig:supersonic:region}
\end{figure}

\begin{proposition}[\bf Convergence of power series for $V$ and $Q$ near $\xi=\infty$]
\label{prop:power:series:infty}
For any given $|v_1|<+\infty$ and $0<q_1 < +\infty$, there exist unique real coefficients $\{ v_n \}_{n \ge 2}$ and  $\{ q_n \}_{n \ge 2}$, defined recursively in \eqref{rec:infty} below, such that the functions defined by the convergent power series expansions
\begin{equation}
V(\xi) =  \sum_{n\geq 1} v_n \xi^{- \frac{n}{\crr}},
\qquad
Q(\xi) =  \sum_{n\geq 1}  q_n  \xi^{- \frac{n}{\crr}},
\label{eq:V:Q:R=infty}
\end{equation}
solve \eqref{eq:euler:ODE} in its domain of convergence. Moreover, there exists a constant $\mathsf{C}(\gamma, \crr) \gg 1$ such that the power series in~\eqref{eq:V:Q:R=infty} converges uniformly and absolutely for $\xi > \mathsf{C}^{\crr} (|v_1|^2 + |q_1|^2)^{\frac{\crr}{2}}$.
\end{proposition}
\begin{proof}[Proof of Proposition~\ref{prop:power:series:infty}]
To compute the recursive relation that defines the coefficients in \eqref{eq:V:Q:R=infty}, it is more convenient to substitute the ansatz \eqref{eq:V:Q:R=infty} into the system \eqref{eq:new}; this leads to the relations
\begin{subequations} \label{rec:infty}
\begin{align}
\crr (n-1) q_n &= \sum_{\substack{m+j = n, \\ m, j \ge 1}} (  \alpha j + m -\crr (1+3 \alpha ) ) q_m v_j,   \\
\crr(n-1)\,v_n
&=
\sum_{\substack{m+j=n\\ m,j\ge1}}
\left[
(2j-1-\crr)v_m v_j
+\alpha\left(j-\crr+\tfrac{\cbb}{\gamma}\right)q_m q_j
\right]
\\
&
+{\bf 1}_{n\geq 3}  \!\!\sum_{\substack{m+\ell+j=n\\ m,\ell,j\ge1}}
(1-\tfrac{j}{\crr})\,
 v_m\bigl(v_\ell v_j+\alpha q_\ell q_j\bigr).
\end{align}
\end{subequations}
For any $n \ge 2$,  \eqref{rec:infty} defines a recursive relation that determines the values of $v_n, q_n$ from knowledge of $v_1$ and $q_1$. To conclude the proof of the proposition, it remains to prove a positive radius of convergence, with the stated exterior lower bound for $\xi$. We will prove that for a sufficiently large constant $\mathsf{C}=\mathsf{C}(\gamma, \crr) \ge 1$ and a sufficiently small $0 <\mathsf{C}'=\mathsf{C}'(\gamma, \crr)<1$, the following bounds hold for $n\ge 1$
\begin{subequations} \label{bound:infty}
\begin{align}
| v_n| \le \mathsf{C}'  \mathsf{C}^n ( | v_1 |^2 + |q_1|^2 )^{ \frac{n}{2}} n^{-\frac{3}{2}},  \\
| q_n| \le \mathsf{C}'  \mathsf{C}^n ( | v_1 |^2 + |q_1|^2 )^{ \frac{n}{2}} n^{-\frac{3}{2}}.
\end{align}
\end{subequations}
We proceed by induction. The bound \eqref{bound:infty} is trivially true for $n=1$ as long as $\mathsf{C}'  \mathsf{C} \ge 1$.
We use the convolution bounds $\sum_{j=1}^{n-1} j^{-\frac{3}{2}} (n-j)^{-\frac{3}{2}} \le 12 n^{-\frac{3}{2}}$ and $\sum_{\substack{j,k\ge 1\\ j+k\le n-1}} j^{-\frac{3}{2}}  k^{-\frac{3}{2}} (n-j-k)^{-\frac{3}{2}} \le 36 n^{-\frac{3}{2}}$. Substituting the inductive bounds \eqref{bound:infty} into the relations \eqref{rec:infty}, and applying these convolution estimates, gives for $n \ge 2$
\begin{subequations} \label{bound:ugly:infty}
\begin{align}
| q_n |&  \le \tfrac{24 (\mathsf{C}')^2 (1+\alpha+\crr(1+3\alpha))}{\crr}   \mathsf{C}^n ( | v_1 |^2 + |q_1|^2 )^{ \frac{n}{2}} n^{-\frac{3}{2}}, \\
| v_n | &\le \tfrac{ (\mathsf{C}')^2 }{\crr }( 24 ( (2+\alpha)  + \crr(1+\alpha) +\tfrac{\cbb}{\gamma}  ) + 72 (1+\tfrac{1}{\crr}) ) \mathsf{C}^n ( | v_1 |^2 + |q_1|^2 )^{ \frac{n}{2}} n^{-\frac{3}{2}}.
\end{align}
\end{subequations}
From \eqref{bound:ugly:infty}, we deduce that if we choose $\mathsf{C}'$ small enough so that $ \frac{1}{\mathsf{C}'} \ge \frac{1}{\crr} \max\{ 24 (1+\alpha+\crr(1+3\alpha)),  24 ( (2+\alpha)  + \crr(1+\alpha) +\frac{\cbb}{\gamma}  ) + 72 (1+\frac{1}{\crr}) \}  $, and $\mathsf{C}$ large enough such that $\mathsf{C} \mathsf{C}' \ge 1$, the inductive bounds \eqref{bound:infty} hold for any $n \ge 1$. The claim about the radius of convergence follows now from \eqref{bound:infty}. After increasing $\mathsf{C}$ if necessary, the convergent series satisfies $|V|+\alpha |Q|<\frac{\crr}{2}$ on the stated exterior interval, so $\Delta[V,Q]\neq0$ there. Hence the system \eqref{eq:new} is equivalent to \eqref{eq:euler:ODE} on this interval.
\end{proof}

 Proposition~\ref{prop:power:series:infty} gives a local branch for each finite $v_1$ and each positive finite $q_1$. We now select the branch relevant to the continuation problem by taking these coefficients to be precisely $\underline v_1$ and $\underline q_1$ from \eqref{eq:lim:infty}. This chooses the solution near $\xi=+\infty$ that matches the data at the time of implosion.
 We record this specialization as the following proposition.
\begin{proposition}
\label{prop:power:implosion}
Let $d =3$,  $1< \gamma \le 3+2\sqrt{2}$ and $\NNN = 1$. Consider $\underline v_1$ and $\underline q_1>0$ as determined in \eqref{eq:lim:infty}.

There exists a unique smooth solution $(V,Q)$ to \eqref{eq:euler:ODE} and a constant $\mathsf{C}(d, \gamma, \crr) \gg 1$ such that $(V,Q)$ is defined on $\xi > \mathsf{C}^{\crr} (|\underline v_1|^2 + |\underline q_1|^2)^{\frac{\crr}{2}}$, and $(V,Q)$ satisfies the boundary conditions
\begin{equation} \label{bc:infty:explosion}
\lim_{\xi \to +\infty} \xi^{\frac{1}{\crr}}V(\xi) = \underline v_1, \quad \lim_{\xi \to +\infty}  \xi^{\frac{1}{\crr}}Q(\xi) =\underline q_1.
\end{equation}
\end{proposition}
\begin{proof}[Proof of Proposition~\ref{prop:power:implosion}]
Existence is the specialization of Proposition~\ref{prop:power:series:infty} with
$v_1=\underline v_1$ and $q_1=\underline q_1$. We prove uniqueness of such a trajectory. Let
$(\widetilde V,\widetilde Q)$ be another solution, defined for all sufficiently large
$\xi$, satisfying the same boundary conditions as $(V,Q)$:
\[
\lim_{\xi\to+\infty}\xi^{\frac{1}{\crr}}\widetilde V(\xi)=\underline v_1,
\qquad
\lim_{\xi\to+\infty}\xi^{\frac{1}{\crr}}\widetilde Q(\xi)=\underline q_1.
\]
In particular, both $(V,Q)$ and $(\widetilde V, \widetilde Q)$ converge to $(0,0)$ as $\xi \to +\infty$.  Expanding the vector field \eqref{main:ODE} at $(0,0)$ gives, for $\xi$ large,
\begin{align*}
\left|\xi\pp_\xi\left(\xi^{\frac{1}{\crr}}(\widetilde V-V)\right)\right|
\lesssim
|\widetilde V-V|+|\widetilde Q-Q|,
    \\
\left|\xi\pp_\xi\left(\xi^{\frac{1}{\crr}}(\widetilde Q-Q)\right)\right|
\lesssim
|\widetilde V-V|+|\widetilde Q-Q|.
\end{align*}
Applying Gronwall, for any $\xi'>\xi$ we get
\begin{align*}
&\xi^{\frac{1}{\crr}}
\left(|\widetilde V(\xi)-V(\xi)|+|\widetilde Q(\xi)-Q(\xi)|\right) \\
&\qquad\lesssim
(\xi')^{\frac{1}{\crr}}
\left(|\widetilde V(\xi')-V(\xi')|+|\widetilde Q(\xi')-Q(\xi')|\right)
\exp\left(\int_\xi^{\xi'}  s^{-\frac{1}{\crr}}\frac{ds}{s}\right).
\end{align*}
The integral is bounded uniformly as $\xi'\to+\infty$, and the common boundary condition implies
\begin{equation*}
\lim_{\xi' \to + \infty}(\xi')^{\frac{1}{\crr}}\left[ |\widetilde V(\xi')-V(\xi')|+|\widetilde Q(\xi')-Q(\xi')|\right] = 0.
\end{equation*}
  Thus
$(\widetilde V,\widetilde Q)=(V,Q)$ for all sufficiently large $\xi$. Standard ODE uniqueness then gives equality on the common interval of definition.
\end{proof}

The remaining task is global. Starting from the analytic solution constructed at $\xi=+\infty$, we will continue the trajectory inward in $\xi$. The next subsection uses the sharper bound \eqref{upp:bound:imp} on the ratio $\tfrac{\underline v_1}{\underline q_1}$ to control the maximal development of the local solution at $\xi=+\infty$.
\subsection[The trajectory (V,Q) does not hit the triple point (V2, Q2)]{The trajectory \texorpdfstring{$(V,Q)$}{(V,Q)} does not hit the triple point \texorpdfstring{$(V_2, Q_2)$}{(V2, Q2)}}
The goal of this section is to prove that the solutions constructed in Proposition~\ref{prop:power:implosion} extend up to some radius $ \xi_s$ at which $ V + \alpha  Q = \crr$. Using in addition the information from \eqref{upp:bound:imp}
 \[
 \tfrac{\underline v_1}{\underline q_1} < \rr,
 \]
we will show that the trajectory does not hit the sonic point $(V_2,Q_2)$.
The conclusion is contained in the following proposition.
\begin{proposition} \label{prop:sub:sonic}
    Let  $d=3$, $\NNN=1$ and $\alpha\in(0, 1+\sqrt{2}]$. Consider the unique local solution $(V,Q)$ to \eqref{main:ODE} at $\xi=+\infty$ that satisfies the boundary conditions \eqref{bc:infty:explosion} constructed in Proposition~\ref{prop:power:implosion}. There exists $0 \le \xi_s<+\infty$ such that $(V,Q)$ extends smoothly and uniquely on $(\xi_s, +\infty)$. Moreover, as $\xi \to \xi_s^+$ the solution $(V, Q)$ hits the sonic line $V+\alpha Q = \crr$ \textit{on the left} of the sonic point $(V_2, Q_2)$. That is
    \begin{equation} \label{eq:xis:expl}
\lim_{\xi \to \xi_s^+} V(\xi) = V_s, \quad \lim_{\xi \to \xi_s^+} Q(\xi) = Q_s, \quad V_s+\alpha Q_s = \crr, \quad Q_s > Q_2.
\end{equation}

\end{proposition}
The proof has two steps. First, we show that the solution $(V,Q)$ cannot hit the sonic point $(V_2,Q_2)$ by constructing a suitable lower barrier; see Proposition~\ref{prop:barrier:ex}. Second, we show that $(V,Q)$ actually hits the sonic line by ruling out escape to infinity; see Lemma~\ref{lemma:escape}.

Similarly to what we have done for the implosion, we will construct a suitable lower barrier
\[
\mathcal{W}_E \colon [0,Q_2] \to (-\infty, V_2].
\]
We introduce the variable $t:=\tfrac{Q}{Q_2}$, a function $F \colon [0, 1] \to (-\infty, V_2]$ such that
\begin{equation*}
\mathcal{W}_E(Q) := F\left(\tfrac{Q}{Q_2} \right).
\end{equation*}

We seek a function $F$ such that
\begin{enumerate}[leftmargin=2em]
\item \label{cond:bar:ex:i} Endpoint conditions: $F(0)=0$ and $F(1)=V_2$. This condition ensures that the barrier curve connects the origin $(0,0)$ to the triple point $(V_2,Q_2)$.
\item \label{cond:bar:ex:ii} Derivative lower bound: $ F'(0) \ge  Q_2 \rr$. This is equivalent to $\mathcal{W}_E'(0)\ge \rr$. Together with the strict bound $\underline v_1/\underline q_1<\rr$, it ensures that near the origin the trajectory $(V,Q)$ lies to the left of the barrier $\mathcal{W}_E$.
\item \label{cond:bar:ex:iii} Barrier condition: we must ensure that for $t \in (0,1)$ we have
\begin{equation} \label{eq:lower:barrier}
\tfrac{P_v}{P_q}[F(t), Q_2 t] \le \tfrac{F'(t)}{Q_2}
\end{equation}
\item \label{cond:bar:ex:iv} Eigendirection condition at $P_2$: if $\mm_{\pm}$
denotes the slopes of the eigendirections of the desingularized system at $P_2$ \eqref{eq:tran:ODE} (see the definitions \eqref{eq:slope:ev}), we ask
\[
\frac{F'(1)}{Q_2} >
\mmm>\mmp >-\alpha.
\]
This condition guarantees that solutions to the left of the barrier $\mathcal{W}_E$ cannot enter the point $P_2$.
\end{enumerate}

\begin{proposition}\label{prop:barrier:ex}
Let $d=3$, $\alpha \in (0, 1 + \sqrt{2}]$ and $\NNN=1$. Then, the function
\begin{equation} \label{def:F}
F(t) := V_2 t^2 + Q_2 \rr (t-t^2)
\end{equation}
satisfies the conditions \eqref{cond:bar:ex:i}--\eqref{cond:bar:ex:iv}.
\end{proposition}
We define
\begin{equation} \label{def:b:ex}
\mathcal{B}(t):=P_v[F(t), Q_2 t] - \tfrac{F'(t)}{Q_2} P_q[F(t), Q_2 t].
\end{equation}
For $0<t<1$, the definition of $F$ gives $F(t)<V_2$ and $F(t)+\alpha Q_2t<\crr$; hence Lemma~\ref{lemma:pos:Pq} gives $P_q[F(t),Q_2t]>0$. Thus, the inequality \eqref{eq:lower:barrier} is equivalent to
\begin{equation} \label{eq:B:equiv}
\mathcal{B}(t) \le 0.
\end{equation}

 $\mathcal{B}(t)$ is a polynomial of degree at most $8$ in $t$. However, we have the following lemma, analogous to Lemma~\ref{pol:division}.
 \begin{lemma}\label{lemma:pol:division}
 Let $\mathcal{B}(t)$ be the polynomial defined in \eqref{def:b:ex}. Then, $\mathcal{B}(t)$ is divisible by $t^2(1-t)$. In particular we have
\begin{equation} \label{def:q}
 \QQ(t):=\mathcal{B}(t) t^{-2} (1-t)^{-1},
 \end{equation}
 where $\QQ$ is a polynomial of degree at most $5$ in $t$.
 \end{lemma}
 \begin{proof}[Proof of Lemma~\ref{lemma:pol:division}]
 The proof is identical to the proof of Lemma~\ref{pol:division}, once we observe that
 \begin{equation*}
\begin{gathered}
P_v[0,0]=P_q[0,0]=P_v[V_2, Q_2]=P_q[V_2, Q_2]=0, \\
(\pp_V P_v,\pp_Q P_v)[0,0]=(\crr^2,0),\quad
(\pp_V P_q,\pp_Q P_q)[0,0]=(0,\crr^2).
\end{gathered}
 \end{equation*}
 \end{proof}
We write
\begin{equation} \label{coeff:QQ}
\QQ(t) := \cq_0(\alpha) + \cq_1(\alpha) t + \cq_2(\alpha) t^2 + \cq_3(\alpha) t^3 + \cq_4(\alpha) t^4 + \cq_5(\alpha) t^5.
\end{equation}
The explicit expressions for the coefficients $\cq_i$ are recorded in Appendix~\ref{app:coef:b}; see \eqref{coeff:qq:i}.

\begin{lemma} \label{lemma:Bernstein:q}
We introduce the Bernstein coefficients of degree 5 associated to the polynomial \eqref{coeff:QQ}. For any integer $0 \le j \le 5$ we define
\begin{equation} \label{Bernstein:coeff:exp}
\BB_j(\alpha) :=  \sum_{i=0}^j \cq_i(\alpha) \binom{j}{i}\binom{5}{i}^{-1}.
\end{equation}
For every $0 \le j \le 5$ and $\alpha \in (0, 1+\sqrt{2}]$, we have
\begin{equation} \label{bound:BB}
\BB_j(\alpha) < 0.
\end{equation}
\end{lemma}
\begin{proof}[Proof of Lemma~\ref{lemma:Bernstein:q}]
The explicit expressions for the coefficients $\BB_j$ are recorded in Appendix~\ref{app:coef:b}; see \eqref{coeff:bb:i}. For each fixed rational value of $\alpha$, the bound \eqref{bound:BB} can be verified elementarily, since it reduces to determining the signs of six constant algebraic expressions $\BB_j(\alpha)$. In \S\ref{app:elementary}, we give an elementary proof for monatomic and diatomic gases.

Extending the same proof to the whole interval $\alpha\in(0,1+\sqrt{2}]$ is also elementary, but requires substantial computation. We provide two computer-assisted approaches: Appendix~\ref{app:rational-cover-bernstein-q} gives an implementation using a rational cover of the whole interval (in \textsc{Python}), while Appendix~\ref{app:CA} gives a proof based on interval arithmetic (using \textsc{SageMath}).
 \end{proof}
At this point we have all the elements we need to prove Proposition~\ref{prop:barrier:ex}.
\begin{proof}[Proof of Proposition~\ref{prop:barrier:ex}]
From the definition \eqref{def:F}, we have $F(0)=0$ and $F(1)=V_2$. Thus, \eqref{cond:bar:ex:i} is satisfied. Moreover,
\begin{equation*}
F'(t)=Q_2\rr+2(V_2-Q_2\rr)t.
\end{equation*}
Hence we have $F'(0) = Q_2\rr$, and \eqref{cond:bar:ex:ii} holds.

We now verify the barrier condition \eqref{cond:bar:ex:iii}. By the definition \eqref{def:b:ex}, this is equivalent to showing that
\begin{equation*}
\mathcal{B}(t)\le0 \quad \forall t\in[0,1].
\end{equation*}
By Lemma~\ref{lemma:pol:division}, $\mathcal{B}(t)=t^2(1-t)\QQ(t)$. Writing $\QQ$ in the degree $5$ Bernstein basis and using \eqref{Bernstein:coeff:exp}, we have
\begin{equation*}
\QQ(t)=\sum_{j=0}^{5}\binom{5}{j}\BB_j t^j(1-t)^{5-j}.
\end{equation*}
By Lemma~\ref{lemma:Bernstein:q}, $\BB_j < 0$ for every $0\le j\le5$. Hence $\QQ(t)\le0$ on $[0,1]$, and therefore $\mathcal{B}(t)\le0$ on $[0,1]$. 

To complete the proof, we are left with verifying the eigendirection condition \eqref{cond:bar:ex:iv}. 
By \eqref{eq:slope:ordering}, $\mmm>\mmp>-\alpha$, so it is enough to prove $\frac{F'(1)}{Q_2}>\mmm$. Since $\frac{F'(1)}{Q_2}=2\frac{V_2}{Q_2}-\rr>0\ge\mmp$, and since $\mm_{\pm}$ are the roots of the quadratic equation $F_{\mm}=0$ defined in \eqref{eq:slope:quadratic}, whose quadratic coefficient is negative, this is equivalent to
\[
F_{\mm}\left(\tfrac{F'(1)}{Q_2} \right)<0.
\]
Differentiating \eqref{def:b:ex} at $t=1$, and using $P_v(V_2,Q_2)=P_q(V_2,Q_2)=0$, gives
\[
-\QQ(1)=\mathcal{B}'(1)=-Q_2 F_{\mm}\left(\tfrac{F'(1)}{Q_2} \right).
\]
Hence, because $Q_2>0$, the condition $F_{\mm}\left(\frac{F'(1)}{Q_2} \right)<0$ is equivalent to $\QQ(1)<0$. Evaluating the Bernstein expansion at $t=1$ gives $\QQ(1)=\BB_5(\alpha)$, and Lemma~\ref{lemma:Bernstein:q} gives $\BB_5(\alpha)<0$. Therefore $\frac{F'(1)}{Q_2} >\mmm>\mmp>-\alpha$, which proves \eqref{cond:bar:ex:iv} and completes the proof.
\end{proof}

\begin{lemma}
\label{lemma:escape}
Let $\NNN=1$, $d=3$ and $\alpha\in (0, 1+\sqrt{2}]$.
Consider the region
\begin{equation*}
\mathcal{T}:=\{(V,Q)\in\RR^2:\ Q>0,\ V+\alpha Q<\crr,\, (Q>Q_2 \ \text{or}\ V<\mathcal{W}_E(Q))\}.
\end{equation*}
Consider any solution to the ODE $(V,Q)$ \eqref{eq:euler:ODE} with initial conditions $(V,Q)(\xi_{\mathrm{in}}) \in \mathcal{T}$. Then, $Q(\xi)$ is a strictly monotone decreasing function of $\xi$ and there exists a maximal $0\le\xi_s<\xi_{\mathrm{in}}$ such that $(V,Q)$ can be extended uniquely and smoothly on $(\xi_s, \xi_{\mathrm{in}})$ and
\begin{equation} \label{eq:sonic:2}
\lim_{\xi \to \xi_s^+} V(\xi) = V_s, \quad \lim_{\xi \to \xi_s^+} Q(\xi) = Q_s, \quad V_s+\alpha Q_s = \crr, \quad Q_s > Q_2.
\end{equation}

\end{lemma}
\begin{proof}[Proof of Lemma~\ref{lemma:escape}]
The monotonicity of $Q$ follows immediately from Lemma~\ref{lemma:pos:Pq}, once we observe that $\mathcal{T} \subset \{(V,Q): Q>0, V<V_2, V+\alpha Q <\crr \}$ and $\Delta[V,Q]<0$ in $\mathcal{T}$.

By standard ODE theory the solution $(V,Q)$ can be extended to maximal $\xi_s \ge 0$ such that $(V,Q)\in\mathcal T$ for all $\xi_s<\xi\le \xi_{\mathrm{in}}$. We also observe that $Q_s = \lim_{\xi \to \xi_s^+} Q(\xi)$ is well defined, by monotonicity of $Q$.
We will now show that the trajectory $(V(\xi), Q(\xi))$ is contained in a compact subset of $\overline{\mathcal{T}}$.
To accomplish this, for $\kappa\ge0$ we define the lower barrier
\begin{equation*}
\ell(t)=(\ell_v(t), \ell_q(t)) := ( -\kappa t,  t),
\end{equation*}
where $\kappa$ is a parameter to be chosen. A direct computation gives
\begin{align}
&P_v[\ell_v(t), \ell_q(t)] + \kappa P_q[\ell_v(t), \ell_q(t)] = \notag\\
& t^2 \left( 2 \kappa^2 \alpha ( \kappa^2 + \alpha)t^2 + \kappa \alpha \left( (5 \crr-1)\kappa^2 + \alpha \tfrac{4+ \crr +6\alpha + 7 \crr \alpha+12 \crr \alpha^2}{1+5\alpha + 6 \alpha^2}\right)t + \alpha \crr \left(\kappa^2(3 \crr-1)+3\alpha V_\infty \right) \right).  \label{eq:ell}
\end{align}
Since  $\alpha>0$, $\crr>1$, and $V_\infty \ge 0$, for $t> 0$ we have $P_v[\ell_v(t), \ell_q(t)] + \kappa P_q[\ell_v(t), \ell_q(t)] \ge 0$.
We define $\kappa:= \max\{ -\tfrac{V(\xi_{\mathrm{in}})}{Q(\xi_{\mathrm{in}})}, 0\}$. We then have $V(\xi_{\mathrm{in}})+\kappa Q(\xi_{\mathrm{in}}) \ge 0$, and the computation \eqref{eq:ell} now guarantees that $\ell$ is a lower barrier for $(V,Q)$ backward-in-$\xi$. That is, for $\xi_s<\xi<\xi_{\mathrm{in}}$ we have
\begin{equation} \label{ineq:slope}
 -\tfrac{V(\xi)}{Q(\xi)} \le \kappa \implies \alpha \le \tfrac{\crr-V}{Q} \le \tilde{\kappa}:= \kappa+\tfrac{\crr}{Q(\xi_{\mathrm{in}})}<+\infty,
\end{equation}
where we used the monotonicity of $Q$.

By taking advantage of the monotonicity of $Q$, we introduce the map $\mathcal{V} \colon [Q(\xi_\mathrm{in}), Q_s] \to \RR$ such that $\mathcal{V}(q) = V(Q^{-1}(q))$. Then, 
\begin{equation}\label{mcV:diff}
\tfrac{d\mathcal{V}}{d q} = \tfrac{P_v}{P_q}[\mathcal{V}(q), q].
\end{equation}
We now prove that there exists a constant $\mathsf{C}(\tilde{\kappa})>0$ such that the following differential inequality holds for $q$ large enough:
\begin{equation} \label{diff:ineq}
q \tfrac{d}{dq} \left( \tfrac{\crr-\mathcal{V}}{q} \right) \le - \mathsf{C}(\tilde{\kappa})<0.
\end{equation}
To prove \eqref{diff:ineq}, we first observe that, for $K \in [\alpha, \tilde{\kappa}]$, we have
\begin{align*}
P_q[\crr - K Q, Q] &= Q^4 K ( \gamma K^2 - \alpha^2) +\mathcal{O}(Q^3) \le Q^4 \tilde{\kappa} ( \gamma \tilde{\kappa}^2 - \alpha^2)  + \mathcal{O}(Q^3), \\
P_v[\crr - K Q, Q]+K P_q[\crr-K Q, Q] &= 2 \alpha Q^4 K^2 (K^2 +\alpha)+\mathcal{O}(Q^3) \ge 2 \alpha^3(\alpha^2+\alpha) Q^4+ \mathcal{O}(Q^3)  ,
\end{align*}
where $\mathcal{O}(Q^3)$ is uniform in $K$ (since $K$ lies in a compact interval). By differentiating and using \eqref{mcV:diff} we obtain
\begin{align*}
q \tfrac{d}{dq} \left( \tfrac{\crr-\mathcal{V}}{q} \right)
&= - \tfrac{P_v[\mathcal{V}(q), q] + \tfrac{\crr-\mathcal{V}(q)}{q} P_q[\mathcal{V}(q), q]}{P_q[\mathcal{V}(q),q]} \\
&\le - \tfrac{2\alpha^3(\alpha^2+\alpha)}{\tilde{\kappa} ( \gamma \tilde{\kappa}^2 - \alpha^2)} + O(\tfrac{1}{q}),
\end{align*}
which proves \eqref{diff:ineq} for $q>q_*$, where $q_*$ is large enough.

If $Q_s>q_*$, integrating from $q_*$ to $Q_s>q_*$ leads to 
\begin{equation*}
\tfrac{\crr-\mathcal{V}(Q_s)}{Q_s} \le \tfrac{\crr-\mathcal{V}(q_*)}{q_*} - \mathsf{C}(\tilde{\kappa}) \log \tfrac{Q_s}{q_*}.
\end{equation*}
Since $\tfrac{\crr-\mathcal{V}(q_*)}{q_*} > \alpha$, we obtain an upper bound for $Q_s$
\begin{equation*}
Q_s \le q_* \exp \left( \tfrac{\crr- \mathcal{V}(q_*)-\alpha q_*}{\mathsf{C}(\tilde{\kappa}) q_*} \right).
\end{equation*}
If $Q_s \le q_*$, observe that the same bound holds since the right-hand side is larger than $q_*$.
Writing $Q_{\mathrm{in}}:=Q(\xi_{\mathrm{in}})$, this upper bound for $Q_s$ together with \eqref{ineq:slope} gives a uniform bound for $\mathcal V(q)$ on $Q_{\mathrm{in}}<q<Q_s$. In particular the trajectory $\{ V(\xi), Q(\xi) \}$ is contained in a compact set. We now show that also $V_s = \lim_{\xi \to \xi_s^+} V(\xi)$ is well defined. We observe that $\frac{d\mathcal V}{dq}=\frac{P_v}{P_q}$ remains bounded, backward-in-$\xi$ along any trajectory. If the trajectory does not accumulate at $P_2$, this is true by compactness, since in $\overline{\mathcal T}$, $P_q$ vanishes only at $P_2$ and $Q=0$. If it does accumulate at $P_2$ we consider the desingularized system \eqref{eq:tran:ODE}
\[
\partial_\psi V=P_v[V,Q],\qquad \partial_\psi Q=P_q[V,Q].
\]
For this system, $P_2$ is a hyperbolic sink. By \eqref{sign:lambda}, its eigenvalues satisfy
\[
\lambda_-<\lambda_+<0,
\]
and the corresponding eigenvectors $\eta_{\pm}$ are given by \eqref{eq:evec:J}. 
Hence every nontrivial trajectory of the desingularized system which converges to $P_2$ has tangent asymptotic to one of the eigendirections $\eta_\pm$. This is incompatible with remaining in $\overline{\mathcal{T}}$: indeed, $\mathcal{W}_E'(Q_2)=\frac{F'(1)}{Q_2}>\mmm>\mmp$, and for $Q<Q_2$ close to $Q_2$ a curve through $P_2$ with slope $\mm_\pm$ lies to the right of the barrier $V=\mathcal{W}_E(Q)$.
In particular, $\frac{P_v}{P_q}$ is bounded along the trajectory $(V(\xi), Q(\xi))$. This implies $\mathcal{V}$ is a Lipschitz function and hence admits a limit as $q \to Q_s^-$. By the discussion in Subsection~\ref{subsec:general:phase:portrait}, there are no stationary points in the interior of $\mathcal T$. Hence $(V_s,Q_s)\in\partial\mathcal T$.
The barrier condition \eqref{cond:bar:ex:iii} excludes the boundary $V=\mathcal{W}_E(Q)$, the eigendirection condition \eqref{cond:bar:ex:iv} excludes $(V_2,Q_2)$, and monotonicity of $Q$ excludes $Q_s=0$. Hence the only remaining possibility is
\[
V_s+\alpha Q_s=\crr.
\]
Since $(V_s,Q_s)\neq (V_2,Q_2)$ and the trajectory stays on the left of the barrier, this gives $Q_s>Q_2$.
\end{proof}

With this Lemma in mind, the conclusion in Proposition~\ref{prop:sub:sonic} is now immediate.
\begin{proof}[Proof of Proposition~\ref{prop:sub:sonic}]
    Consider the trajectory $(V,Q)$ constructed in Proposition~\ref{prop:power:implosion}. There exists a large, but finite $\bar \xi <+\infty$ such that
    \begin{equation*}
V(\bar \xi) < \mathcal{W}_E(Q(\bar \xi)).
\end{equation*}
This follows from~\eqref{bc:infty:explosion}, the strict upper bound \eqref{upp:bound:imp}, and the condition~\eqref{cond:bar:ex:ii} $\mathcal{W}'_E(0) \ge \rr$. In particular $(V(\bar \xi), Q(\bar \xi)) \in \mathcal{T}$, and we can apply Lemma~\ref{lemma:escape}.
\end{proof}

\subsection[The trajectory connecting P2 to P6]{The trajectory connecting \texorpdfstring{$P_2$}{P2} to \texorpdfstring{$P_6$}{P6}}
We first construct the branch of $(V,Q)$ near the origin $\xi=0$ by means of a power series.
 It will be convenient to denote
 \begin{equation*}
 \mu :=\tfrac{2+3(\gamma-1)}{2+3 \gamma}.
 \end{equation*}
 We aim to find a solution for \eqref{eq:new} by postulating
 \begin{equation} \label{ansatz:R:0}
V^\infty (\xi) =  \sum_{n \ge 0} \overline v_n \xi^{ 2 \mu n}, \quad Q^\infty (\xi) = \xi^{-\mu} \left( \sum_{n \ge 0} \overline q_n \xi^{ 2 \mu n} \right).
 \end{equation}
 \subsubsection{Recursion relation for the coefficients $\overline v_n, \overline q_n$}
Plugging the ansatz \eqref{ansatz:R:0} into the equations \eqref{eq:new} leads to the following relations for $\overline v_0$ and $\mu$
\begin{subequations} \label{1st:order:R}
\begin{equation}
\mu ( \overline v_0 - \crr) = (1+3\alpha) \overline v_0 - 1, \quad
(1-\mu)(\overline v_0 - \crr) + \tfrac{\cbb}{\gamma}=0.
\end{equation}
\end{subequations}
Solving the linear system gives
\begin{equation} \label{values:infty}
\overline v_0 = \tfrac{1-\crr + \tfrac{\cbb}{\gamma}}{3\alpha} = V_\infty, \quad \mu = -\tfrac{1 + (1+3\alpha)( \tfrac{\cbb}{\gamma}-\crr)}{(1+3\alpha) \crr -1 -\tfrac{\cbb}{\gamma}} = \tfrac{2 + 3 (\gamma -1)}{2 + 3 \gamma}.
\end{equation}
For $n \ge 1$, from \eqref{eq:Q:new} we have
\begin{equation} \label{eq:R:0:Q}
2 \mu n ( V_\infty - \crr) \overline q_n + \overline q_0 (1+3\alpha + \mu (2 \alpha n - 1)) \overline v_n = F_1(n),
\end{equation}
where
\begin{equation} \label{def:F:R:0}
F_1(n) :=- \sum_{m+j=n, \, m, j \ge 1} (1+3\alpha + \mu (2 n -1 + 2 (\alpha-1) j ))\overline v_j \overline q_m.
\end{equation}
From \eqref{eq:V:new}, we obtain
\begin{equation} \label{eq:R:0:V}
\alpha (1-\mu) \overline q_0^2 \overline v_n+ 2 \alpha \mu n ( \overline v_0 - \crr) \overline q_0 \overline q_n = F_2(n),
\end{equation}
where
\begin{align}
F_2(n) &:=
-2 \mu \sum_{j+m+l=n-1, j,m,l \ge 0} j \overline v_j
(\overline v_m - \delta_{m0}\crr)(\overline v_l - \delta_{l0}\crr)
\notag\\
&\quad
- \sum_{j+m+l=n-1, j,m,l \ge 0} \overline v_j
(\overline v_m - \delta_{m0})(\overline v_l - \delta_{l0}\crr)
\notag\\
&\quad
-\alpha \sum_{\substack{j+m+l=n\\0\le j,m,l\le n-1}}
\left[
\mu(2m-1)(\overline v_j-\delta_{j0}\crr)
+\overline v_j-\delta_{j0}\left(\crr-\tfrac{\cbb}{\gamma}\right)
\right]\overline q_m\overline q_l.
 \end{align}
 We observe that the previous recursion does not fix $\overline q_0$, which is a scaling parameter corresponding to the scaling symmetry generated by $\xi \partial_\xi$. 

 If we introduce the matrix
 \begin{equation}
 M_n := \begin{bmatrix}
(1+3\alpha + \mu (2 \alpha n -1)) \overline q_0 & 2 \mu n (V_\infty - \crr) \\
\alpha (1-\mu) \overline q_0^2 & 2 \alpha \mu n ( V_\infty - \crr) \overline q_0
 \end{bmatrix}
 \end{equation}
 we can summarize the above system as
 \begin{equation} \label{summ:rec:0}
 M_n \begin{bmatrix} \overline v_n \\ \overline q_n
 \end{bmatrix} =
 \begin{bmatrix}F_1(n) \\ F_2(n)
 \end{bmatrix},
 \end{equation}
 for $n \ge 1$.
 To show the system is solvable, we compute
 \begin{equation}
 \det M_n =  2\alpha^2 \mu n (3+ 2 \mu n)( V_\infty - \crr) \overline q_0^2 \neq 0,
 \end{equation}
where the nonvanishing follows from $\overline q_0>0$, $\mu>0$, and $V_\infty<1<\crr$.
 We can now prove the following lemma.
 \begin{lemma}\label{lemma:power:R:0}
Define $\mu$ and $\overline v_0$ as in \eqref{values:infty}, fix $\overline q_0>0$ and define the coefficients $\overline v_n, \overline q_n$ by \eqref{summ:rec:0}. There exists a constant $\mathsf{C} \ge 1$ such that the series
\begin{equation*}
V^\infty (\xi) =  \sum_{n \ge 0} \overline v_n \xi^{ 2 \mu n}, \quad Q^\infty (\xi) = \xi^{-\mu} \left( \sum_{n \ge 0} \overline q_n \xi^{ 2 \mu n} \right),
\end{equation*}
converge uniformly and absolutely in $0 < \xi < \left(\mathsf{C}^{-1}\overline q_0^2\right)^{\frac{1}{2\mu}}$. The functions $V^\infty$ and $Q^\infty$ solve the system \eqref{main:ODE} for all $\xi \in \left(0, \left(\mathsf{C}^{-1}\overline q_0^2\right)^{\frac{1}{2\mu}} \right)$.
 \end{lemma}
 \begin{proof}[Proof of Lemma~\ref{lemma:power:R:0}]
We only have to verify the radius of convergence of the power series. Similarly to what we have done in the proof of Proposition~\ref{prop:power:series:infty}, keeping track of the powers of $\overline q_0$, we can show that there are constants $0<\mathsf{C} < + \infty$ and $0 < \mathsf{C}' \le 1$ such that, for every $n\ge1$,
\begin{subequations} \label{bound:power:R:0}
\begin{align}
| \overline v_n | \le \mathsf{C}' \mathsf{C}^n \overline q_0^{-2n}  n^{-\frac{3}{2}}, \\
| \overline q_n | \le \mathsf{C}' \mathsf{C}^n \overline q_0^{1-2n}  n^{-\frac{3}{2}}.
\end{align}
\end{subequations}
The proof of \eqref{bound:power:R:0} follows the same induction as the proof of the bounds \eqref{bound:infty}, with the above $\overline q_0$-weights, and we do not repeat it for the sake of brevity.
 \end{proof}
Next, we show that the trajectory $(V^\infty, Q^\infty)$ converges to the sonic point $P_2$. This will be achieved by using the inverse branch $V=\Gamma_1(Q)$ of $P_q=0$ and the vertical line $V=V_2$ as barriers. We will also show some monotonicity properties of the trajectory $(V^\infty, Q^\infty)$ that we will use later to prove uniqueness.
\begin{proposition} \label{prop:infty}
Fix $d=3$, $\NNN=1$ and $\alpha \in (0, 1+\sqrt{2}]$.
 There exists a trajectory $(V^\infty, Q^\infty)$ such that
 \begin{equation} \label{tr:infty}
\lim_{\xi \to 0^+}  (V^\infty, Q^\infty)(\xi) = (V_\infty, +\infty).
 \end{equation}
As functions of $\xi$, we have
 \begin{subequations} \label{lim:vinfty}
 \begin{align}
V^\infty(\xi) &= V_\infty + O\left(\xi^{\frac{2(2+3(\gamma-1))}{2+3\gamma}}\right), \\
Q^\infty(\xi) &= \overline q_0 \xi^{-\frac{2+3(\gamma-1)}{2+3 \gamma}} + O\left(\xi^{\frac{2+3(\gamma-1)}{2+3 \gamma}}\right).
 \end{align}
 \end{subequations}
 Moreover, the solution can be continued up to some $0 < \xi_2 \le +\infty$\footnote{It is easy to show that $\xi_2<+\infty$, but here we do not need such information.} such that
 \begin{equation} \label{lim:xi2}
\lim_{\xi \to \xi_2^-}  (V^\infty, Q^\infty)(\xi) = (V_2, Q_2).
\end{equation}
Moreover, we have the following monotonicity properties
\begin{equation} \label{mono:Q:infty}
\pp_\xi V^\infty(\xi) \ge 0,\qquad \pp_\xi Q^\infty(\xi)<0,\qquad
\pp_\xi\left(\tfrac{\crr-V^\infty(\xi)}{Q^\infty(\xi)}\right)>0.
\end{equation}

 \end{proposition}
 \begin{proof}[Proof of Proposition \ref{prop:infty}.]
We introduce the function $\Gamma_1 \colon [Q_2, +\infty) \to (-\infty,V_2]$ (where $f_Q$ is defined in \eqref{eq:Q:lvset})
\begin{equation} \label{eq:gamma1:def}
\Gamma_1(Q):= \left(f_Q|_{(-\infty,V_2]}\right)^{-1}(Q).
\end{equation}
To prove that $\Gamma_1$ is well defined we show that $f_Q$ is monotone for $V<V_2$. We recall that from \eqref{eq:Q:lvset} we have
\begin{equation*}
f_Q(V)^2 = \tfrac{1+2\alpha}{\alpha^2} \tfrac{(\crr-V)(V_-^Q - V)(V_+^Q-V)}{V_*^Q-V},
\end{equation*}
where, from \eqref{ineq:series}, we have $ V_2<V_-^Q<V_*^Q <V_+^Q$.  For $V<V_2$, we compute
\begin{equation} \label{monotone:pr}
\partial_V \log f_Q(V)^2
=-\tfrac{1}{\crr - V} -\tfrac{1}{V_-^Q - V} -\tfrac{1}{V_+^Q - V}+\tfrac{1}{V_*^Q - V} < 0.
\end{equation}
Moreover, $f_Q(V_2)=Q_2$ and $f_Q(V)\to+\infty$ as $V\to-\infty$, so this strict monotonicity gives a well-defined inverse branch on $[Q_2,+\infty)$.
We define the region
\begin{equation} \label{region:inv:V}
\mathcal{R} := \{(V,Q): Q\ge Q_2,\ \Gamma_1(Q) \le V \le V_2 \}.
\end{equation}
We will show that the only way to escape the region $\mathcal{R}$ is through the point $(V_2, Q_2)$.
 To prove this, it is enough to observe the following two inequalities for $Q>Q_2$ and $V=\Gamma_1(Q)$
\begin{align*}
P_v[\Gamma_1(Q), Q] - \Gamma_1'(Q) P_q[\Gamma_1(Q), Q] &= P_v[\Gamma_1(Q), Q] = P_v[V, f_Q(V)]  \\
& =- (V-\crr)^2 F_T(V) \tfrac{(1+3\alpha)V-1}{V-1+3\alpha V_\infty} >0, \\
P_v[V_2, Q] &=3\alpha^2 (V_2-\crr)(V_2-V_\infty)(Q^2-Q_2^2) \le 0,
\end{align*}
where we used \eqref{bound:V2}, \eqref{ineq:series}, and that $V_2$ is the smallest root of $F_T(V)$.

Since in $\mathcal R$, $P_q$ is nonpositive and $\Delta$ is nonnegative, this implies that if $Q^\infty(\xi)>Q_2$, then
\begin{equation*}
\Gamma_1(Q^\infty(\xi))< V^\infty(\xi)<V_2, \quad \pp_\xi Q^\infty(\xi)<0.
\end{equation*}
 Since $Q^\infty$ is monotone in $\xi$, and since there are no stationary points for the ODE in the interior of $\mathcal R$ from the discussion in Subsection~\ref{subsec:general:phase:portrait}, the Poincaré--Bendixson theorem implies that for some $\xi_2 \le +\infty$, $(V^\infty, Q^\infty)$ approaches $P_2$ as $\xi \to \xi_2^-$. Combining this observation with Lemma~\ref{lemma:power:R:0} proves \eqref{tr:infty}--\eqref{lim:xi2}.

 To prove $\pp_\xi V^\infty(\xi) \ge 0$, we start by recording the sign of $P_v$ along the trajectory. In the case
$0<\alpha<1+\sqrt2$, the proof of \eqref{bound:V2} gives $V_\infty<V_2$. By
\eqref{P:exp} and the identity
$\frac{\cbb}{\gamma}=\crr-1+3\alpha V_\infty$, the equation $P_v[V,Q]=0$ in the
strip $V_\infty<V\le V_2$ is the graph
\begin{equation}\label{eq:Pv:zero:branch}
Q=f_V(V):=\tfrac1\alpha\sqrt{\tfrac{-V(V-1)(\crr-V)}{3(V-V_\infty)}}.
\end{equation}
This graph connects $P_6$ to $P_2$:  $f_V(V_2)=Q_2$ follows from
\eqref{triple:eq} and \eqref{Q2:def}. Using
$V_\infty<V<V_2<V_0<1$ from \eqref{bound:V2}, a computation similar to \eqref{monotone:pr} shows that $f_V'(V) <0$ for $V \in (V_\infty, V_2)$.

We now observe that if $\xi$ is small enough we have $P_v[V^\infty(\xi), Q^\infty(\xi)] >0$. Indeed, from \eqref{summ:rec:0} with
$n=1$, using $F_1(1)=0$ and
$F_2(1)=-V_\infty(V_\infty-1)(V_\infty-\crr)$, we get
\begin{equation} \label{over:v1}
\overline v_1
=\tfrac{V_\infty(V_\infty-1)(V_\infty-\crr)}{\alpha^2 \overline q_0^2 (3+2\mu)}.
\end{equation}
On the other hand, if $V=\Gamma_V(Q)$ denotes
the inverse of the graph $Q=f_V(V)$ from \eqref{eq:Pv:zero:branch}, the inverse function theorem yields
\[
\Gamma_V(Q)=V_\infty+
\tfrac{V_\infty(V_\infty-1)(V_\infty-\crr)}{3\alpha^2Q^2}
+O(Q^{-4})
\qquad\text{as }Q\to+\infty.
\]
Since $3+2\mu>3$, \eqref{over:v1} shows that, for $\xi>0$ sufficiently small, we have $V^\infty(\xi) < \Gamma_V(Q^\infty(\xi))$. That is equivalent to $f_V(V^\infty(\xi))> Q^\infty(\xi)$, and since $P_v[V,Q]= 3\alpha^2(V-\crr)(V-V_\infty) (Q^2 - f_V(V)^2) $, we have $P_v[V^\infty(\xi), Q^\infty(\xi)] >0$.

We now show that $(V^\infty, Q^\infty)$ cannot cross the graph $V=\Gamma_V(Q)$.
We already observed, from the definition of $\mathcal{R}$ in \eqref{region:inv:V}, that $P_q\le0$ in $\mathcal R$. Hence, on $V=\Gamma_V(Q)$,
\[
\pp_\psi(V-\Gamma_V(Q))=P_v[V,Q]-\Gamma_V'(Q)P_q[V,Q]
=-\Gamma_V'(Q)P_q[V,Q]\le0.
\]
Thus a trajectory starting on the $P_v>0$ side cannot cross to the $P_v<0$
side. Therefore
\[
P_v[V^\infty(\xi),Q^\infty(\xi)]>0
\]
for $0<\xi<\xi_2$. Since $\Delta>0$ in the interior of $\mathcal R$, from
\eqref{main:ODE} we
obtain $\partial_\xi V^\infty>0$ for $0<\xi<\xi_2$.

We also record the monotonicity of the Mach-type ratio along this branch. Set
\[
 \mathcal{M}:=\frac{\crr-V}{Q},
\]
we wish to prove
\begin{equation*}
\pp_\xi \mathcal{M}(\xi)>0.
\end{equation*}
Using \eqref{main:ODE}, we have
\[
\xi\partial_\xi \mathcal{M}
=-\tfrac{P_v[V,Q]+\frac{\crr-V}{Q}P_q[V,Q]}{Q\Delta[V,Q]}.
\]
Since $\Delta>0$ on the interior of $\mathcal R$, it is enough to prove
\begin{equation} \label{ineq:mach:n}
P_v[V,Q]+\tfrac{\crr-V}{Q}P_q[V,Q]<0
\end{equation}
along $(V^\infty,Q^\infty)$. A direct computation from \eqref{P:exp}, using
$\frac{\cbb}{\gamma}=\crr-1+3\alpha V_\infty$, gives
\[
P_v[V,Q]+\tfrac{\crr-V}{Q}P_q[V,Q]=(\crr-V)\left((\crr-V)\mathcal C(V)-\alpha^2Q^2\mathcal D(V)\right),
\]
where
\[
\mathcal C(V):=(\crr-V)(1-(1+3\alpha)V)-(1+\alpha)V(V-1),
\qquad
\mathcal D(V):=1+2V-3(1+\alpha)V_\infty.
\]
For $V_\infty<V<V_2$, the bounds \eqref{bound:V2} imply
$\mathcal C(V)>0$ and $\mathcal D(V)>0$. Hence \eqref{ineq:mach:n} is equivalent to
\[
\mathcal{M}<\mathcal M^*(V):=\alpha\sqrt{\tfrac{(\crr-V)\mathcal D(V)}{\mathcal C(V)}}.
\]
We claim that $\mathcal M^*$ is strictly increasing on $(V_\infty,V_2)$. Set
\[
R(V):=\tfrac{(\crr-V)\mathcal D(V)}{\mathcal C(V)}.
\]
Since $R(V)>0$, the sign of $(\mathcal M^*)'(V)$ is the sign of $R'(V)$. A direct computation gives
\begin{align*}
\pp_V \left(\mathcal{C}(V)^2 \pp_V \tfrac{(\crr-V)\mathcal D(V)}{\mathcal C(V)}\right)
&=4(1+\alpha)(\crr-3\alpha V_\infty)V
-4\crr(1+\alpha (1-3(1+\alpha)V_\infty)) \\
&\le -12\alpha\crr(1+\alpha)(V_0-V_\infty)<0
 \end{align*}
for $V_\infty<V<V_2$, where we used $0 \le V_\infty<V<V_2<V_0$, and $\crr-3\alpha V_\infty>0$. Moreover,
using \eqref{triple:eq} and \eqref{Q2:def},
\[
\mathcal{C}(V_2)^2 \pp_V \left( \tfrac{(\crr-V)\mathcal D(V)}{\mathcal C(V)} \right)|_{V=V_2}
=(1+\alpha)\mathcal C(V_2)(3\crr+3V_\infty-1-4V_2)>0.
\]
Here $\mathcal C(V_2)>0$, and the last factor is positive by \eqref{def:V2}.
It follows that for $V \in (V_\infty,V_2)$ we have
 \begin{equation*}
\mathcal{C}(V)^2 \pp_V \left( \tfrac{(\crr-V)\mathcal D(V)}{\mathcal C(V)} \right)>0.
\end{equation*}
As $\xi\to0^+$, we have $\mathcal M(\xi)\to0 < \mathcal M^*(V_\infty)$. If $\mathcal{M}(\xi)-\mathcal{M}^*(V^\infty(\xi))$ vanished for the first time along the branch, then $\pp_\xi \mathcal{M} =0$, while
\[
\pp_\xi \mathcal M^*(V^\infty)>0,
\]
a contradiction. Therefore $\mathcal M<\mathcal M^*(V^\infty)$ throughout the branch, and hence
\[
\partial_\xi\left(\tfrac{\crr-V^\infty(\xi)}{Q^\infty(\xi)}\right)>0.
\]
At the endpoint $\alpha=1+\sqrt2$, we have $V_\infty=V_2=0$ and $P_v[0,Q]=0$, so $V^\infty\equiv0$ and the same monotonicity follows from $\partial_\xi Q^\infty<0$.
This concludes the proof of \eqref{mono:Q:infty}.
\end{proof}

\subsection{The Rankine--Hugoniot jump conditions and Lax entropy conditions}
Assume we have a jump in the phase diagram for $(V,Q)$, and that it occurs at some value $\xiRH>0$. Let
\[
(\VR^+,  \QR^+) := \lim_{\xi \to \xiRH^+} (V(\xi), Q(\xi)), \qquad
(\VR^-,  \QR^-) := \lim_{\xi \to \xiRH^-} (V(\xi), Q(\xi))
\]
be the upstream $(+)$ and downstream $(-)$ values of $(V,Q)$ at the shock.
Whenever $H$ is defined on both sides of the shock, set
\[
H_{\mathrm{RH}}^\pm:=\lim_{\xi \to \xiRH^\pm}H(\xi).
\]
\begin{itemize}[leftmargin=2em]
\item At any time $t>0$, the shock location is given by $\s(t) := \xiRH t^{\crr}$. Thus, the speed of the shock is
\begin{equation*}
\ds(t) := \crr \xiRH t^{\crr-1}.
\end{equation*}
\item The outward normal is $e_r:=x/|x|$, pointing from downstream (the $-$ side) to upstream (the $+$ side). Throughout this section $f|_\s^\pm(t) := f(\s(t)^\pm,t)$. The jump of a function across the shock is $\jump{f}(t) := f|_\s^-(t) - f|_\s^+(t)$.

\item Velocity on the shock is given by
\[
u|_\s^\pm(t):= u(\s(t)^\pm,t) \cdot e_r = t^{\crr-1} \xiRH \VR^\pm
\]
and the relative velocity is
\[
(u|_\s^\pm - \ds)(t):= u(\s(t)^\pm,t) \cdot e_r - \dot{\s}(t)
= t^{\crr-1} \xiRH (\VR^\pm - \crr)
\]
\item The rescaled sound speed on the shock is given by
    \[
    \sigma|_\s^\pm(t):= \sigma(\s(t)^\pm,t) = t^{\crr-1} \xiRH \QR^\pm
    \]

\item The square root of the pseudo-entropy on the shock is
    \[
    b|_\s^\pm(t):= b(\s(t)^\pm,t) = t^{\cbb} \xiRH H_{\mathrm{RH}}^\pm
    \]

\item Density on the shock is
    \[
    \rho|_\s^\pm(t):= \rho(\s(t)^\pm,t) = t^{\frac{\crr-1-\cbb}{\alpha}} (\frac{\alpha \QR^\pm}{H_{\mathrm{RH}}^\pm})^{\frac{1}{\alpha}}
    =: t^{\frac{\crr-1-\cbb}{\alpha}} \Theta_{\mathrm{RH}}^\pm
    \]

\item The pressure on the shock is
    \[
    p|_\s^\pm(t):= p(\s(t)^\pm,t) = \tfrac{1}{\gamma}t^{\frac{\gamma(\crr-1)-\cbb}{\alpha}} \Theta_{\mathrm{RH}}^\pm \xiRH^2 (\alpha \QR^\pm)^2
    \]
\end{itemize}

For a $3$-shock with shock curve $\s$ and upstream state $(\usp,\rsp,\ssp)$ satisfying the upstream Lax inequality, the Rankine--Hugoniot relations give the downstream state by
\begin{equation} \label{RH:re}
\begin{aligned}
\usm &= \ds+\tfrac{  \alpha^2 (\ssp)^2 + \alpha (\usp-\ds)^2 }{(1+\alpha)\,(\usp-\ds)}, \\
\rsm &= \tfrac{ (1+\alpha) (\usp-\ds)^2 \rsp}{\alpha^2 (\ssp)^2 + \alpha (\usp-\ds)^2 }, \\
\ssm &= \tfrac{\sqrt{( -\alpha^3 (\ssp)^2 + \gamma (\usp-\ds)^2)( \alpha^2 (\ssp)^2 + \alpha (\usp-\ds)^2 ) }}{\alpha (1+\alpha) (\ds-\usp)}.
\end{aligned}
\end{equation}
We also recall that the Lax inequalities are
\begin{equation*}
\usp + \alpha \ssp < \ds < \usm + \alpha \ssm, \quad \usm< \ds.
\end{equation*}
Translated into the self-similar ansatz, these inequalities become
 \begin{equation} \label{lax:ineq}
  \VR^+ + \alpha \QR^+ < \crr <  \VR^- + \alpha \QR^-, \quad \VR^- <\crr. 
 \end{equation}
  Using the itemized list, we are able to rewrite~\eqref{RH:re} just in terms of $(\VR^\pm, \QR^\pm)$:
\begin{subequations}\label{RH:jump}
\begin{align} 
 \VR^- &= \crr + \tfrac{\alpha^2 (\QR^+)^2 + \alpha ( \VR^+-\crr)^2 }{(1+\alpha)( \VR^+-\crr)},
 \\
 \QR^- &=\tfrac{\sqrt{(-\alpha^3 (\QR^+)^2 + \gamma (\VR^+-\crr)^2)(\alpha^2 (\QR^+)^2 + \alpha (\VR^+-\crr)^2)}}{\alpha(1+\alpha)( \crr-\VR^+)}.
 \end{align}
 While not directly relevant in the analysis of the phase portrait, we also record the jump conditions for $H$. For $\QR^+>0$, from the density identity in \eqref{RH:re}, we have
\begin{equation} \label{RH:jump:H}
 H_{\mathrm{RH}}^- =
 H_{\mathrm{RH}}^+\tfrac{\QR^-}{\QR^+}
 \left(\tfrac{\alpha^2(\QR^+)^2+\alpha(\VR^+-\crr)^2}{(1+\alpha)(\VR^+-\crr)^2}\right)^\alpha.
\end{equation}
\end{subequations}
 
  We introduce the map $(\VR^-, \QR^-) := \mathrm{RH}(\VR^+, \QR^+)$. We observe $\mathrm{RH}(V,Q)$ is well defined as long as
  \begin{itemize}[leftmargin=2em]
  \item $Q\ge0$ (nonnegativity of sound speed),
  \item $V + \alpha Q \le \crr$ and $V<\crr$ (Lax condition),
 \end{itemize}
 and that it maps the {\textit{subsonic} region $\{ Q \in [0, +\infty), \; V+\alpha Q \le \crr,\; V<\crr \}$} into the {\textit{supersonic} region $\{ Q \in [0, +\infty), \; V+\alpha Q \ge \crr, \; V< \crr \}$}.
We now look at the Rankine--Hugoniot jump locus associated with the trajectory $(V,Q)$ from Proposition~\ref{prop:sub:sonic}. This is the \textit{Hugoniot locus}
\begin{equation}
 \mathsf{G}_{\mathrm{RH}}(\xi)= (\mathsf{G}_{\mathrm{RH}, v}, \mathsf{G}_{\mathrm{RH}, q}) := \mathrm{RH}(  V(\xi),  Q(\xi) ), \label{hugoniot}
\end{equation}
for $\xi \in (\xi_s, +\infty)$. Using \eqref{eq:xis:expl}, we can extend $\mathsf{G}_{\mathrm{RH}}(\xi)$ continuously to $\xi=\xi_s$, by defining $\mathsf{G}_{\mathrm{RH}}(\xi_s)=(V_s, Q_s)$.
Then, from~\eqref{RH:jump} we have
\begin{equation} \label{hugoniot:limit}
\lim_{\xi \to +\infty} \mathrm{RH}( V(\xi), Q(\xi) ) = \mathrm{RH}(0,0) = \left(\tfrac{\crr}{1+\alpha}, \tfrac{\crr \sqrt{\gamma \alpha}}{\alpha(1+\alpha)} \right), \quad  \mathsf{G}_{\mathrm{RH}}(\xi_s) = (  V_s,  Q_s ).
 \end{equation}
 The formula for $\mathrm{RH}(0,0)$ follows directly from \eqref{RH:jump}. The endpoint extension is consistent with \eqref{RH:jump}, since $V_s+\alpha Q_s=\crr$ implies $\mathrm{RH}(V_s,Q_s)=(V_s,Q_s)$.
 We observe that
 \begin{equation} \label{PRH:loc}
 V_2 < \tfrac{\crr}{1+\alpha}.
 \end{equation}
 Indeed, by \eqref{bound:V2}, $V_2<V_0=\frac{1}{1+3\alpha}$. Since $\alpha>0$ and $\crr>1$, we have
 $\frac{1}{1+3\alpha}<\frac{\crr}{1+\alpha}$, which proves the claim.

 We collect these facts about the Rankine--Hugoniot jump conditions and a simple consequence into the next Lemma.
\begin{figure}[ht]
\centering
\includegraphics[width=0.7\textwidth]{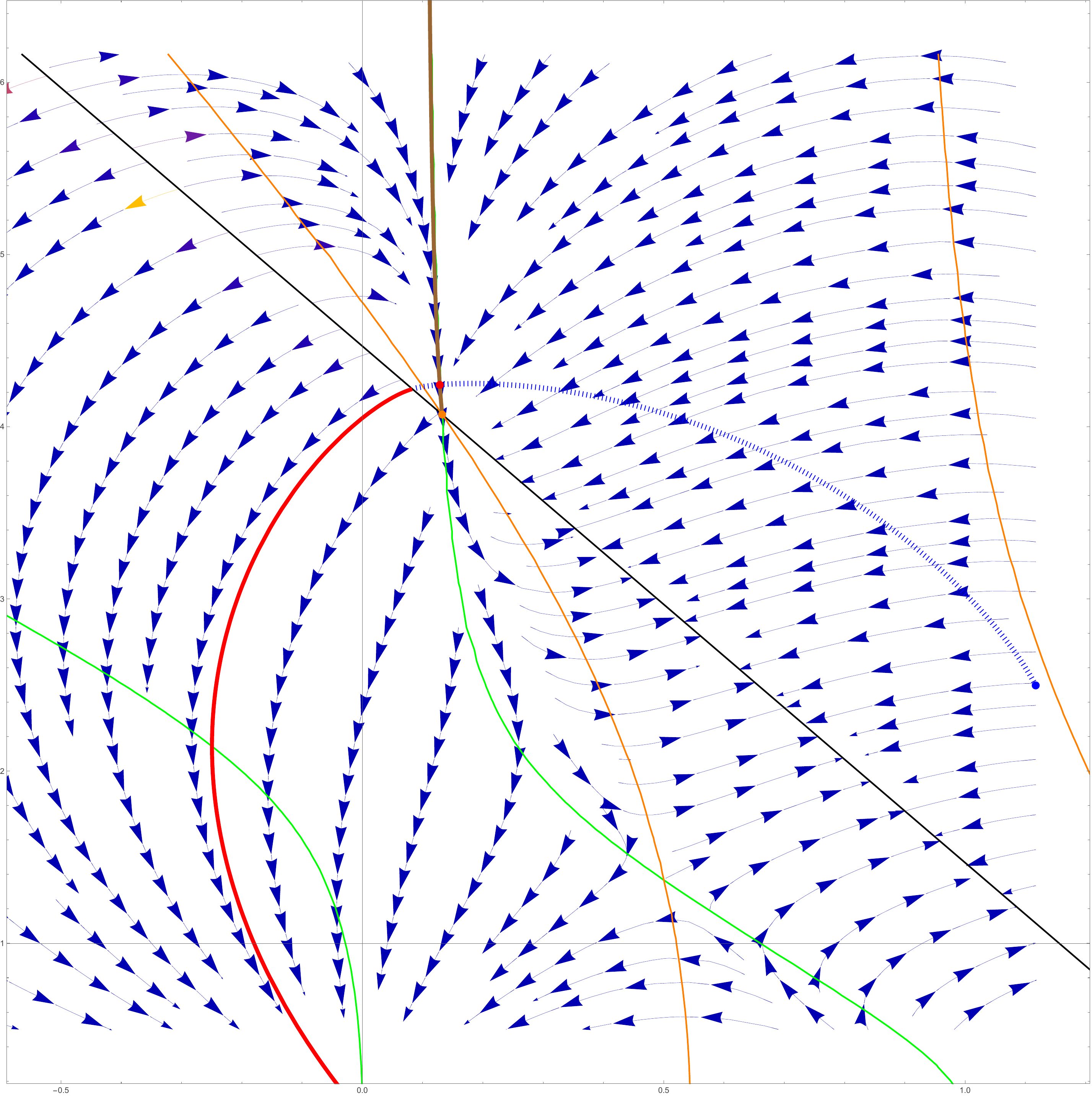}
\caption{\textbf{The Rankine--Hugoniot jump locus for $\gamma=\frac{5}{3}$, $d=3$ and $\NNN=1$.} The vertical axis represents $Q$, while the horizontal axis represents $V$. The orange point represents the triple point $P_2$. The green curve represents $\{ P_{v}[ V,  Q]=0\}$, the orange one $\{ P_{q}[ V,  Q]=0\}$, and the black one $\{ \Delta[V,Q]=0\}$. The brown curve is the downstream trajectory $(V^\infty, Q^\infty)$ from Proposition~\ref{prop:infty}, extended until it hits the triple point $(V_2, Q_2)$. The red curve represents the upstream trajectory, while the dotted blue line represents the Hugoniot locus $\mathsf{G}_{\mathrm{RH}}$. The blue point represents $\mathsf{G}_{\mathrm{RH}}(+\infty) = \mathrm{RH}(0,0)$.}
\label{fig:RH:locus}
\end{figure}

\begin{lemma}\label{lemma:RH:properties}
Let $Q\ge0$, $V<\crr$, and $V+\alpha Q\le\crr$. Define
$\mathrm{RH}(V,Q)$ by the Rankine--Hugoniot jump conditions \eqref{RH:jump}. Then the following properties hold:
\begin{enumerate}[leftmargin=2em]
\item $\mathrm{RH}(V,Q)_q\ge0$.
\item $\mathrm{RH}(V,Q)_v<\crr$ and
\[
\mathrm{RH}(V,Q)_v+\alpha \mathrm{RH}(V,Q)_q\ge\crr.
\]
If $V+\alpha Q<\crr$, then the last inequality is strict.
\item Set $\frac{\crr-V}{Q}=\mathcal M$; if $Q>0$ then
\begin{equation}\label{eq:RH:Mach:map}
\tfrac{\crr-\mathrm{RH}(V,Q)_v}{\mathrm{RH}(V,Q)_q}
=\tfrac{\alpha^2(\mathcal M^2+\alpha)}
{\sqrt{(\gamma\mathcal M^2-\alpha^3)(\alpha^2+\alpha\mathcal M^2)}}.
\end{equation}
In particular, the strict upstream Lax inequality $V+\alpha Q<\crr$ implies the strict downstream Lax inequality $\mathrm{RH} (V,Q)_v+\alpha \mathrm{RH}(V,Q)_q>\crr$.
\item If $V+\alpha Q=\crr$, then $\mathrm{RH}(V,Q)=(V,Q)$.
\end{enumerate}
\end{lemma}

\begin{proof}
The nonnegativity of the $q$-component, and the inequality $\mathrm{RH}(V,Q)_v<\crr$ follow directly from \eqref{RH:jump}. Indeed, $\crr-V\ge\alpha Q$ and $\gamma=1+2\alpha$ give
\[
-\alpha^3Q^2+\gamma(\crr-V)^2
\ge \alpha^2(\gamma-\alpha)Q^2
=\alpha^2(1+\alpha)Q^2\ge0,
\]
so the square root in \eqref{RH:jump} is real and nonnegative. Moreover,
\[
\crr-\mathrm{RH}(V,Q)_v
=\tfrac{\alpha^2Q^2+\alpha(\crr-V)^2}{(1+\alpha)(\crr-V)}>0.
\]

Assume now that $Q>0$. Substituting $\crr-V=\mathcal M Q$ in \eqref{RH:jump} gives \eqref{eq:RH:Mach:map}. Since $V+\alpha Q\le\crr$, we have $\mathcal M\ge\alpha$, and from
\eqref{eq:RH:Mach:map},
\[
\left(\tfrac{\crr-\mathrm{RH}(V,Q)_v}{\mathrm{RH}(V,Q)_q}\right)^2
=\tfrac{\alpha^3(\mathcal M^2+\alpha)}{\gamma\mathcal M^2-\alpha^3}
\le \alpha^2,
\]
with equality exactly when $\mathcal M=\alpha$. This proves $\mathrm{RH}(V,Q)_v+\alpha \mathrm{RH}(V,Q)_q\ge\crr$, and shows that the inequality is strict if $Q>0$ and $V+\alpha Q<\crr$. If $Q=0$, then \eqref{RH:jump} gives
\[
\tfrac{\crr-\mathrm{RH}(V,0)_v}{\mathrm{RH}(V,0)_q}
=\tfrac{\alpha^{3/2}}{\sqrt{\gamma}}<\alpha,
\]
so the downstream Lax inequality is again strict. Finally, substituting $\crr-V=\alpha Q$ in \eqref{RH:jump} gives $\mathrm{RH}(V,Q)=(V,Q)$.
\end{proof}
\subsection{Uniqueness of the supersonic branch}
The goal of this subsection is twofold. First, we show that in the supersonic region $\{Q>0, \, V+\alpha Q>\crr,\, V<\crr\}$ there is a unique physically relevant trajectory, $(V^\infty,Q^\infty)$ from Proposition~\ref{prop:infty}. Second, we show monotonicity properties of the trajectory $\mathsf{G}_{\mathrm{RH}}$, that we will use in \S~\ref{subsec:proof:th} to prove that $\mathsf{G}_{\mathrm{RH}}$ intersects uniquely $(V^\infty,Q^\infty)$.

We start by proving that any trajectory in the supersonic region either hits the sonic line $\{V+\alpha Q=\crr\}$ as $\xi$ decreases, or else is a rescaling of the trajectory $(V^\infty,Q^\infty)$ constructed in Proposition~\ref{prop:infty}.
\begin{lemma}\label{lemma:supersonic:uniqueness}
Fix $d=3$, $\NNN=1$, and $\alpha\in(0,1+\sqrt{2}]$.
Set
\[
\mathcal{S}:=\{(V,Q): Q>0,\ V+\alpha Q>\crr,\ V<\crr\}.
\]
Let $(\widetilde V,\widetilde Q)$ be a trajectory of \eqref{main:ODE} such that,
for some $\xi_{\mathrm{in}}>0$, $(\widetilde V(\xi_{\mathrm{in}}),\widetilde Q(\xi_{\mathrm{in}}))\in \mathcal{S}$. Continue this trajectory backward in $\xi$. Then, either there exists $\xi_0\ge0$ such that $(\widetilde V,\widetilde Q)$ hits the sonic line as $\xi\to\xi_0^+$, or $(\widetilde V,\widetilde Q)$ is a rescaling of the trajectory $(V^\infty,Q^\infty)$ from Proposition~\ref{prop:infty}.
\end{lemma}

\begin{proof}
We use the desingularized system \eqref{eq:tran:ODE},
\[
\pp_\psi V=P_v[V,Q],\qquad \pp_\psi Q=P_q[V,Q].
\]
Since $\Delta>0$ in $\mathcal{S}$, the variables $\psi$ and $\xi$ have the same orientation in this region. Thus hitting the sonic line backward in $\xi$ is equivalent to hitting it backward in the phase portrait
\eqref{eq:tran:ODE}.
The proof is divided in three steps:
\begin{itemize}[leftmargin=2em]
    \item first, we prove that for any solution $(\widetilde V,\widetilde Q)$ that does not hit the sonic line, the component $\widetilde V$ is bounded from below;
    \item then, we show that if $\widetilde Q$ is unbounded it must coincide,
    up to rescaling, with $(V^\infty,Q^\infty)$;
    \item we conclude by applying Poincaré--Bendixson to trajectories contained in a compact set.
\end{itemize}
\emph{Step 1: exclusion of the escape $\widetilde V\to-\infty$.}
For $\kappa>0$ and $t>0$, the identity \eqref{eq:ell} gives
\[
P_v[-\kappa t,t]+\kappa P_q[-\kappa t,t]\ge
( 2\alpha^2\kappa^2 +  2 \alpha \kappa^4 )t^4>0.
\]
Equivalently, the line $V=-\kappa Q$ is a backward barrier: on this line,
\[
\pp_\psi(V+\kappa Q)=P_v+\kappa P_q\ge0.
\]
Thus, in backward time, a trajectory starting on this line can only enter the region $V+\kappa Q\le0$.

If, before hitting the sonic line, there is a time
$\xi'\le\xi_{\mathrm{in}}$ such that $\widetilde V(\xi')<0$, set
\[
\kappa:=-\tfrac{\widetilde V(\xi')}{\widetilde Q(\xi')}>0.
\]
Then
\[
\tfrac{\widetilde V}{\widetilde Q}\le -\kappa
\]
for all earlier times. Moreover, while the trajectory remains in $\mathcal{S}$ and $\widetilde V<0$,
\[
\widetilde Q>\tfrac{\crr-\widetilde V}{\alpha}>\tfrac{\crr}{\alpha}.
\]
At any such point, we have
\[
\pp_\psi\left(\tfrac{\widetilde V}{\widetilde Q}\right)
=\tfrac{P_v[\widetilde V,\widetilde Q]
-\tfrac{\widetilde V}{\widetilde Q}P_q[\widetilde V,\widetilde Q]}
{\widetilde Q}
\ge 2\alpha^2\left(-\tfrac{\widetilde V}{\widetilde Q}\right)^2
\widetilde Q^3
\ge 2\alpha^2\kappa^2\widetilde Q^3.
\]
Therefore, as $\psi$ decreases,
\[
\tfrac{\widetilde V}{\widetilde Q}(\psi_1)
\le \tfrac{\widetilde V}{\widetilde Q}(\psi_0)
-2\alpha^2\kappa^2\left(\tfrac{\crr}{\alpha}\right)^3(\psi_0-\psi_1),
\qquad \psi_1<\psi_0.
\]
It follows that $\tfrac{\widetilde V}{\widetilde Q}$ reaches $-\alpha$ in finite backward time unless the trajectory has already left $\mathcal{S}$. Since in $\mathcal{S}$ we have
\[
\tfrac{\crr-\widetilde V}{\widetilde Q}
= -\tfrac{\widetilde V}{\widetilde Q}+\tfrac{\crr}{\widetilde Q} < \alpha,
\]
it follows that $( \tilde V, \tilde Q)$ must intersect the sonic line $V+\alpha Q=\crr$. In particular, we conclude that if $(\tilde V, \tilde Q)$ is contained for all times $0<\xi \le \xi_{\mathrm{in}}$ in $\mathcal{S}$ we must have $\tilde V(\xi)>0$.

\emph{Step 2: the unbounded $Q$ alternative.}
It follows from Step 1 that any non-compact backward trajectory avoiding the sonic line has $\widetilde V$ contained in a compact subinterval of $[0,\crr]$. We now prove that, if $\widetilde Q$ is unbounded in backward time, then $(\widetilde V,\widetilde Q)$ converges to $P_6$.

Let $\Gamma_2(Q)$ denote the inverse of the branch $\Gamma_2^Q$ of $P_q=0$ defined in \eqref{eq:Q:lvset}; thus $\Gamma_2(Q)\in(V_*^Q,\crr)$ and $P_q[\Gamma_2(Q),Q]=0$. Moreover, $\Gamma_2(Q)\to V_*^Q$ and $\Gamma_2(Q)-V_*^Q=O(Q^{-2})$ as $Q\to+\infty$. On this branch,
\[
P_v[\Gamma_2(Q),Q]=-\tfrac{(1+3\alpha)\Gamma_2(Q)-1}{\Gamma_2(Q)-1+3\alpha V_\infty}(\Gamma_2(Q)-\crr)^2F_T(\Gamma_2(Q))<0,
\]
because $V_2<V_*^Q<\Gamma_2(Q)<\crr<V_3$. Hence, on $V=\Gamma_2(Q)$,
\[
\pp_\psi\bigl(V-\Gamma_2(Q)\bigr)
=P_v-\Gamma_2'(Q)P_q=P_v<0.
\]
Consequently, if for some time $\widetilde V>\Gamma_2(\widetilde Q)$, then the same inequality holds for all earlier times. In that region $P_q[\widetilde V,\widetilde Q]>0$, so $\widetilde Q$ decreases as $\xi$ decreases. This is incompatible with $\widetilde Q$ being unbounded.

We are left with the case where $\widetilde V(\xi) \le\Gamma_2(\widetilde Q(\xi))$ for $\xi<\xi_{\mathrm{in}}$. We may also assume that $\widetilde Q$ is large enough, so that in this region, from the definition of $\Gamma_2$, we have  $P_q[\widetilde V,\widetilde Q]<0$.
 Thus $\widetilde Q$ is monotone in backward time, and since there are no stationary points in $\mathcal{S}$ (following the discussion in \S~\ref{subsec:general:phase:portrait}), we must have $\tilde Q \to +\infty$. From \eqref{P:exp}, for bounded $\widetilde V$ and large $\widetilde Q$,
\[
\widetilde Q\tfrac{d\widetilde V}{d\widetilde Q}
=\tfrac{3\alpha^2(\widetilde V-\crr)(\widetilde V-V_\infty)\widetilde Q^2+O(1)}
{\alpha^2(\widetilde V-V_*^Q)\widetilde Q^2
-P_{2,Q}[\widetilde V](\widetilde V-\crr)}.
\]
The denominator is negative, and therefore for $\tilde Q$ large the sign of $\widetilde Q\,\frac{d\widetilde V}{d\widetilde Q}$ is the sign of $\widetilde V-V_\infty$. This sign is uniform away from $V_\infty$. Hence, if for some $\varepsilon>0$ and arbitrarily large $\widetilde Q$ we had $\widetilde V\ge V_\infty+\varepsilon$, then $\widetilde V$ would increase as
$\widetilde Q$ increases until it enters the already excluded region
$\widetilde V>\Gamma_2(\widetilde Q)$.
Similarly, $\widetilde V\le V_\infty-\varepsilon$ for arbitrarily large
$\widetilde Q$ would imply that the trajectory enters the region $\widetilde V<0$. Since
$\varepsilon>0$ is arbitrary,
\[
\widetilde Q\to+\infty,\qquad \widetilde V\to V_\infty
\]
in backward time.

We identify this alternative with the branch from Proposition~\ref{prop:infty}.
Since $\widetilde Q\to+\infty$ and $\widetilde V\to V_\infty$ in backward time,
the expansion near $P_6$ gives
\[
\xi\pp_\xi(\widetilde V-V_\infty)
=-3(\widetilde V-V_\infty)+O((\widetilde V-V_\infty)^2+\widetilde Q^{-2}),
\]
and
\[
\xi\pp_\xi(\widetilde Q^{-2})
=2\mu \widetilde Q^{-2}
+O(|\widetilde V-V_\infty|\widetilde Q^{-2}+\widetilde Q^{-4}).
\]
As in Lemma~\ref{lemma:power:R:0}, it follows that
$\widetilde V(\xi)=V_\infty+O(\xi^{2\mu})$ and that
\[
\overline q_0:=\lim_{\xi\to0^+}\xi^\mu \widetilde Q(\xi)
\]
exists in $(0,+\infty)$. Comparing with the convergent series solution from Lemma~\ref{lemma:power:R:0}, and applying Gronwall to the difference after multiplying the $Q$ equation by $\xi^\mu$, gives equality with the local branch with leading coefficient $\overline q_0$. Since changing $\xi$ to $\lambda\xi$ changes the leading coefficient by $\overline q_0\mapsto \overline q_0\lambda^{-\mu}$, this curve is a rescaling of $(V^\infty,Q^\infty)$.

\emph{Step 3: exclusion of compact trajectories that avoid the sonic line.}
It remains to rule out the case in which $(\widetilde V,\widetilde Q)$ stays in a compact subset of $\mathcal{S}$ for all backward time $0<\xi<\xi_{\mathrm{in}}$. In such a compact set the desingularized vector field has no stationary points. Hence Poincaré--Bendixson would force the backward limit set to contain a periodic orbit. This is impossible by the Bendixson--Dulac criterion. Indeed, with $B(V,Q)=\bigl(Q(\crr-V)^2\bigr)^{-1}$, a direct computation gives
\[
\pp_V(BP_v)+\pp_Q(BP_q)
=-\tfrac{2V-1}{Q}
+\tfrac{\alpha^2 Q}{(\crr-V)^2}
\left(2V-5\crr+\tfrac{3}{1+3\alpha}\right)<0
\]
throughout $\mathcal{S}$. The last inequality follows from $V<\crr$, $Q>\tfrac{\crr-V}{\alpha}$, and $\crr>1$.

Therefore, if $(\widetilde V,\widetilde Q)$ avoids the sonic line, it is a rescaling of the trajectory $(V^\infty, Q^\infty)$ from Proposition~\ref{prop:infty}. Equivalently, every trajectory in $\mathcal{S}$ that is not a rescaling of $(V^\infty,Q^\infty)$ hits the sonic line backward in $\xi$.
\end{proof}

The next lemma proves monotonicity of the $v$-component of the Hugoniot locus along the outer trajectory. This monotonicity will be used in the proof of Theorem~\ref{th:main:sec} to show that the Hugoniot locus intersects uniquely the trajectory $(V^\infty, Q^\infty)$.
\begin{lemma}
\label{transv:hug}
Consider the Hugoniot locus $\mathsf{G}_{\mathrm{RH}}$ defined in \eqref{hugoniot}. For $\xi \in (\xi_s, +\infty)$ we have
\begin{equation} \label{claim:tr}
\pp_\xi \mathsf{G}_{\mathrm{RH},v}(\xi) \ge 0.
\end{equation}
\end{lemma}
\begin{proof}
By the definition \eqref{hugoniot} of $\mathsf{G}_{\mathrm{RH},v}$, the expression \eqref{RH:jump}, and using \eqref{eq:euler:ODE}, the claim \eqref{claim:tr} is equivalent to
\begin{align} \label{tbp:ineq}
\pp_\xi \mathsf{G}_{\mathrm{RH},v}(\xi)
&=\tfrac{\alpha}{(1+\alpha)\xi\Delta[V,Q]}
\left[
\left(1-\tfrac{\alpha Q^2}{(V-\crr)^2}\right)P_v[V,Q]
+\tfrac{2\alpha Q}{V-\crr}P_q[V,Q]
\right] \notag\\
&=\tfrac{\alpha\,\mathsf{T}[V,Q]}{(1+\alpha)\xi(\crr-V)^2(\crr-V-\alpha Q)(\crr-V+\alpha Q)}\ge0,
\end{align}
where $(V(\xi),Q(\xi))$ is the trajectory from Proposition~\ref{prop:sub:sonic}. A direct computation gives
\begin{align}
\mathsf{T}[V,Q]
&=\bigl((V-\crr)^2-\alpha Q^2\bigr)
\left[-V(V-1)(V-\crr)+\alpha Q^2\left(\crr-\tfrac{\cbb}{\gamma}-1+3\alpha V\right)\right] \label{def:T:T} \\
&\quad+2\alpha Q^2
\left[-(V-\crr)^2(V(1+3\alpha)-1)+\alpha(V-\crr)(V(V-1)+\alpha Q^2)
+\tfrac{\alpha^2\cbb}{\gamma}Q^2\right]. \notag
\end{align}
We first observe that along the trajectory $V<V_2$. If $Q\le \tfrac{Q_2}{2}$, then $V<0$; we treat the case $Q>\tfrac{Q_2}{2}$ below. Indeed, from Lemma~\ref{lemma:escape},
$(V,Q)$ stays in $\mathcal{T}$, so in particular $V<V_2$. Moreover, since
$\mathcal{W}_E(Q)=\rr Q+\tfrac{V_2-Q_2\rr}{Q_2^2}Q^2$, we get
\[
\mathcal{W}_E(\tfrac{Q_2}{2})
=\tfrac14(V_2+\rr Q_2).
\]
We observe that  $V_2+\rr Q_2$ is negative.
We first note the elementary bound
\[
\crr>\tfrac{6+13\alpha}{4(1+3\alpha)}.
\]
Indeed, by the explicit formula for $\crr$ \eqref{cr:formula}, this follows from squaring
\[
\sqrt{1+102\alpha+249\alpha^2}>1+16\alpha,
\]
and using $0<\alpha\le 1+\sqrt{2}<3$.

Now, since $Q_2=\tfrac{\crr-V_2}{\alpha}$, $\rr<0$, and $V_2<V_0$, we get
\[
\begin{aligned}
V_2+\rr Q_2
&=
V_2-\sqrt{\tfrac{6\alpha}{25(1+2\alpha)}}\tfrac{\crr-V_2}{\alpha}  \\
&<
\frac1{1+3\alpha}
-\sqrt{\tfrac{6\alpha}{25(1+2\alpha)}}
\tfrac{\tfrac{6+13\alpha}{4(1+3\alpha)}-\tfrac1{1+3\alpha}}{\alpha} \\
&=
\tfrac1{1+3\alpha}
\left[
1-\sqrt{\tfrac{6\alpha}{25(1+2\alpha)}}\tfrac{2+13\alpha}{4\alpha}
\right].
\end{aligned}
\]
The bracket is negative because
\[
\tfrac{6\alpha}{25(1+2\alpha)}(2+13\alpha)^2-16\alpha^2
=
\tfrac{2\alpha(107\alpha^2-44\alpha+12)}{25(1+2\alpha)}>0,
\]
and $107\alpha^2-44\alpha+12>0$ since its discriminant is negative. Therefore
\[
V_2+\rr Q_2<0.
\]

Since $\mathcal{W}_E(0)=0$ and $\mathcal{W}_E$ is convex, this implies $\mathcal{W}_E(Q)<0$ for $0<Q\le \tfrac{Q_2}{2}$. Therefore, if $Q\le \tfrac{Q_2}{2}$, the lower barrier gives $V<\mathcal{W}_E(Q)<0$.

It remains to prove
\begin{equation} \label{eq:T:pos:region}
\mathsf{T}[V,Q]>0
\end{equation}
in the subsonic region whenever either $V<0$, or $V<V_2$ and
$Q>\tfrac{Q_2}{2}$. Set
\[
Z:=\left(\tfrac{\crr-V}{Q}\right)^2.
\]
Since $V+\alpha Q<\crr$, we have $Z>\alpha^2$. Dividing \eqref{def:T:T} by
$(\crr-V)^3$, multiplying by $Z^2$, and using
$\frac{\cbb}{\gamma}=\crr-1+3\alpha V_\infty$ gives
\begin{align} \label{RHS:mono}
\tfrac{(\crr-V)\mathsf{T}[V,Q]}{Q^4}
&=Z^2V(V-1)\\ 
&\quad+\alpha Z\left((\crr-V)\bigl(2-(2+3\alpha)V-3\alpha V_\infty\bigr)
-\gamma V(V-1)\right)\notag\\
&\quad+\alpha^3(\crr-V)(3\gamma V_\infty-V-2).\notag
\end{align}
The derivative of the right-hand side with respect to $Z$ is
\begin{equation} \label{der:T:pr}
2ZV(V-1)+\alpha\left((\crr-V)\bigl(2-(2+3\alpha)V-3\alpha V_\infty\bigr)
-\gamma V(V-1)\right).
\end{equation}
 From $3\alpha V_\infty <2 \alpha  V_0<2$ and $\gamma=1+2\alpha$ this derivative is positive for $V<0$ and $Z>0$. For $0\le V<V_2$ and $Q>\tfrac{Q_2}{2}$, using  $2V_2<\crr$ from \eqref{V2:crr:2}, we have
\[
\alpha^2<Z<4\alpha^2\left(\tfrac{\crr-V}{\crr-V_2}\right)^2 \le 16 \alpha^2,
\]
and $V(V-1)<0$. Hence the derivative \eqref{der:T:pr} is bounded below by
\begin{align*}
    A_T(V):= \alpha\left[
\crr(2-3\alpha V_\infty)
-\bigl(\crr(2+3\alpha)+1+30\alpha-3\alpha V_\infty\bigr)V
+(1+33\alpha)V^2
\right].
\end{align*}
 We now observe that $A_T(V)$ is positive for $0\le V\le V_2$. At $V=V_2$, we have
\[
A_T(V_2)=\tfrac{\alpha}{2}\left[ \left((3V_\infty-\crr-3 +3\alpha(35V_\infty+31\crr-31)\right)V_2+\crr\left(4-3(1+35\alpha)V_\infty\right) \right]>0,
\]
where the inequality follows from $0\le V_\infty\le V_2<V_0$ from \eqref{pos:vinfty:range} and
\eqref{bound:V2}, $V_2\le \frac{\crr}{2}$ from \eqref{V2:crr:2}, and
$\crr>\frac{3+2\alpha}{2(1+3\alpha)}$ from \eqref{eq:crr:lb}. 
To conclude that $A_T$ is positive on $0 \le V \le V_2$ we observe that $A_T$ is monotone decreasing
\[
A'_T(V)=\alpha(-\crr(2+3\alpha)-1-30\alpha+3\alpha V_\infty+2(1+33\alpha)V)<0,
\]
for $V \in (0, V_2)$, where we used $V_2<V_0$ from \eqref{bound:V2},
$V_\infty<\frac23V_0$ from \eqref{eq:v0:ub}, and $\crr>1$.
 
Hence the right-hand side of \eqref{RHS:mono} is increasing in $Z$ in both cases. To conclude that $\mathsf{T}[V,Q]>0$ it is enough to check the sign at the left endpoint $Z=\alpha^2$, where we have
\[
\tfrac{(\crr-V)\mathsf{T}[V,Q]}{Q^4}
=-\alpha^3(1+\alpha)F_T(V)>0,
\]
because $F_T(V)<0$ for $V<V_2$. Since $\crr-V>0$ and $Q>0$, \eqref{eq:T:pos:region} follows, and then \eqref{tbp:ineq} follows.

\end{proof}

Lastly, we prove that the Mach ratio $\frac{\crr-V}{Q}$ is monotone along the trajectory $(V,Q)$ from Proposition~\ref{prop:sub:sonic}; as a consequence, we obtain the monotonicity of $\tfrac{\crr-\mathsf{G}_{\mathrm{RH},v}}{\mathsf{G}_{\mathrm{RH},q}}$
\begin{lemma}\label{lemma:supersonic:ratio}
Let $Q>0$, $V<V_2$, and $V+\alpha Q<\crr$. Then
\begin{equation}\label{eq:supersonic:ratio:sign}
P_v[V,Q]+\tfrac{\crr-V}{Q}P_q[V,Q]>0.
\end{equation}
Consequently, for the trajectory $(V,Q)$ from Proposition~\ref{prop:sub:sonic},
\begin{equation}\label{eq:supersonic:ratio:mono}
\xi\partial_\xi\left(\tfrac{\crr-V(\xi)}{Q(\xi)}\right)>0,
\qquad \xi\in(\xi_s,+\infty).
\end{equation}
Moreover, for the associated Hugoniot locus $\mathsf{G}_{\mathrm{RH}}$ defined in \eqref{hugoniot},
\begin{equation}\label{eq:supersonic:RH:ratio:mono}
\xi\partial_\xi\left(
\tfrac{\crr-\mathsf{G}_{\mathrm{RH},v}(\xi)}
{\mathsf{G}_{\mathrm{RH},q}(\xi)}
\right)<0,
\qquad \xi\in(\xi_s,+\infty).
\end{equation}
\end{lemma}

\begin{proof}
Since $V<V_2<\crr$ by \eqref{bound:V2}, and since $V+\alpha Q<\crr$, we have
$\crr-V>0$ and
\begin{equation}\label{eq:supersonic:ratio:theta}
0<\tfrac{\alpha^2Q^2}{(\crr-V)^2}<1.
\end{equation}
Using the formulas for $P_v$ and $P_q$ in \eqref{P:exp}, together with the identity $\frac{\cbb}{\gamma}=\crr-1+3\alpha V_\infty$ following from the definition of $V_\infty$ in Subsection~\ref{subsec:general:phase:portrait}, we compute
\begin{align}
&P_v[V,Q]+\tfrac{\crr-V}{Q}P_q[V,Q]\notag\\
&\quad=(\crr-V)^2\left((\crr-V)(1-(1+3\alpha)V)-(1+\alpha)V(V-1)\right)\notag\\
&\qquad+\alpha^2(\crr-V)Q^2\left(3(1+\alpha)V_\infty-2V-1\right).
\label{eq:supersonic:ratio:identity}
\end{align}
After dividing \eqref{eq:supersonic:ratio:identity} by the positive quantity $(\crr-V)^3$, the right-hand side becomes a convex combination, with coefficient $\frac{\alpha^2Q^2}{(\crr-V)^2}$, of
\begin{equation}\label{eq:supersonic:first:term}
\tfrac{(\crr-V)(1-(1+3\alpha)V)-(1+\alpha)V(V-1)}{\crr-V}
\end{equation}
and
\begin{equation}\label{eq:supersonic:second:term}
-\tfrac{(1+\alpha)F_T(V)}{\crr-V}.
\end{equation}
Here we used the definition of $F_T$ in \eqref{triple:eq} to identify the second expression.

We now check that both quantities are positive. For \eqref{eq:supersonic:second:term}, this follows from \eqref{triple:eq} and \eqref{def:V2}: the quadratic coefficient of the polynomial $F_T$ is negative and hence $F_T$ is negative for $V<V_2$. Since $\crr-V>0$, the expression in \eqref{eq:supersonic:second:term} is positive.

For \eqref{eq:supersonic:first:term}, the denominator is positive. If $0\le V<V_2$, then \eqref{bound:V2} and the definition $V_0=\frac{1}{1+3\alpha}$ give $V<V_0<1$, and therefore $1-(1+3\alpha)V>0$ and $V(V-1)\le0$. Thus the numerator in \eqref{eq:supersonic:first:term} is positive. If $V<0$, the numerator in \eqref{eq:supersonic:first:term} equals
\[
\crr+\bigl(\alpha-(1+3\alpha)\crr\bigr)V+2\alpha V^2,
\]
which is positive because $\alpha>0$ and $\crr>1$ by the discussion following \eqref{eq:cr:1}. Hence both terms in the convex combination are positive, and \eqref{eq:supersonic:ratio:sign} follows.

We now prove the monotonicity statement. By Proposition~\ref{prop:sub:sonic} and Lemma~\ref{lemma:escape}, the trajectory $(V,Q)$ lies in this region for $\xi\in(\xi_s,+\infty)$. Using \eqref{main:ODE}, we obtain
\[
\xi\partial_\xi\left(\tfrac{\crr-V(\xi)}{Q(\xi)}\right)
=-\tfrac{P_v[V(\xi),Q(\xi)]+\tfrac{\crr-V(\xi)}{Q(\xi)}P_q[V(\xi),Q(\xi)]}
{Q(\xi)\Delta[V(\xi),Q(\xi)]}.
\]
The numerator is positive by \eqref{eq:supersonic:ratio:sign}. Also $Q(\xi)>0$, and $\Delta[V(\xi),Q(\xi)]<0$ by \eqref{P:exp} and $V(\xi)+\alpha Q(\xi)<\crr$. This proves \eqref{eq:supersonic:ratio:mono}.

It remains to prove the monotonicity along the Hugoniot locus. Set $\mathcal M(\xi):=\frac{\crr-V(\xi)}{Q(\xi)}$. By
\eqref{eq:supersonic:ratio:mono}, $\mathcal M$ is strictly increasing. On the
other hand, Lemma~\ref{lemma:RH:properties} gives
\[
\left(
\tfrac{\crr-\mathsf{G}_{\mathrm{RH},v}(\xi)}
{\mathsf{G}_{\mathrm{RH},q}(\xi)}
\right)^2
=\tfrac{\alpha^3(\mathcal M(\xi)^2+\alpha)}
{\gamma \mathcal M(\xi)^2-\alpha^3}.
\]
The function $x\mapsto \frac{\alpha^3(x+\alpha)}{\gamma x-\alpha^3}$ is strictly decreasing for $x>\alpha^2$, since its derivative is $-\frac{\alpha^3(\gamma\alpha+\alpha^3)}{(\gamma x-\alpha^3)^2}<0$. Since $\mathcal M(\xi)>\alpha$ and the downstream ratio is positive, \eqref{eq:supersonic:RH:ratio:mono} follows.
\end{proof}

\subsection{Proof of Theorem~\ref{th:main:sec}} \label{subsec:proof:th}
\begin{proof}[Proof of Theorem~\ref{th:main:sec}]
   Consider the trajectory $(V,Q)(\xi)$ constructed in Proposition~\ref{prop:sub:sonic} for $\xi \in (\xi_s, +\infty)$, and the associated Hugoniot locus  $\mathsf{G}_{\mathrm{RH}}(\xi)$, defined in \eqref{hugoniot}. By \eqref{hugoniot:limit}, \eqref{eq:xis:expl}, \eqref{PRH:loc}, and Lemma~\ref{lemma:RH:properties}, we have
   \begin{equation} \label{lim:ineq}
  \mathsf{G}_{\mathrm{RH}, v}(\xi_s)=V_s< V_2 < \lim_{\xi \to + \infty} \mathsf{G}_{\mathrm{RH}, v}(\xi), \quad \mathsf{G}_{\mathrm{RH},v}(\xi) + \alpha \mathsf{G}_{\mathrm{RH},q}(\xi) \ge \crr.
\end{equation}
 The proof of Proposition~\ref{prop:infty} gives a function $\mathcal{V}^\infty:[Q_2,\infty)\to\mathbb{R}$ defined by $\mathcal{V}^\infty(q):=V^\infty((Q^\infty)^{-1}(q))$. Since $Q^\infty$ is strictly monotone, it has an inverse for $q>Q_2$, and we extend $\mathcal{V}^\infty$ continuously by setting $\mathcal{V}^\infty(Q_2)=V_2$. For $v+\alpha q \ge \crr$, define the continuous function
 \begin{equation*}
     \mathcal{F}(v,q):= \begin{cases}
     \mathcal{V}^\infty(q) - v, \quad& q \ge Q_2, \\
     V_2 -v, \quad& \text{otherwise}.
     \end{cases}
 \end{equation*}
 Observe that $\mathcal{F}(v,q) =0$ on $v+\alpha q \ge \crr$ implies $q \ge Q_2$.

For $q>Q_2$, the branch $(\mathcal{V}^\infty(q),q)$ lies in the supersonic region, whereas $\mathsf{G}_{\mathrm{RH}}(\xi_s)=(V_s,Q_s)$ lies on the sonic line with $Q_s>Q_2$; hence $\mathcal{V}^\infty(Q_s)>V_s$. Also $\mathcal{V}^\infty(q)\le V_2$, while \eqref{lim:ineq} gives $\lim_{\xi\to\infty}\mathsf{G}_{\mathrm{RH},v}(\xi)>V_2$. Therefore,
\begin{equation*}
	\mathcal{F}(\mathsf{G}_{\mathrm{RH}}(\xi_s))>0, \quad \lim_{\xi \to +\infty} \mathcal{F}(\mathsf{G}_{\mathrm{RH}}(\xi)) <0.
\end{equation*}
Applying the intermediate value theorem, we deduce that there exists $\xiRH$ such that $\mathcal{F}(\mathsf{G}_{\mathrm{RH}}(\xiRH))=0$; that is, there exists $\eta>0$ such that $\mathsf{G}_{\mathrm{RH}}(\xiRH)=(V^\infty(\eta),Q^\infty(\eta))$. Moreover, the uniqueness of such intersection follows from the fact that $V^\infty(\xi)$ and $\mathsf{G}_{\mathrm{RH},v}(\xi)$ are both monotone increasing (respectively, from \eqref{mono:Q:infty} and \eqref{claim:tr}), while $\frac{\crr-V^\infty(\xi)}{Q^\infty(\xi)}$ is monotone increasing from \eqref{mono:Q:infty}, and $\frac{\crr-\mathsf{G}_{\mathrm{RH},v}(\xi)}{\mathsf{G}_{\mathrm{RH},q}(\xi)}$  is monotone decreasing from \eqref{eq:supersonic:RH:ratio:mono}.
After rescaling the trajectory $(V^\infty, Q^\infty)$, such that $(V^\infty, Q^\infty)(\xiRH) = \mathsf{G}_{\mathrm{RH}}(\xiRH)$, we extend the definition of $(V,Q)$ for $\xi \in (0,+\infty)$ 
\begin{equation*}
(V,Q)(\xi)= \begin{cases} 
     (V,Q)(\xi) \quad & \xiRH < \xi < +\infty, \\
(V^\infty,Q^\infty)(\xi) \quad & 0 < \xi <\xiRH.
\end{cases}
\end{equation*}
The piecewise smoothness claim is then evident from standard ODE theory.  $(V,Q)$ solve \eqref{eq:euler:ODE} by construction. Positivity of $Q$ follows also from construction. The Rankine--Hugoniot jump conditions follow from the definition of the Hugoniot locus $\mathsf{G}_{\mathrm{RH}}$ in \eqref{hugoniot}, and the Lax entropy inequality is satisfied by construction. The local behavior at $\xi=0$ is described in Proposition~\ref{prop:infty}. The asymptotic behavior at $\xi=+\infty$ is described in Proposition~\ref{prop:power:implosion}.
Also, by Lemma~\ref{lemma:supersonic:uniqueness}, any trajectory starting from a point $\mathsf{G}_{\mathrm{RH}}(\xiRH)$ that is not the rescaling of $(V^\infty,Q^\infty)$ fixed above must hit the sonic line as $\xi\to\xi_0^+$ for some $\xi_0\ge0$.\footnote{It is possible to show that such a solution cannot be continued as an Euler solution.} Thus, after the Rankine--Hugoniot matching point and the rescaling of the inner branch are fixed, the two pieces of the profile are uniquely determined.
\end{proof}


\section{Proof of the main results}
In this section we translate Theorem~\ref{th:main:sec} to the variables $(U,\Sigma,B)$ in order to prove Theorem~\ref{th:main}, and then translate back to the original variables $(\rho,\rho u,E)$ in order to prove Corollary~\ref{cor:main}.

The construction of $(V,Q)$ is supplied by Theorem~\ref{th:main:sec}. It remains to recover $H$, fix the unique normalization giving the $\underline h_1$ asymptotic at infinity, impose the $H$ jump condition at $\xiRH$, and then translate the resulting asymptotics for $H$ into the stated properties of $B=\xi H$.
\begin{proof}[Proof of Theorem~\ref{th:main}]
Integrating \eqref{eq:H}, from $\xiRH$ to $\xi>\xiRH$ we get
\begin{equation*}
H(\xi) = H^+_{\mathrm{RH}} \exp\left( - \int_{\xiRH}^\xi \tfrac{V(\xi')-V_0}{(V(\xi')-\crr)} \tfrac{1}{\xi'} d \xi' \right).
\end{equation*}
Using Proposition~\ref{prop:power:series:infty}, specialized in Proposition~\ref{prop:power:implosion}, for $\xi \to +\infty$, we have
\begin{equation*}
H(\xi ) = H^+_{\mathrm{RH}}\xiRH^{\frac{\crr-\cbb}{\crr}}\xi^{\frac{\cbb-\crr}{\crr}}  \exp\left( - \int_{\xiRH}^\infty \left(\tfrac{V(\xi')-V_0}{V(\xi')-\crr}-\tfrac{\crr-\cbb}{\crr}\right) \tfrac{1}{\xi'} d \xi' + \mathcal{O}(\xi^{-\frac{1}{\crr}}) \right).
\end{equation*}
In particular, we can choose $H^+_{\mathrm{RH}}$ such that
\begin{equation}
H(\xi ) = \underline h_1  \xi^{\frac{\cbb-\crr}{\crr}} + O( \xi^{\frac{\cbb-\crr-1}{\crr}}).
\end{equation}
We then compute $ H^-_{\mathrm{RH}}$ using the jump condition \eqref{RH:jump:H}, and define $H$ for $0<\xi<\xiRH$
\begin{equation*}
H(\xi) = H^-_{\mathrm{RH}} \exp\left(  \int_{\xi}^{\xiRH} \tfrac{V(\xi')-V_0}{V(\xi')-\crr} \tfrac{1}{\xi'} d \xi' \right).
\end{equation*}
Then, using Proposition~\ref{prop:infty}, we obtain, for some constant $\overline h_0$
\begin{equation*}
H(\xi) =\overline h_0 \xi^{-\frac{2}{2+3\gamma}}\exp\left(  \int_{\xi}^1 \left(\tfrac{V(\xi')-V_0}{V(\xi')-\crr}-\tfrac{2}{2+3\gamma}\right) \tfrac{1}{\xi'} d \xi' \right).
\end{equation*}
The constants $H^+_{\mathrm{RH}}$ and $H^-_{\mathrm{RH}}$ are thus fixed uniquely by the normalization at infinity and by \eqref{RH:jump:H}, and the integral formulas solve \eqref{eq:H:main} on both sides of the shock. We then have a solution $(V, Q, H)$ to \eqref{eq:main}; by translating it back to the variables $(U, \Sigma, B)$ Theorem~\ref{th:main} follows naturally from Theorem~\ref{th:main:sec}.
\end{proof}
In order to prove Corollary~\ref{cor:main}, we have to show that the solution naturally associated to the imploding profiles $(\bar U, \bar \Sigma, \bar B)$, and the consequent explosion $(U, \Sigma, B)$ is a weak solution to the Euler equations. The only delicate points are the implosion-explosion transition at $t=0$ and the limited regularity at $r=0$ for $t>0$. We first record the radial weak formulation used below (see also~\cite{JK18}). We will use 
\[ C_c^1((-1,+\infty)\times\RR^+_0) := \{  \varphi(r,t) \in C^1((-1,+\infty)\times[0,+\infty)) \makebox{ with compact support}\} \]
 and 
 \[C_0^1((-1,+\infty)\times\RR^+_0) :=\{ \varphi\in C_c^1((-1,+\infty)\times\RR^+_0) : \varphi(0,t) = 0 \}.\]
\begin{definition} \label{def:radial:weak:solution}
    Let $\rho \in L^1_{loc}([-1, +\infty)\times \RR^3; \RR)$, $\rho \uu \in L^1_{loc}([-1, +\infty)\times \RR^3; \RR^3)$ and $E\in L^1_{loc}([-1, +\infty)\times \RR^3; \RR)$ given. Assume moreover that $\rho, \uu, E$ are given in spherical symmetry, that is
    \begin{equation*}
\rho(x,t) = \rho(r,t), \quad \uu(x,t) = \tfrac{x}{|x|} u(r, t), \quad E(x,t)= E(r,t), 
\end{equation*}
where $r=|x|$. 
Then, we say that $(\rho, \rho \uu, E)$ is a radial weak solution of \eqref{eq:euler:cl} on $[-1, +\infty)$ if
\begin{enumerate}[leftmargin=2em]
  \item   \label{weak:i} $\rho(r,t) > 0$ and $E(r,t) \ge 0$ almost everywhere,
    \item  \label{weak:ii} $(\rho, \rho u, E)(r,t) \in C([-1, +\infty), L^1_{loc}(r^2 dr))$,
    \item   \label{weak:iii} $\rho u^2, p, (p+E)u  \in L^1_{loc}(dt \times r^2 dr)$,  
    \item   \label{weak:iv} the conservation laws for mass, momentum and energy are satisfied in a weak sense
    \begin{subequations} \label{weak:iv:s}
         \begin{align}
    \int_{-1}^\infty \int_{r>0} (\rho \varphi_t + \rho u \varphi_r)r^2 dr dt =0, \quad & \forall \varphi \in C_c^1((-1, +\infty) \times \RR^+_0), \\
    \int_{-1}^\infty \int_{r>0} (\rho u \varphi_t + \rho u^2 \varphi_r + p (\varphi_r + \tfrac{2 \varphi}{r} ))r^2 dr dt=0, \quad & \forall \varphi \in C_0^1((-1, +\infty) \times \RR^+_0), \\
\int_{-1}^\infty \int_{r>0} (E \varphi_t + (E+p)u \varphi_r)r^2 dr dt=0, \quad & \forall \varphi \in C_c^1((-1, +\infty) \times \RR^+_0).
    \end{align} 
    \end{subequations}
  
\end{enumerate}

\end{definition}
We observe that, with these definitions, $(\rho,\rho\uu,E)$ is naturally a weak solution of \eqref{eq:euler:cl}; see Proposition 5.1 in~\cite{JK18}.

We are now ready to prove Corollary~\ref{cor:main}.
\begin{proof}
Given the imploding profiles $(\bar U,\bar \Sigma, \bar B)$ from Theorem~\ref{th:csv} and the exploding solutions $(U, \Sigma, B)$ from Theorem~\ref{th:main}, we define
\begin{align*}
(u,\,\sigma,\,b)(r,t):= \begin{cases}
\left((-t)^{\crr-1} \bar U(\tfrac{r}{(-t)^{\crr}}),\,(-t)^{\crr-1} \bar \Sigma(\tfrac{r}{(-t)^{\crr}}),\,(-t)^{\cbb} \bar B(\tfrac{r}{(-t)^{\crr}})\right) \quad & t<0, \\
\left(\underline v_1 r^{1-\frac{1}{\crr}},\,\underline q_1 r^{1-\frac{1}{\crr}},\,\underline h_1 r^{\frac{\cbb}{\crr}}\right) \quad & t=0, \\
\left(t^{\crr-1} U(\tfrac{r}{t^{\crr}}),\,t^{\crr-1} \Sigma(\tfrac{r}{t^{\crr}}),\,t^{\cbb} B(\tfrac{r}{t^{\crr}})\right) \quad & t>0.
\end{cases},
\end{align*}
Using \eqref{def:pres}, \eqref{def:b} and \eqref{def:sigma}, we can recover the variables $(\rho, E)$:
\begin{equation}
\rho(r,t) =( \alpha \tfrac{\sigma(r,t)}{b(r,t)})^{\tfrac{1}{\alpha}}, \quad E = \tfrac{\rho^\gamma b^2}{\gamma(\gamma-1)} + \tfrac{1}{2} \rho u^2.
\end{equation}

For $t<0$, using the fact that $(\bar U, \bar \Sigma, \bar B)$ are smooth it is immediate to verify that \eqref{weak:i}--\eqref{weak:iv} are satisfied. We now verify that $(\rho(r,t), \rho u(r,t), E(r,t))$ is continuous in $L^1_{loc}$ as $t \to 0^-$. We start with the density $\rho$. By Theorem~\ref{th:csv}, the profiles $(\bar \Sigma, \bar B)(R)$ are locally bounded, vanish exactly to first order at the origin, and have the asymptotics \eqref{eq:lim:infty}; in particular we have the global bound
\begin{equation*}
\left(\tfrac{\bar \Sigma}{\bar B}(R ) \right)^{\frac{1}{\alpha}} \lesssim R^{-\frac{3}{(1+3\alpha)\crr}}.
\end{equation*}
Translating this back in physical variables, we have the uniform bound, for $t<0$, 
\begin{equation*}
|\rho(r,t)| \lesssim r^{-\frac{3}{(1+3\alpha)\crr}}.
\end{equation*}
Together with the pointwise limit \eqref{eq:fields:at:time:of:implosion:b}, this bound gives continuity of $\rho$ at $t=0^-$ by the dominated convergence theorem; its local integrability is exactly \eqref{loc:fin:m}. The pointwise limit for $\rho u$ follows from \eqref{eq:fields:at:time:of:implosion:a} and \eqref{eq:fields:at:time:of:implosion:b}, and the pointwise limit for $E$ is \eqref{eq:fields:at:time:of:implosion:f}. The bounds
\begin{equation*}
| \rho(r,t) u(r,t)| \lesssim r^{-\frac{3}{(1+3\alpha)\crr}-\frac{1}{\crr}+1}, \quad | E(r,t)| \lesssim r^{\frac{\gamma(\crr-1)-\cbb}{\alpha \crr}}.
\end{equation*}
together with \eqref{loc:fin:e} and \eqref{loc:fin:mom} give continuity of $\rho u$ and $E$ at $t=0^-$ by the dominated convergence theorem.

We now observe that, thanks to \eqref{lim:infty:ex}, the following limits hold as $t \to 0^+$
\begin{subequations} \label{lim:explosion}
\begin{align}
&\lim_{t\to 0^+}  \rho(r,t) = \underline \rho_1 r^{\frac{\crr - 1 - \cbb}{\alpha \crr}} 
= \underline \rho_1  r^{-\frac{3}{(1+3\alpha) \crr}}
\label{eq:fields:at:time:of:explosion:a}
, \\
&\lim_{t\to 0^+} p(r,t) = \underline p_1   r^{\frac{\gamma (\crr - 1) - \cbb}{\alpha \crr}},
\label{eq:fields:at:time:of:explosion:b}\\
&\lim_{t\to 0^+}  E(r,t) = \underline e_1 r^{\frac{\gamma (\crr - 1) - \cbb}{\alpha \crr}} 
\label{eq:fields:at:time:of:explosion:c}.
\end{align}
\end{subequations}
The corresponding pointwise limit for the momentum is
\[
\lim_{t\to0^+}\rho(r,t)u(r,t)
=\underline \rho_1 \underline v_1 r^{-\frac{3}{(1+3\alpha)\crr}+1-\frac{1}{\crr}}.
\]
To show the continuity as $t \to 0^+$ we use a similar argument as for $t \to 0^-$, but we will have to take more care of the asymptotics at $\xi=0$ for the density. From the asymptotics \eqref{asy:origin}, we have that for $\xi<1$ 
\begin{equation*}
\left(\tfrac{\Sigma}{B} \right)^{\frac{1}{\alpha}}
\lesssim \xi^{-\frac{6}{2+3\gamma}}
= \xi^{-\frac{3}{(1+3\alpha)(\crr+\frac{3\gamma}{2}V_\infty)}}
\lesssim \xi^{-\frac{3}{(1+3\alpha)\crr}},
\end{equation*}
where the equality uses \eqref{Vinfty:def} and where we used $V_\infty \ge 0$. Then, from \eqref{lim:infty:ex}, we obtain the global bound
\begin{equation*}
\left(\tfrac{ \Sigma}{ B}(\xi ) \right)^{\frac{1}{\alpha}} \lesssim \xi^{-\frac{3}{(1+3\alpha)\crr}},
\end{equation*}
which, translated in original variables, gives the uniform bound for $t>0$
\begin{equation*}
| \rho(r,t)| \lesssim r^{-\frac{3}{(1+3\alpha)\crr}}.
\end{equation*}
A similar argument gives the bounds
\begin{equation*}
| \rho(r,t) u(r,t)| \lesssim r^{-\frac{3}{(1+3\alpha)\crr}-\frac{1}{\crr}+1}, \quad |E(r,t)| \lesssim r^{\frac{\gamma(\crr-1)-\cbb}{\alpha \crr}},
\end{equation*}
and allows us to deduce the continuity of $\rho$, $\rho u$, and $E$ as $t \to 0^+$.

For $t \ge 0$, the conditions \eqref{weak:i} and \eqref{weak:iii} follow from Theorem~\ref{th:main}. We now verify condition \eqref{weak:iv}. Since we have already shown the continuity at $t=0$, and for $t<0$ the identities in \eqref{weak:iv:s} follow from smoothness of the imploding solution and integration by parts, we can consider test functions supported on $[\tfrac{1}{T}, T] \times [0, X]$, for fixed $T>1$. Choose $X>\xiRH T^\crr$ so that the radial support is contained in $[0,X]$.
We then divide the region of integration into the three regions
\begin{equation*}
\mathsf{B}_1 =\{ t \in [\tfrac{1}{T}, T], 0<r\le \delta \},  
\mathsf{B}_2 =\{ t \in [\tfrac{1}{T}, T], \delta<r \le \xiRH t^{\crr} \},
\mathsf{B}_3 =\{ t \in [\tfrac{1}{T}, T], \xiRH t^{\crr} \le r< X\},
\end{equation*}
where $\delta$ is small enough such that $\delta < \xiRH T^{-\crr}$.
For the mass equation, write $\s(t)=\xiRH t^\crr$. Integrating by parts on $\mathsf{B}_2$ and $\mathsf{B}_3$ gives
\begin{align*}
&\int_{\mathsf{B}_2} (\rho \varphi_t + \rho u \varphi_r)r^2 \,dr\,dt
 + \int_{\mathsf{B}_3} (\rho \varphi_t + \rho u \varphi_r)r^2 \,dr\,dt = \\
 &- \int_{\tfrac{1}{T}}^{T} \delta^2 \rho(\delta,t)u(\delta,t)\varphi(\delta,t)\,dt
 + \int_{\tfrac{1}{T}}^{T} \s(t)^2 \varphi(\s(t),t)
\jump{\rho(u-\ds)}(t)\,dt = - \int_{\tfrac{1}{T}}^{T} \delta^2 \rho(\delta,t)u(\delta,t)\varphi(\delta,t)\,dt,
\end{align*}
where in the last equality we used the Rankine--Hugoniot jump condition for mass. Hence
\begin{align*}
\int_{-1}^\infty \int_{r>0} (\rho \varphi_t + \rho u \varphi_r)r^2 \,dr\,dt= \int_{\mathsf{B}_1} (\rho \varphi_t + \rho u \varphi_r)r^2 \,dr\,dt  - \int_{\tfrac{1}{T}}^{T} \delta^2 \rho(\delta,t)u(\delta,t)\varphi(\delta,t)\,dt.
\end{align*}
Similarly, integration by parts on $\mathsf{B}_2$ and $\mathsf{B}_3$, together with the Rankine--Hugoniot jump conditions, gives
\begin{align*}
&\int_{-1}^\infty \int_{r>0}
\left(\rho u \varphi_t + \rho u^2 \varphi_r
+ p\left(\varphi_r + \tfrac{2\varphi}{r}\right)\right)r^2 \,dr\,dt \\
&\qquad =
\int_{\mathsf{B}_1}
\left(\rho u \varphi_t + \rho u^2 \varphi_r
+ p\left(\varphi_r + \tfrac{2\varphi}{r}\right)\right)r^2 \,dr\,dt 
- \int_{\tfrac{1}{T}}^{T}
\delta^2(\rho u^2+p)(\delta,t)\varphi(\delta,t)\,dt.
\end{align*}
For the energy equation, the same argument gives
\begin{align*}
&\int_{-1}^\infty \int_{r>0}
(E\varphi_t + (E+p)u\varphi_r)r^2 \,dr\,dt \\
 &=
\int_{\mathsf{B}_1} (E\varphi_t + (E+p)u\varphi_r)r^2 \,dr\,dt 
- \int_{\tfrac{1}{T}}^{T}
\delta^2(E+p)(\delta,t)u(\delta,t)\varphi(\delta,t)\,dt.
\end{align*}
Using the asymptotics \eqref{asy:origin}, we have the bounds (depending on $T$)
\begin{equation*}
\begin{aligned}
|\rho(r,t)r^2| &\lesssim r^{\frac{6 \gamma -2}{2+3\gamma}}, &
|\rho(r,t)u(r,t)r^2| &\lesssim r^{\frac{9 \gamma}{2+3\gamma}}, \\
|\rho(r,t)u(r,t)^2r^2| &\lesssim r^{\frac{12\gamma+2}{2+3\gamma}}, &
|p(r,t)| + |E(r,t)| &\lesssim 1, \\
|u(r,t)| &\lesssim r, &
|(E+p)(r,t)u(r,t)r^2| &\lesssim r^3.
\end{aligned}
\end{equation*}
In particular,
\begin{align*}
&\left|
\int_{\mathsf{B}_1} (\rho \varphi_t + \rho u \varphi_r)r^2 \,dr\,dt
\right|
+ \left|
\int_{\tfrac{1}{T}}^{T}
\delta^2 \rho(\delta,t)u(\delta,t)\varphi(\delta,t)\,dt
\right|  \lesssim \delta^{\frac{9 \gamma}{2+3\gamma}}, \\
&\left|
\int_{\mathsf{B}_1}
\left(\rho u \varphi_t + \rho u^2 \varphi_r
+ p\left(\varphi_r + \tfrac{2\varphi}{r}\right)\right)r^2 \,dr\,dt
\right|
+ \left|
\int_{\tfrac{1}{T}}^{T}
\delta^2(\rho u^2+p)(\delta,t)\varphi(\delta,t)\,dt
\right|
\lesssim \delta^2, \\
&\left|
\int_{\mathsf{B}_1} (E\varphi_t + (E+p)u\varphi_r)r^2 \,dr\,dt
\right|
+ \left|
\int_{\tfrac{1}{T}}^{T}
\delta^2(E+p)(\delta,t)u(\delta,t)\varphi(\delta,t)\,dt
\right|
\lesssim \delta^3.
\end{align*}
Sending $\delta \to 0$ we obtain all three identities in \eqref{weak:iv:s}.
\end{proof}


\section{Smooth explosions} \label{smooth:explosion}
 In this section we prove Theorem~\ref{thm:2}. The proof could be carried out in the phase portrait by mirroring the proof of existence of implosion profiles in~\cite{CSV26}; however, it is simpler to use the time-reversal observation in Remark~\ref{remark:time:rev}.
 \begin{proof}
Let $(\bar U,\bar \Sigma,\bar B)$ be the smooth profiles from Theorem~\ref{th:csv} for the fixed parameters $d,\gamma,\NNN,\crr$ and $\cbb$ as in Theorem~\ref{thm:2}. Define new profiles
\begin{equation*} 
U(\xi)=-\bar U(\xi), \quad \Sigma(\xi)=\bar \Sigma(\xi), \quad B(\xi) = \bar B(\xi).
\end{equation*}
Consider the global self-similar imploding solution $(\bar u, \bar \sigma, \bar b)$ defined for $t<0$ by the transformation \eqref{bw:ss}. Using the time reversibility of the Euler equations, we define for $t>0$
\begin{equation*}
\begin{aligned}
 u(r,t) &= -\bar u(r, -t) = t^{\crr-1} U\left(\tfrac{r}{t^{\crr}} \right), \\
 \sigma(r,t) &= \bar \sigma(r, -t) = t^{\crr-1} \Sigma\left(\tfrac{r}{t^{\crr}} \right), \\
 b(r,t) &= \bar b(r, -t) = t^{\cbb} B\left(\tfrac{r}{t^{\crr}} \right). 
\end{aligned}
\end{equation*}
This is a global self-similar explosion, and in particular $(U, \Sigma, B)$ is a solution to the forward self-similar Euler equations \eqref{fw:ss:eq}. Equivalently, after identifying $R=\xi$, the system \eqref{bw:ss:euler} becomes \eqref{fw:ss:eq} under $\bar U=-U$, $\bar\Sigma=\Sigma$, and $\bar B=B$.
 \end{proof}

 \begin{remark}[Smooth explosions do not arise as the continuation of the implosions from~\cite{CSV26}]
The smooth explosion we constructed does not arise as a continuation of the implosions $(\bar U, \bar \Sigma, \bar B)$ from Theorem~\ref{th:csv}. Indeed, the time-reversed smooth explosion has asymptotic velocity coefficient $-\underline v_1$, whereas a continuation of the implosions from~\cite{CSV26} would have to match the coefficient $\underline v_1$ at $t=0$. Since $\frac{\underline v_1}{\underline q_1}<0$, these are different phase-plane directions.
 \end{remark}

 \begin{remark}[A family of explosions]
 In the parameter regime treated in Section~\ref{exp:pp}, namely $d=3$, $\NNN=1$, and $\alpha\in(0,1+\sqrt{2}]$, arguments similar to those in that section suggest that one can construct a family of self-similar explosions $(U,\Sigma,B)$ with the same exponents $\crr$ and $\cbb$. These explosions \textit{do not} arise as continuations of the smooth implosions from~\cite{CSV26}. Indeed, there is a family of continuous explosions such that the trajectory $(V,Q)$ in the phase portrait crosses the triple point $(V_2, Q_2)$. The expected generic regularity through the triple point is governed by the exponent $\tfrac{\lambda_-}{\lambda_+}>1$, where $\lambda_\pm$ are defined in \eqref{ev:j}. There is also a family of shock explosions different from the explosions $(U, \Sigma, B)$ constructed in Theorem~\ref{th:main}. Apart from the time-reversed smooth branch in Theorem~\ref{thm:2}, for $\xi$ close to $0$ these continuous and shock explosions all coincide, up to the scaling symmetry in $\xi$, with the solution $(V^\infty,Q^\infty)$ from Proposition~\ref{prop:infty}. However, they all have different behavior as $\xi \to + \infty$; indeed the ratio
 \begin{equation*}
\mathsf{M} := \lim_{\xi \to +\infty} \tfrac{V(\xi)}{Q(\xi)}
\end{equation*}
determines the shock and regularity regime of the explosion. In particular, one expects thresholds $\mathsf{M}_\mathrm{s}$ and $\mathsf{M}_\mathrm{t}$, such that for $- \infty < \mathsf{M} <\mathsf{M}_\mathrm{s}$ the associated explosion cannot be continued smoothly and contains a shock, while for $\mathsf{M}_{\mathrm{s}} < \mathsf{M} < \mathsf{M}_\mathrm{t}$ the profile should be continuous at the triple point. The threshold $\mathsf{M}_\mathrm{t}$ corresponds to the smooth explosion of Theorem~\ref{thm:2}, while the shock threshold $\mathsf{M}_{\mathrm{s}}$ corresponds to the unique solution that reaches the triple point $P_2$ along the eigendirection $\eta_-$.
\end{remark}

\appendix


\section[The polynomial P]{The coefficients of the polynomial \texorpdfstring{$\PP$}{P}} \label{app:coef:p}
The purpose of this Appendix is to prove the bounds stated earlier in~\eqref{eq:lemma:pr}.
We record here the explicit expressions of the coefficients $\cp_j$, for $0 \le j \le 6$, defined in \eqref{polynomial:coeff}.

\begin{subequations}\label{coeff:PP} \allowdisplaybreaks 
\begin{align}
&\begin{aligned}
\cp_0(\alpha)&=\tfrac{3}{20000(1+3\alpha)^4}
\Big(-25(50-36\alpha-810\alpha^2)\sqrt{1+102\alpha+249\alpha^2} \\
&\qquad \qquad \qquad \qquad -25(130-798\alpha+3462\alpha^2+13710\alpha^3)\Big),
\end{aligned}\\[0.3em]
&\begin{aligned}
\cp_1(\alpha)&=\tfrac{3}{20000(1+3\alpha)^4}
\Big(30(45+130\alpha-425\alpha^2)\sqrt{1+102\alpha+249\alpha^2} \\
&\qquad \qquad \qquad \qquad  +30(165-447\alpha-\alpha^2+6675\alpha^3)\Big),
\end{aligned}\\[0.3em]
&\begin{aligned}
\cp_2(\alpha)&=\tfrac{3}{20000(1+3\alpha)^4}
\Big(-45-9372\alpha+30165\alpha^2+65250\alpha^3 \\
&\qquad \qquad \qquad \qquad  +(315-1725\alpha-3750\alpha^2)\sqrt{1+102\alpha+249\alpha^2}\Big),
\end{aligned}\\[0.3em]
&\begin{aligned}
\cp_3(\alpha)&=\tfrac{3}{20000(1+3\alpha)^4}
\Big(-18(116+45\alpha-1275\alpha^2)  -18(20+125\alpha)\sqrt{1+102\alpha+249\alpha^2}\Big),
\end{aligned}\\[0.3em]
&\begin{aligned}
\cp_4(\alpha)&=\tfrac{3}{20000(1+3\alpha)^4}
\Big(-18(-1-146\alpha-235\alpha^2)  -18(5+25\alpha)\sqrt{1+102\alpha+249\alpha^2}\Big),
\end{aligned}\\[0.3em]
&\begin{aligned}
\cp_5(\alpha)&=\tfrac{3}{20000(1+3\alpha)^4}
\Big(108(3+14\alpha)\Big),
\end{aligned}\\[0.3em]
&\begin{aligned}
\cp_6(\alpha)&=\tfrac{3}{20000(1+3\alpha)^4}
\Big(54(1+4\alpha)\Big).
\end{aligned}
\end{align}
\end{subequations}
We recall the Bernstein coefficients introduced in \eqref{Bernstein:coeff:imp}
\begin{equation*}
\beta_j(\alpha) = \tfrac{20000(1+3\alpha)^4}{3} \sum_{i=0}^8 \cp_i(\alpha) \binom{j}{i}\binom{8}{i}^{-1},
\end{equation*}
 and we introduce the following notation to indicate the combinations used in Lemma \ref{lemma:Bernstein}
\begin{equation*}
\beta_{56}(\alpha):= \beta_5(\alpha) + \tfrac{2}{3} \beta_6(\alpha), \quad  \beta_{76}(\alpha):=\beta_7(\alpha) + \tfrac{2}{3} \beta_6(\alpha).
\end{equation*}
Using \eqref{coeff:PP}, a direct but long computation yields the explicit expressions of the Bernstein coefficients (and combinations) $\beta_j$, for $0 \le j \le 8$ and $j=56,76.$
For each $j$, we can introduce $M_j$ and $N_j$ such that
\begin{equation} \label{bernstein:12} 
\beta_j(\alpha) = M_j(\alpha) + \sqrt{1+102\alpha+249\alpha^2} N_j(\alpha).
\end{equation}
A direct computation from \eqref{coeff:PP} gives
\begin{subequations} \label{expr:M}
\begin{align}\allowdisplaybreaks
M_0&=50(-65+399\alpha-1731\alpha^2-6855\alpha^3),\\
M_1&=\tfrac54(-2105+14619\alpha-69243\alpha^2-254175\alpha^3),\\
M_2&=\tfrac1{28}(-56395+455358\alpha-2393445\alpha^2-8130000\alpha^3),\\
M_3&=\tfrac1{28}(-40204+389274\alpha-2321745\alpha^2-7298625\alpha^3),\\
M_4&=\tfrac1{70}(-65347+785148\alpha-5488095\alpha^2-16005000\alpha^3),\\
M_5&=\tfrac1{28}(-15067+232167\alpha-1999065\alpha^2-5440125\alpha^3),\\
M_6&=\tfrac1{28}(-7471+148830\alpha-1716675\alpha^2-4413000\alpha^3),\\
M_7&=\tfrac14(-454+10464\alpha-189975\alpha^2-474375\alpha^3),\\
M_8&=-37+714\alpha-29235\alpha^2-77250\alpha^3,\\
M_{56}&=\tfrac1{84}(-60143+994161\alpha-9430545\alpha^2
-25146375\alpha^3),\\
M_{76}&=\tfrac1{84}(-24476+517404\alpha-7422825\alpha^2
-18787875\alpha^3),
\end{align}
\end{subequations}
and
\begin{subequations} \label{expr:N}
\begin{align}\allowdisplaybreaks
N_0&=50(-25+18\alpha+405\alpha^2),\\
N_1&=\tfrac54(-865+1110\alpha+14925\alpha^2),\\
N_2&=\tfrac1{28}(-25235+50775\alpha+474000\alpha^2),\\
N_3&=\tfrac1{28}(-20060+59850\alpha+421875\alpha^2),\\
N_4&=\tfrac1{70}(-37415+161925\alpha+915000\alpha^2),\\
N_5&=\tfrac1{28}(-10205+64050\alpha+306375\alpha^2),\\
N_6&=\tfrac1{28}(-6065+56025\alpha+243000\alpha^2),\\
N_7&=\tfrac14(-410+5550\alpha+25125\alpha^2),\\
N_8&=-35+375\alpha+3750\alpha^2,\\
N_{56}&=\tfrac1{84}(-42745+304200\alpha+1405125\alpha^2),\\
N_{76}&=\tfrac1{84}(-20740+228600\alpha+1013625\alpha^2).
\end{align}
\end{subequations}

\subsection{Proof of the algebraic inequalities~\eqref{eq:lemma:pr}}\label{app:proof:PP}
We now give the proof of the inequalities \eqref{eq:lemma:pr}; for $0 \le j \le 8$ and $j \neq 6$, or $j=56,76$ those are
\begin{equation*}
\beta_j(\alpha) <0.
\end{equation*}
We first show a useful upper bound on the radical $\sqrt{1+102\alpha+249\alpha^2}$.
\begin{lemma}[Radical bounds]\label{lemma:radical:bounds:new}
For $\alpha\ge0$, set
\[
\mathscr{R}(\alpha):=\sqrt{1+102\alpha+249\alpha^2}.
\]
Then, for $\alpha \ge 0$, we have
\[
\mathscr{R}(\alpha)\le \tfrac{2061\alpha+271}{121},
\qquad
\mathscr{R}(\alpha)\le 18\alpha+2 .
\]
\end{lemma}
\begin{proof}[Proof of Lemma~\ref{lemma:radical:bounds:new}]
The radical $\mathscr{R}$ is concave on $[0,\infty)$, since
\[
\mathscr{R}''(\alpha)
=-\tfrac{2352}{(1+102\alpha+249\alpha^2)^{3/2}}<0.
\]
Hence, given $\alpha_* \ge 0$, for every $\alpha \ge 0$ we have the inequality
\begin{equation*}
\mathscr{R}(\alpha) \le
\mathscr{R}(\alpha_*) + (\alpha-\alpha_*)\mathscr{R}'(\alpha_*).
\end{equation*}
The lemma follows by choosing, respectively, $\alpha_*=\frac{5}{16}$ and $\alpha_*=\frac{1}{5}$.
\end{proof}
We now collect two elementary cubic positivity criteria.
\begin{lemma}[A cubic positivity criterion]\label{lemma:cubic:criterion:new}
Let $\mathsf{a}_0,\mathsf{a}_1,\mathsf{a}_2,\mathsf{a}_3>0$ and assume
\[
\mathsf{a}_1^2-4\mathsf{a}_0\mathsf{a}_2<0.
\]
Then
\[
\mathsf{a}_3\alpha^3+\mathsf{a}_2\alpha^2-\mathsf{a}_1\alpha
+\mathsf{a}_0>0
\qquad\text{for every }\alpha\ge0.
\]
\end{lemma}
\begin{proof}[Proof of Lemma~\ref{lemma:cubic:criterion:new}]
Since $\mathsf{a}_3\alpha^3\ge0$, it is enough to prove that
$\mathsf{a}_2\alpha^2-\mathsf{a}_1\alpha+\mathsf{a}_0>0$.  This quadratic has negative discriminant by the assumed inequality, and it is positive at $\alpha=0$; hence the claim follows.
\end{proof}

\begin{lemma}[A one-point cubic criterion]\label{lemma:cubic:onepoint:new}
Let
\[
f(\alpha)=\mathsf{a}_3\alpha^3+\mathsf{a}_2\alpha^2-\mathsf{a}_1\alpha
+\mathsf{a}_0,
\qquad
\mathsf{a}_0,\mathsf{a}_1,\mathsf{a}_2,\mathsf{a}_3>0.
\]
Suppose that, for some $\alpha^*>0$,
\[
f'(\alpha^*)>0,
\qquad
\mathsf{a}_0-\mathsf{a}_2(\alpha^*)^2-2\mathsf{a}_3(\alpha^*)^3>0.
\]
Then $f(\alpha)>0$ for every $\alpha\ge0$.
\end{lemma}
\begin{proof}[Proof of Lemma~\ref{lemma:cubic:onepoint:new}]
We have $f''(\alpha)=6\mathsf{a}_3\alpha+2\mathsf{a}_2>0$ for
$\alpha\ge0$, so $f'$ is
strictly increasing.  Since $f'(0)=-\mathsf{a}_1<0$ and
$f'(\alpha^*)>0$, the unique critical point $\alpha_0$ of $f$ on
$[0,\infty)$ lies in $(0,\alpha^*)$.  At this point,
$\mathsf{a}_1=3\mathsf{a}_3\alpha_0^2+2\mathsf{a}_2\alpha_0$, and therefore
\[
f(\alpha_0)=\mathsf{a}_0-\mathsf{a}_2\alpha_0^2-2\mathsf{a}_3\alpha_0^3
\ge \mathsf{a}_0-\mathsf{a}_2(\alpha^*)^2-2\mathsf{a}_3(\alpha^*)^3>0.
\]
Thus the global minimum of $f$ on $[0,\infty)$ is positive.
\end{proof}
We can now prove \eqref{eq:lemma:pr}.
\begin{proposition}\label{prop:new:proof:p:sign}
For every $\alpha\in(0,1+\sqrt2]$ and $j \in \{0,1,2,3,4,5,7,8,56,76\}$ we have
\[
\beta_j(\alpha)<0.
\]
\end{proposition}
\begin{proof}[Proof of Proposition~\ref{prop:new:proof:p:sign}]
To prove the coefficients $\beta_j$ are negative we will use the following argument. Assume we have an upper bound $\mathscr{R}(\alpha) \le L(\alpha)$ for $\alpha \in (0, 1+\sqrt{2}]$. Then, 
\begin{equation} \label{beta:negative:arg}
\beta_j(\alpha) = M_j(\alpha) + \mathscr{R} N_j(\alpha)  \le \max\{ M_j(\alpha), M_j(\alpha) + L(\alpha) N_j(\alpha) \}, \qquad \alpha \in (0, 1+\sqrt{2}].
\end{equation}
In particular, by using Lemma \ref{lemma:radical:bounds:new}, we reduce the proof of $\beta_j(\alpha) < 0 $ to proving that the $20$ cubics in $\alpha$, $ \{ M_j(\alpha), M_j(\alpha) + L(\alpha) N_j(\alpha) \}$, are negative on the interval $(0, 1+\sqrt{2}]$. 

\textit{The $M_j$ are negative.}
We will first prove that $M_j(\alpha)<0$ for $\alpha \in (0, 1+\sqrt{2}]$ and for each index $j\in\{0,1,2,3,4,5,7,8,56,76\}$.  The cubics $-M_j(\alpha)$ satisfy the assumptions
of Lemma~\ref{lemma:cubic:criterion:new}.  Let $\Delta_j$ denote the discriminant $\mathsf{a}_1^2-4\mathsf{a}_0\mathsf{a}_2$ associated to the cubic $-M_j$, then
a direct computation gives
\[
\begin{alignedat}{3}
\Delta_0&=-727147500<0,\quad&
\Delta_1&=-\tfrac{9232772475}{16}<0,\quad&
\Delta_2&=-\tfrac{41570301867}{98}<0,\\
\Delta_3&=-\tfrac{7922839173}{28}<0,\quad&
\Delta_4&=-\tfrac{204516198489}{1225}<0,\quad&
\Delta_5&=-\tfrac{1358737419}{16}<0,\\
\Delta_7&=-\tfrac{29437413}{2}<0,\quad&
\Delta_8&=-3816984<0,\quad&
\Delta_{56}&=-\tfrac{426789659273}{2352}<0,\\
\Delta_{76}&=-\tfrac{195160442}{3}<0.
\end{alignedat}
\]

\textit{ $\beta_0$ and $\beta_1$ are negative.} Using Lemma \ref{lemma:radical:bounds:new}, and \eqref{beta:negative:arg}, it is enough to prove
\begin{align*}
M_0+\tfrac{2061\alpha+271}{121}N_0&=-\tfrac{300}{121}
(-875\alpha^3+10433\alpha^2-272\alpha+2440)<0,
\\
M_1+\tfrac{2061\alpha+271}{121} N_1&=-\tfrac{15}{242}
(-875\alpha^3+341003\alpha^2-47824\alpha+81520)<0.
\end{align*}
Since $0<\alpha\le1+\sqrt2<\frac{29}{12}$, we have
\begin{align*}
-875\alpha^3+10433\alpha^2-272\alpha+2440
&>
\tfrac{99821}{12}\alpha^2-272\alpha+2440,\\
-875\alpha^3+341003\alpha^2-47824\alpha+81520
&>
\tfrac{4066661}{12}\alpha^2-47824\alpha+81520.
\end{align*}
The two quadratics on the right have discriminants
$-\frac{243341288}{3}$ and $-\frac{324652799792}{3}$ respectively, and hence are positive.
Thus $\beta_0<0$ and $\beta_1<0$.

\textit{ $\beta_2$, $\beta_3$,  $\beta_4$,  $\beta_7$ and $\beta_8$ are negative.}
Using Lemma \ref{lemma:radical:bounds:new}, and \eqref{beta:negative:arg}, it is enough to prove
\begin{align*}
M_2+\tfrac{2061\alpha+271}{121}N_2
&=-\tfrac{6}{3388}
(1136000\alpha^3+9417595\alpha^2-2808168\alpha+2277080)<0,\\
M_3+\tfrac{2061\alpha+271}{121}N_3
&=-\tfrac{6}{3388}
(2274875\alpha^3+7208695\alpha^2-3662974\alpha+1716824)<0,\\
M_4+\tfrac{2061\alpha+271}{121}N_4
&=-\tfrac{6}{8470}
(8465000\alpha^3+13727845\alpha^2-10295378\alpha+3007742)<0,\\
M_7+(18\alpha+2)N_7
&=-\tfrac14(22125\alpha^3+39825\alpha^2-14184\alpha+1274)<0,\\
M_8+\tfrac{2061\alpha+271}{121}N_8
&=-\tfrac{6}{121}
(269750\alpha^3+291385\alpha^2-19314\alpha+2327)<0.
\end{align*}
The cubics on the right hand side satisfy the assumptions of Lemma~\ref{lemma:cubic:criterion:new}. Let
$\tilde{\Delta}_j$ denote the discriminant
$\mathsf{a}_1^2-4\mathsf{a}_0\mathsf{a}_2$ associated to the cubic
$-M_j-L N_j$, a direct computation gives
\[
\begin{alignedat}{3}
\tilde{\Delta}_2&=-\tfrac{15932589826536}{65219}<0,\quad&
\tilde{\Delta}_3&=-\tfrac{81195443581599}{717409}<0,\quad&
\tilde{\Delta}_4&=-\tfrac{76068585952812}{2562175}<0,\\
\tilde{\Delta}_7&=-\tfrac{220293}{2}<0,\quad&
\tilde{\Delta}_8&=-\tfrac{84210515424}{14641}<0.
\end{alignedat}
\]
Thus $\beta_i<0$ for $i=2,3,4,7,8$.

\textit{ $\beta_5$, $\beta_{56}$ and $\beta_{76}$ are negative.}
Again, using
Lemma \ref{lemma:radical:bounds:new}, and \eqref{beta:negative:arg}, it is enough
to prove
\begin{align*}
f_5(\alpha)&:=M_5+\tfrac{2061\alpha+271}{121}N_5
=-\tfrac{6}{3388}
(4469375\alpha^3+4475365\alpha^2-4069542\alpha+764777)<0,\\
f_{56}(\alpha)&:=M_{56}+\tfrac{2061\alpha+271}{121}N_{56}
=-\tfrac{6}{10164}
(24458125\alpha^3+22225145\alpha^2-19105706\alpha+3143533)<0,\\
f_{76}(\alpha)&:=M_{76}+(18\alpha+2)N_{76}
=-\tfrac1{84}(542625\alpha^3+1280775\alpha^2-601284\alpha+65956)<0.
\end{align*}
The cubics $-f_5$, $-f_{56}$, and $-f_{76}$ are positive by
Lemma~\ref{lemma:cubic:onepoint:new}.  A direct computation gives
\[
\begin{aligned}
\alpha_5^*&=\tfrac5{16},&
(-f_5)'(\alpha_5^*)&=\tfrac{28376319}{433664},&
\mathsf{a}_0-\mathsf{a}_2(\alpha_5^*)^2-2\mathsf{a}_3(\alpha_5^*)^3
&=\tfrac{337555263}{3469312}>0,\\
\alpha_{56}^*&=\tfrac7{24},&
(-f_{56})'(\alpha_{56}^*)&=\tfrac{6456271}{108416},&
\mathsf{a}_0-\mathsf{a}_2(\alpha_{56}^*)^2-2\mathsf{a}_3(\alpha_{56}^*)^3
&=\tfrac{270577961}{11708928}>0,\\
\alpha_{76}^*&=\tfrac5{24},&
(-f_{76})'(\alpha_{76}^*)&=\tfrac{193699}{5376},&
\mathsf{a}_0-\mathsf{a}_2(\alpha_{76}^*)^2-2\mathsf{a}_3(\alpha_{76}^*)^3
&=\tfrac{1275749}{193536}>0.
\end{aligned}
\]
This concludes the proof of the fact that $\beta_j<0$, for $j=0,1,2,3,4,5,7,8,56,76$. \end{proof}


\section[The polynomial Q]{The coefficients of the polynomial \texorpdfstring{$\QQ$}{Q}}\label{app:coef:b}
 We recall that $\crr=\crr(\alpha)$ is defined in~\eqref{formula:NNN:1} as
	 \begin{equation*}
	\crr(\alpha)=\tfrac{5-3\alpha+\sqrt{1+102\alpha+249\alpha^2}}{4+12\alpha},
	\end{equation*}
while $V_2=V_2(\alpha)$ is defined in~\eqref{def:V2} as 
	\begin{equation*}
	V_2(\alpha)=\tfrac{1}{4}\left(
	-1+3\crr+\tfrac{\tfrac{3}{1+3\alpha}+2-2\crr}{1+2\alpha}
	-\sqrt{\left(1-3\crr-\tfrac{\tfrac{3}{1+3\alpha}+2-2\crr}{1+2\alpha}\right)^2
	-\tfrac{8\crr\left(\tfrac{3}{1+3\alpha}+2-2\crr\right)}{1+2\alpha}}
	\right).
	\end{equation*}
 We now record here the explicit expressions of the coefficients $\cq_j$, for $0 \le j \le 5$, defined in \eqref{coeff:QQ}, 
\begingroup\scriptsize
\allowdisplaybreaks
\begin{subequations}\label{coeff:qq:i}
\begin{align}
\cq_0(\alpha)
&= {}-\tfrac{1}{25 \alpha (1 + 5 \alpha + 6 \alpha^2)} \crr \Bigl(32 \crr^3 \alpha (1 + 3 \alpha) - V_2^2 \alpha (119 + 132 \alpha)\notag\\
&\qquad {} + \crr V_2 (150 \alpha^3 - 5 \sqrt{6} \sqrt{\tfrac{\alpha}{1 + 2 \alpha}} + \alpha^2 (389 + 96 V_2 - 30 \sqrt{6} \sqrt{\tfrac{\alpha}{1 + 2 \alpha}}) + \alpha (263 + 32 V_2 - 25 \sqrt{6} \sqrt{\tfrac{\alpha}{1 + 2 \alpha}})) \notag\\
&\qquad {} + \crr^2 (5 \sqrt{6} \sqrt{\tfrac{\alpha}{1 + 2 \alpha}}  + 6 \alpha^2 (-22 - 32 V_2 + 5 \sqrt{6} \sqrt{\tfrac{\alpha}{1 + 2 \alpha}}) + \alpha (-119 - 64 V_2 + 25 \sqrt{6} \sqrt{\tfrac{\alpha}{1 + 2 \alpha}}))\Bigr), \\[0.5ex]
\cq_1(\alpha)
&= \tfrac{1}{125 \alpha (1 + 5 \alpha + 6 \alpha^2)} \Bigl(-2 \sqrt{6} V_2^3 \sqrt{\tfrac{\alpha}{1 + 2 \alpha}} (47 + 66 \alpha)  + 5 \crr^4 (1 + 3 \alpha) (-6 + 11 \sqrt{6} \sqrt{\tfrac{\alpha}{1 + 2 \alpha}} + \alpha (-86 + 20 \sqrt{6} \sqrt{\tfrac{\alpha}{1 + 2 \alpha}})) \notag\\
&\qquad {}- \crr V_2^2 (60 - 332 \sqrt{6} \sqrt{\tfrac{\alpha}{1 + 2 \alpha}} + 55 \sqrt{6} V_2 \sqrt{\tfrac{\alpha}{1 + 2 \alpha}} + 450 \sqrt{6} \alpha^3 \sqrt{\tfrac{\alpha}{1 + 2 \alpha}} 
+ 15 \alpha^2 (-62 + 5 \sqrt{6} \sqrt{\tfrac{\alpha}{1 + 2 \alpha}} + 20 \sqrt{6} V_2 \sqrt{\tfrac{\alpha}{1 + 2 \alpha}}) \notag\\
&\qquad {}+ \alpha (  -505 - 571 \sqrt{6} \sqrt{\tfrac{\alpha}{1 + 2 \alpha}} + 265 \sqrt{6} V_2 \sqrt{\tfrac{\alpha}{1 + 2 \alpha}})) + \crr^3 (-60 + 69 \sqrt{6} \sqrt{\tfrac{\alpha}{1 + 2 \alpha}} 
- 30 \alpha^2 (-31 + 5 \sqrt{6} \sqrt{\tfrac{\alpha}{1 + 2 \alpha}}) \notag\\
&\qquad {}+ \alpha (505 + 7 \sqrt{6} \sqrt{\tfrac{\alpha}{1 + 2 \alpha}}) 
- 5 V_2 (1 + 3 \alpha) (-12 + 38 \sqrt{6} \sqrt{\tfrac{\alpha}{1 + 2 \alpha}} + 90 \sqrt{6} \alpha^2 \sqrt{\tfrac{\alpha}{1 + 2 \alpha}} + \alpha (-172 + 115 \sqrt{6} \sqrt{\tfrac{\alpha}{1 + 2 \alpha}})))  \notag\\
&\qquad {}+ \crr^2 V_2 (120 - 307 \sqrt{6} \sqrt{\tfrac{\alpha}{1 + 2 \alpha}} 
+ 150 \alpha^3 (-5 + 3 \sqrt{6} \sqrt{\tfrac{\alpha}{1 + 2 \alpha}}) + 5 \alpha^2 (-497 + 45 \sqrt{6} \sqrt{\tfrac{\alpha}{1 + 2 \alpha}}) - \alpha (1135 + 446 \sqrt{6} \sqrt{\tfrac{\alpha}{1 + 2 \alpha}})  \notag\\
&\qquad {}+ 5 V_2 (1 + 3 \alpha) (-6 + 38 \sqrt{6} \sqrt{\tfrac{\alpha}{1 + 2 \alpha}} 
+ 90 \sqrt{6} \alpha^2 \sqrt{\tfrac{\alpha}{1 + 2 \alpha}} + \alpha (-86 + 115 \sqrt{6} \sqrt{\tfrac{\alpha}{1 + 2 \alpha}})))\Bigr), \\[0.5ex]
\cq_2(\alpha)
&= \tfrac{1}{625 \alpha^3 (1 + 3 \alpha)} (\tfrac{\alpha}{1 + 2 \alpha})^{\frac{3}{2}} 
\Bigl(2 \crr^4 (1 + 3 \alpha) (-30 \sqrt{6}  + 186 \sqrt{\tfrac{\alpha}{1 + 2 \alpha}} + 125 \alpha^2 (\sqrt{6} - 10 \sqrt{\tfrac{\alpha}{1 + 2 \alpha}}) - 25 \alpha (9 \sqrt{6} + 13 \sqrt{\tfrac{\alpha}{1 + 2 \alpha}})) \notag\\
&\qquad {}+ V_2^3 (372 V_2 \sqrt{\tfrac{\alpha}{1 + 2 \alpha}} + 1716 V_2 \alpha \sqrt{\tfrac{\alpha}{1 + 2 \alpha}}  - 3900 \alpha^3 \sqrt{\tfrac{\alpha}{1 + 2 \alpha}} + 30 (\sqrt{6} - 5 \sqrt{\tfrac{\alpha}{1 + 2 \alpha}}) - 25 \alpha (5 \sqrt{6} + 81 \sqrt{\tfrac{\alpha}{1 + 2 \alpha}})  \notag\\
&\qquad {}- 90 \alpha^2 (3 \sqrt{6} + 60 \sqrt{\tfrac{\alpha}{1 + 2 \alpha}} - 20 V_2 \sqrt{\tfrac{\alpha}{1 + 2 \alpha}}))
+ \crr^3 (-5 (6 \sqrt{6} - 625 \alpha \sqrt{\tfrac{\alpha}{1 + 2 \alpha}} + \alpha^2 (71 \sqrt{6} - 2000 \sqrt{\tfrac{\alpha}{1 + 2 \alpha}}) \notag\\
&\qquad {}+ 150 \alpha^3 (\sqrt{6} - 10 \sqrt{\tfrac{\alpha}{1 + 2 \alpha}})) + V_2 (1 + 3 \alpha) (180 \sqrt{6} - 1788 \sqrt{\tfrac{\alpha}{1 + 2 \alpha}} 
+ 25 \alpha (59 \sqrt{6} - 165 \sqrt{\tfrac{\alpha}{1 + 2 \alpha}}) \notag\\
&\qquad {}+ 250 \alpha^3 (3 \sqrt{6} - 10 \sqrt{\tfrac{\alpha}{1 + 2 \alpha}}) 
- 125 \alpha^2 (\sqrt{6} + 28 \sqrt{\tfrac{\alpha}{1 + 2 \alpha}}))) - \crr^2 V_2 (-V_2 (1 + 3 \alpha) (-180 \sqrt{6} + 2832 \sqrt{\tfrac{\alpha}{1 + 2 \alpha}} \notag\\
&\qquad {}+ 15000 \alpha^4 \sqrt{\tfrac{\alpha}{1 + 2 \alpha}} - 25 \alpha (59 \sqrt{6} - 457 \sqrt{\tfrac{\alpha}{1 + 2 \alpha}}) - 750 \alpha^3 (\sqrt{6}  - 30 \sqrt{\tfrac{\alpha}{1 + 2 \alpha}}) + 125 \alpha^2 (\sqrt{6} + 166 \sqrt{\tfrac{\alpha}{1 + 2 \alpha}})) \notag\\
&\qquad {}+ 5 (-18 \sqrt{6} + 30 \sqrt{\tfrac{\alpha}{1 + 2 \alpha}}  + 1780 \alpha \sqrt{\tfrac{\alpha}{1 + 2 \alpha}} + 150 \alpha^4 (\sqrt{6} + 10 \sqrt{\tfrac{\alpha}{1 + 2 \alpha}}) + \alpha^3 ( -325 \sqrt{6} + 5780 \sqrt{\tfrac{\alpha}{1 + 2 \alpha}})  \notag\\
&\qquad {}+ \alpha^2 (-188 \sqrt{6} + 5955 \sqrt{\tfrac{\alpha}{1 + 2 \alpha}})))  - \crr V_2^2 (-5 (-18 \sqrt{6} + 60 \sqrt{\tfrac{\alpha}{1 + 2 \alpha}} - 3000 \alpha^5 \sqrt{\tfrac{\alpha}{1 + 2 \alpha}} + 50 \alpha^4 (3 \sqrt{6} - 20 \sqrt{\tfrac{\alpha}{1 + 2 \alpha}})  \notag\\
&\qquad {}+ 5 \alpha (5 \sqrt{6} + 337 \sqrt{\tfrac{\alpha}{1 + 2 \alpha}}) + \alpha^3 (-175 \sqrt{6} 
+ 5310 \sqrt{\tfrac{\alpha}{1 + 2 \alpha}}) + \alpha^2 (-63 \sqrt{6} + 5660 \sqrt{\tfrac{\alpha}{1 + 2 \alpha}}))  \notag\\
&\qquad {}+ V_2 (1 + 3 \alpha) ( -60 \sqrt{6} + 1788 \sqrt{\tfrac{\alpha}{1 + 2 \alpha}} + 2500 \alpha^3 \sqrt{\tfrac{\alpha}{1 + 2 \alpha}} + 250 \alpha^2 (\sqrt{6} + 34 \sqrt{\tfrac{\alpha}{1 + 2 \alpha}}) + \alpha (-450 \sqrt{6} + 6625 \sqrt{\tfrac{\alpha}{1 + 2 \alpha}})))\Bigr), \\[0.5ex]
\cq_3(\alpha)
&= \tfrac{1}{625 \alpha^2 (1 + 2 \alpha)^2} \Bigl(6 \crr^4 (-6 + 10 \sqrt{6} \sqrt{\tfrac{\alpha}{1 + 2 \alpha}} + 50 \alpha^2 (-1 + 2 \sqrt{6} \sqrt{\tfrac{\alpha}{1 + 2 \alpha}}) + \alpha (-73 + 70 \sqrt{6} \sqrt{\tfrac{\alpha}{1 + 2 \alpha}}))  \notag\\
&\qquad {}- 2 \crr^3 (15 \sqrt{6} \sqrt{\tfrac{\alpha}{1 + 2 \alpha}} ( -1 + 4 \alpha^2) + V_2 (-72 + 105 \sqrt{6} \sqrt{\tfrac{\alpha}{1 + 2 \alpha}} + 500 \alpha^3 (-9 + \sqrt{6} \sqrt{\tfrac{\alpha}{1 + 2 \alpha}}) \notag\\
&\qquad {}+ 50 \alpha^2 (-75 + 34 \sqrt{6} \sqrt{\tfrac{\alpha}{1 + 2 \alpha}}) + 17 \alpha (-78 + 55 \sqrt{6} \sqrt{\tfrac{\alpha}{1 + 2 \alpha}}))) \notag\\
&\qquad {}+ \crr V_2^2 (-15 (-1 + 4 \alpha^2) (6 \sqrt{6} \sqrt{\tfrac{\alpha}{1 + 2 \alpha}} + 50 \sqrt{6} \alpha^2 \sqrt{\tfrac{\alpha}{1 + 2 \alpha}} + 5 \alpha (-12 + 5 \sqrt{6} \sqrt{\tfrac{\alpha}{1 + 2 \alpha}}))  \notag\\
&\qquad {}+ V_2 (144 + 25000 \alpha^5 - 150 \sqrt{6} \sqrt{\tfrac{\alpha}{1 + 2 \alpha}}  
+ \alpha (2652 - 3125 \sqrt{6} \sqrt{\tfrac{\alpha}{1 + 2 \alpha}}) - 1000 \alpha^4 (-35 + 16 \sqrt{6} \sqrt{\tfrac{\alpha}{1 + 2 \alpha}}) \notag\\
&\qquad {}- 250 \alpha^3 (-101 + 90 \sqrt{6} \sqrt{\tfrac{\alpha}{1 + 2 \alpha}}) - 100 \alpha^2 (-100 + 129 \sqrt{6} \sqrt{\tfrac{\alpha}{1 + 2 \alpha}}))) \notag\\
&\qquad {}+ \crr^2 V_2 (90 (-1 + 4 \alpha^2) (-5 \alpha + \sqrt{6} \sqrt{\tfrac{\alpha}{1 + 2 \alpha}}) + V_2 (54 (-4 + 5 \sqrt{6} \sqrt{\tfrac{\alpha}{1 + 2 \alpha}}) + 2500 \alpha^4 (-1 + 6 \sqrt{6} \sqrt{\tfrac{\alpha}{1 + 2 \alpha}})  \notag\\
&\qquad {}+ 500 \alpha^3 (-41 + 42 \sqrt{6} \sqrt{\tfrac{\alpha}{1 + 2 \alpha}})  + 12 \alpha (-369 + 320 \sqrt{6} \sqrt{\tfrac{\alpha}{1 + 2 \alpha}}) + 25 \alpha^2 (-601 + 534 \sqrt{6} \sqrt{\tfrac{\alpha}{1 + 2 \alpha}}))) \notag\\
&\qquad {}+ V_2^3 (5 (-1 + 4 \alpha^2) (-250 \alpha^3 + 6 \sqrt{6} \sqrt{\tfrac{\alpha}{1 + 2 \alpha}} + 15 \alpha (-6 + 5 \sqrt{6} \sqrt{\tfrac{\alpha}{1 + 2 \alpha}}) \notag\\
&\qquad {}+ 25 \alpha^2 (-5 + 6 \sqrt{6} \sqrt{\tfrac{\alpha}{1 + 2 \alpha}})) + V_2 (-10000 \alpha^5 + 1000 \alpha^4 (-15 + \sqrt{6} \sqrt{\tfrac{\alpha}{1 + 2 \alpha}}) + 2500 \alpha^3 (-3 + \sqrt{6} \sqrt{\tfrac{\alpha}{1 + 2 \alpha}})  \notag\\
&\qquad {}+ 6 (-6 + 5 \sqrt{6} \sqrt{\tfrac{\alpha}{1 + 2 \alpha}}) + 50 \alpha^2 (  -31 + 47 \sqrt{6} \sqrt{\tfrac{\alpha}{1 + 2 \alpha}}) + \alpha (-438 + 735 \sqrt{6} \sqrt{\tfrac{\alpha}{1 + 2 \alpha}})))\Bigr), \\[0.5ex]
\cq_4(\alpha)
&= \tfrac{1}{625 \alpha^2 (1 + 2 \alpha)^2} \Bigl(72 \crr^4 (1 + 5 \alpha) + 6 \crr^3 V_2 (-48 + 25 \sqrt{6} \sqrt{\tfrac{\alpha}{1 + 2 \alpha}} + 260 \sqrt{6} \alpha^2 \sqrt{\tfrac{\alpha}{1 + 2 \alpha}} + 60 \alpha (-4 + 3 \sqrt{6} \sqrt{\tfrac{\alpha}{1 + 2 \alpha}})) \notag\\
&\qquad {}+ V_2^4 (72 - 10000 \alpha^5 - 150 \sqrt{6} \sqrt{\tfrac{\alpha}{1 + 2 \alpha}} + \alpha (810 - 955 \sqrt{6} \sqrt{\tfrac{\alpha}{1 + 2 \alpha}}) + \alpha^2 (2975 - 810 \sqrt{6} \sqrt{\tfrac{\alpha}{1 + 2 \alpha}})  \notag\\
&\qquad {}+ 500 \alpha^4 (-25 + 2 \sqrt{6} \sqrt{\tfrac{\alpha}{1 + 2 \alpha}})  + 100 \alpha^3 (4 + 15 \sqrt{6} \sqrt{\tfrac{\alpha}{1 + 2 \alpha}})) - 18 \crr^2 V_2^2 (-24 - 300 \alpha^3 + 25 \sqrt{6} \sqrt{\tfrac{\alpha}{1 + 2 \alpha}} \notag\\
&\qquad {}+ 20 \alpha^2 (-10 + 13 \sqrt{6} \sqrt{\tfrac{\alpha}{1 + 2 \alpha}}) + 5 \alpha (-29 + 36 \sqrt{6} \sqrt{\tfrac{\alpha}{1 + 2 \alpha}}))  - \crr V_2^3 (288 - 450 \sqrt{6} \sqrt{\tfrac{\alpha}{1 + 2 \alpha}} + 1000 \sqrt{6} \alpha^4 \sqrt{\tfrac{\alpha}{1 + 2 \alpha}} \notag\\
&\qquad {}+ \alpha^2 (7200 - 3930 \sqrt{6} \sqrt{\tfrac{\alpha}{1 + 2 \alpha}}) + 300 \alpha^3 (36 + 5 \sqrt{6} \sqrt{\tfrac{\alpha}{1 + 2 \alpha}}) - 5 \alpha (-468 + 623 \sqrt{6} \sqrt{\tfrac{\alpha}{1 + 2 \alpha}}))\Bigr), \\[0.5ex]
\cq_5(\alpha)
&= -\tfrac{1}{625 \alpha^2 (1 + 2 \alpha)^2} \Bigl(36 (\crr - V_2)^4 + 144 (\crr - V_2)^4 \alpha + 900 (\crr - V_2)^2 V_2^2 \alpha (1+ 2 \alpha) + 3600 (\crr - V_2)^2 V_2^2 \alpha^2 (1 + 2 \alpha)  \notag\\
&\qquad {}+ 625 V_2^4 \alpha^2 (1 + 2 \alpha)^2 + 2500 V_2^4 \alpha^3 (1 + 2 \alpha)^2 
+  ( 120 (\crr - V_2)^3 V_2 \alpha) \sqrt{\tfrac{6 + 12 \alpha}{\alpha}}
+ ( 480 (\crr - V_2)^3 V_2 \alpha^2) \sqrt{\tfrac{6 + 12 \alpha}{\alpha}} \notag\\
&\qquad {}+  ( 2000 (\crr - V_2) V_2^3 \alpha^3 (1 + 2 \alpha))  \sqrt{\tfrac{6 + 12 \alpha}{\alpha}} + 500 \sqrt{6} (\crr - V_2) (\tfrac{\alpha}{1 + 2 \alpha})^{\frac{3}{2}} (V_2 + 2 V_2 \alpha)^3\Bigr).
\end{align}
\end{subequations}
\endgroup

The corresponding Bernstein coefficients $\BB_j$, for $0 \le j \le 5$, defined in \eqref{Bernstein:coeff:exp} can be computed from \eqref{coeff:qq:i}.
\begingroup\scriptsize
\allowdisplaybreaks
\begin{subequations}\label{coeff:bb:i}
\begin{align}
\BB_0(\alpha)
&= -\tfrac{1}{25 \alpha^{\frac{1}{2}} (1 + \alpha (5 + 6 \alpha))} \crr \Bigl(32 \crr^{3} \sqrt{\alpha} (1 + 3 \alpha) - V_2^{2} \sqrt{\alpha} (119 + 132 \alpha) + \crr^{2} (5 \sqrt{6 + 12 \alpha} + 15 \alpha \sqrt{6 + 12 \alpha}  \notag\\
&\qquad {}- \sqrt{\alpha} (119 + 132 \alpha + 64 V_2 (1 + 3 \alpha))) + \crr V_2 (-5 \sqrt{6 + 12 \alpha} - 15 \alpha \sqrt{6 + 12 \alpha}  \notag\\
&\qquad {}+ \sqrt{\alpha} (263 + 32 V_2 (1 + 3 \alpha) + \alpha (389 + 150 \alpha)))\Bigr), \\[0.5ex]
\BB_1(\alpha)
&= -\tfrac{1}{625 \alpha \sqrt{1 + 2 \alpha} (1 + \alpha (5 + 6 \alpha))} \Bigl(2 \sqrt{6} V_2^{3} \sqrt{\alpha} (47  + 66 \alpha) + 5 \crr^{4} (1 + 3 \alpha) (-11 \sqrt{6} \sqrt{\alpha} - 20 \sqrt{6} \alpha^{\frac{3}{2}} + 6 \sqrt{1 + 2 \alpha}+ 246 \alpha \sqrt{1 + 2 \alpha}) \notag\\
&\qquad {} + \crr^{3} (2 \sqrt{6} (28 + 95 V_2) \sqrt{\alpha} + \sqrt{6} (618 + 1145 V_2) \alpha^{\frac{3}{2}} + 75 \sqrt{6} (12 + 29 V_2) \alpha^{\frac{5}{2}} + 1350 \sqrt{6} V_2 \alpha^{\frac{7}{2}} - 60 (-1 + V_2) \sqrt{1 + 2 \alpha} \notag\\
&\qquad {}- 120 (29 + 22 V_2) \alpha \sqrt{1 + 2 \alpha} - 90 (47 + 82 V_2) \alpha^{2} \sqrt{1 + 2 \alpha}) + \crr^{2} V_2 (2 \sqrt{6} (91- 95 V_2) \sqrt{\alpha} - \sqrt{6} (179 + 1145 V_2) \alpha^{\frac{3}{2}}  \notag\\
&\qquad {}- 75 \sqrt{6} (13 + 29 V_2) \alpha^{\frac{5}{2}} - 450 \sqrt{6} (1 + 3 V_2) \alpha^{\frac{7}{2}} + 30 (-4 + V_2) \sqrt{1 + 2 \alpha} + 30 (257 + 44 V_2) \alpha \sqrt{1 + 2 \alpha} \notag\\
&\qquad {}+ 30 (407 + 123 V_2) \alpha^{2} \sqrt{1 + 2 \alpha} + 4500 \alpha^{3} \sqrt{1 + 2 \alpha}) + \crr V_2^{2} (60 \sqrt{1 + 2 \alpha} - 3480 \alpha \sqrt{1 + 2 \alpha} - 4230 \alpha^{2} \sqrt{1 + 2 \alpha} \notag\\
&\qquad {}+ 5 \sqrt{6} V_2 \sqrt{\alpha} (1 + 3 \alpha) (11 + 20 \alpha) + \sqrt{6} \sqrt{\alpha} (-332 + \alpha (-571 + 75 \alpha (1 + 6 \alpha))))\Bigr), \\[0.5ex]
\BB_2(\alpha)
&= \tfrac{1}{6250 \alpha^{\frac{3}{2}} (1 + 2 \alpha)^{\frac{5}{2}} (1 + 3 \alpha)} 
\Bigl(2 \crr^{4} (1  + 3 \alpha) (-30 \sqrt{6} + 265 \sqrt{6} \alpha + 1775 \sqrt{6} \alpha^{2} + 2250 \sqrt{6} \alpha^{3}  - 114 \sqrt{\alpha} \sqrt{1 + 2 \alpha} - 9225 \alpha^{\frac{3}{2}} \sqrt{1 + 2 \alpha} \notag\\
&\qquad {}- 17850 \alpha^{\frac{5}{2}} \sqrt{1 + 2 \alpha}) + V_2^{3} (30 \sqrt{6} - 1945 \sqrt{6} \alpha - 6920 \sqrt{6} \alpha^{2} - 5820 \sqrt{6} \alpha^{3} + 6 (-25 + 62 V_2) \sqrt{\alpha} \sqrt{1 + 2 \alpha}  \notag\\
&\qquad {}+ 3 (-675+ 572 V_2) \alpha^{\frac{3}{2}} \sqrt{1 + 2 \alpha} + 1800 (-3 + V_2) \alpha^{\frac{5}{2}} \sqrt{1 + 2 \alpha}  - 3900 \alpha^{\frac{7}{2}} \sqrt{1 + 2 \alpha})  \notag\\
&\qquad {}+ \crr V_2^{2} (-V_2 (1 + 3 \alpha) (-60 \sqrt{6} + 530 \sqrt{6} \alpha+ 3550 \sqrt{6} \alpha^{2} + 4500 \sqrt{6} \alpha^{3} + 1788 \sqrt{\alpha} \sqrt{1 + 2 \alpha} + 6625 \alpha^{\frac{3}{2}} \sqrt{1 + 2 \alpha}  \notag\\
&\qquad {}+ 8500 \alpha^{\frac{5}{2}} \sqrt{1 + 2 \alpha}  + 2500 \alpha^{\frac{7}{2}} \sqrt{1 + 2 \alpha}) - 5 (1 + 2 \alpha) (18 \sqrt{6} - 1353 \sqrt{6} \alpha - 2221 \sqrt{6} \alpha^{2}  + 475 \sqrt{6} \alpha^{3} + 1650 \sqrt{6} \alpha^{4} \notag\\
&\qquad {}+ 180 \sqrt{\alpha} \sqrt{1 + 2 \alpha}  - 9535 \alpha^{\frac{3}{2}} \sqrt{1 + 2 \alpha} - 12850 \alpha^{\frac{5}{2}} \sqrt{1 + 2 \alpha}  - 250 \alpha^{\frac{7}{2}} \sqrt{1 + 2 \alpha} + 1500 \alpha^{\frac{9}{2}} \sqrt{1 + 2 \alpha}))  \notag\\
&\qquad {}+ \crr^{2} V_2 (5 (1+ 2 \alpha) (18 \sqrt{6} - 978 \sqrt{6} \alpha - 346 \sqrt{6} \alpha^{2} + 2725 \sqrt{6} \alpha^{3} + 1650 \sqrt{6} \alpha^{4}  + 450 \sqrt{\alpha} \sqrt{1 + 2 \alpha} - 19410 \alpha^{\frac{3}{2}} \sqrt{1 + 2 \alpha} \notag\\
&\qquad {}- 31905 \alpha^{\frac{5}{2}} \sqrt{1 + 2 \alpha} - 11250 \alpha^{\frac{7}{2}} \sqrt{1 + 2 \alpha}) + 3 V_2 (1 + 3 \alpha) (-60 \sqrt{6} + 655 \sqrt{6} \alpha + 5425 \sqrt{6} \alpha^{2} + 10500 \sqrt{6} \alpha^{3} + 5500 \sqrt{6} \alpha^{4} \notag\\
&\qquad {}+ 744 \sqrt{\alpha} \sqrt{1 + 2 \alpha} - 2125 \alpha^{\frac{3}{2}} \sqrt{1 + 2 \alpha} - 4150 \alpha^{\frac{5}{2}} \sqrt{1 + 2 \alpha} + 7500 \alpha^{\frac{7}{2}} \sqrt{1 + 2 \alpha} + 5000 \alpha^{\frac{9}{2}} \sqrt{1 + 2 \alpha}))  \notag\\
&\qquad {}+ \crr^{3} (-1200 \sqrt{\alpha} \sqrt{1 + 2 \alpha} - 588 V_2 \sqrt{\alpha} \sqrt{1 + 2 \alpha} + 40575 \alpha^{\frac{3}{2}} \sqrt{1 + 2 \alpha}  + 29711 V_2 \alpha^{\frac{3}{2}} \sqrt{1 + 2 \alpha} + 141300 \alpha^{\frac{5}{2}} \sqrt{1 + 2 \alpha} \notag\\
&\qquad {}+ 157325 V_2 \alpha^{\frac{5}{2}} \sqrt{1 + 2 \alpha} + 110700 \alpha^{\frac{7}{2}} \sqrt{1 + 2 \alpha}  + 186200 V_2 \alpha^{\frac{7}{2}} \sqrt{1 + 2 \alpha} - 7500 V_2 \alpha^{\frac{9}{2}} \sqrt{1 + 2 \alpha} - 5 \sqrt{6} (1 \notag\\
&\qquad {}+ 2 \alpha) (6 + \alpha (-26 + 3 \alpha (431 + 750 \alpha)) + 3 V_2 (1 + 3 \alpha) (-12 + 5 \alpha (31 + 5 \alpha (31 + 22 \alpha)))))\Bigl), \\[0.5ex]
\BB_3(\alpha)
&= -\tfrac{1}{6250 \alpha^{2} \sqrt{1 + 2 \alpha} (1 + \alpha (5 + 6 \alpha))} 
\Bigl(2 \crr^{4} (1  + 3 \alpha) (60 \sqrt{6} \sqrt{\alpha} - 300 \sqrt{6} \alpha^{\frac{3}{2}} - 1875 \sqrt{6} \alpha^{\frac{5}{2}}+ 18 \sqrt{1 + 2 \alpha}   \notag\\
&\qquad {}+ 75 \alpha \sqrt{1 + 2 \alpha} + 12325 \alpha^{2} \sqrt{1 + 2 \alpha})  + \crr^{2} V_2 (90 \sqrt{6} (-2  + 3 V_2) \sqrt{\alpha} + 5 \sqrt{6} (1610 - 753 V_2) \alpha^{\frac{3}{2}} + 10 \sqrt{6} (377 - 3810 V_2) \alpha^{\frac{5}{2}} \notag\\
&\qquad {}- 375 \sqrt{6} (51 + 245 V_2) \alpha^{\frac{7}{2}} - 11250 \sqrt{6} (1 + 5 V_2) \alpha^{\frac{9}{2}} + 216 V_2 \sqrt{1 + 2 \alpha}- 72 (50 + 41 V_2) \alpha \sqrt{1 + 2 \alpha}   \notag\\
&\qquad {}+ 50 (2503 - 3 V_2) \alpha^{2} \sqrt{1 + 2 \alpha} + 25 (8489  + 428 V_2) \alpha^{3} \sqrt{1 + 2 \alpha} + 3750 (19 - 23 V_2) \alpha^{4} \sqrt{1 + 2 \alpha} - 67500 V_2 \alpha^{5} \sqrt{1 + 2 \alpha})  \notag\\
&\qquad {}+ V_2^{3} (-30 \sqrt{6} (2 + V_2) \sqrt{\alpha} + 45 \sqrt{6} (80 - 17 V_2) \alpha^{\frac{3}{2}} + 5 \sqrt{6} (1143 - 605 V_2) \alpha^{\frac{5}{2}} - 500 \sqrt{6} (3 + 7 V_2) \alpha^{\frac{7}{2}} - 1500 \sqrt{6} (3 + V_2) \alpha^{\frac{9}{2}}  \notag\\
&\qquad {}+ 36 V_2 \sqrt{1 + 2 \alpha} - 642 V_2 \alpha \sqrt{1 + 2 \alpha} + 100 (41  - 10 V_2) \alpha^{2} \sqrt{1 + 2 \alpha} + 25 (267 + 350 V_2) \alpha^{3} \sqrt{1 + 2 \alpha}  \notag\\
&\qquad {}+ 2500 (1+ 8 V_2) \alpha^{4} \sqrt{1 + 2 \alpha} + 7500 (1 + 2 V_2) \alpha^{5} \sqrt{1 + 2 \alpha}) + \crr V_2^{2} (-30 \sqrt{6} (-6 + V_2) \sqrt{\alpha} + 5 \sqrt{6} (-2160 + 607 V_2) \alpha^{\frac{3}{2}} \notag\\
&\qquad {}+ 5 \sqrt{6} (-3354 + 4075 V_2) \alpha^{\frac{5}{2}}  + 125 \sqrt{6} (51 + 328 V_2) \alpha^{\frac{7}{2}} + 750 \sqrt{6} (21 + 32 V_2) \alpha^{\frac{9}{2}} - 144 V_2 \sqrt{1 + 2 \alpha} \notag\\
&\qquad {} + 24 (75 + 107 V_2) \alpha \sqrt{1 + 2 \alpha} + 25 (-2699 + 515 V_2) \alpha^{2} \sqrt{1 + 2 \alpha} + 375 (-278+ 11 V_2) \alpha^{3} \sqrt{1 + 2 \alpha} \notag\\
&\qquad {} - 1250 (3 + 28 V_2) \alpha^{4} \sqrt{1 + 2 \alpha} - 7500 (-3 + 5 V_2) \alpha^{5} \sqrt{1 + 2 \alpha}) + \crr^{3} (-144 V_2 \sqrt{1 + 2 \alpha} + 1800 \alpha \sqrt{1 + 2 \alpha} \notag\\
&\qquad {}+ 768 V_2 \alpha \sqrt{1 + 2 \alpha} - 54275 \alpha^{2} \sqrt{1 + 2 \alpha} - 39325 V_2 \alpha^{2} \sqrt{1 + 2 \alpha}  - 72150 \alpha^{3} \sqrt{1 + 2 \alpha} - 125025 V_2 \alpha^{3} \sqrt{1 + 2 \alpha}  \notag\\
&\qquad {} + 11250 V_2 \alpha^{4} \sqrt{1 + 2 \alpha} + 5 \sqrt{6} \sqrt{\alpha} (12 + \alpha (-170 + \alpha (1457 + 2850 \alpha)) + V_2 (1 + 3 \alpha) (-66 + 5 \alpha (109 + 25 \alpha (29 + 18 \alpha)))))\Bigr), \\[0.5ex]
\BB_4(\alpha)
&= \tfrac{1}{625 \alpha^{\frac{3}{2}} \sqrt{1 + 2 \alpha} (1 + \alpha (5 + 6 \alpha))} 
\Bigl(2 \crr^{4} (1 + 3 \alpha) (-6 \sqrt{6} + 5 \alpha (7 \sqrt{6} + 55 \sqrt{6} \alpha - 327 \sqrt{\alpha} \sqrt{1 + 2 \alpha})) \notag\\
&\qquad {}+ V_2^{3} (90 \sqrt{\alpha} \sqrt{1 + 2 \alpha} + 210 V_2 \sqrt{\alpha} \sqrt{1 + 2 \alpha}  - 605 \alpha^{\frac{3}{2}} \sqrt{1 + 2 \alpha} + 545 V_2 \alpha^{\frac{3}{2}} \sqrt{1 + 2 \alpha} - 1500 \alpha^{\frac{5}{2}} \sqrt{1 + 2 \alpha} \notag\\
&\qquad {} - 3005 V_2 \alpha^{\frac{5}{2}} \sqrt{1 + 2 \alpha}  - 1000 \alpha^{\frac{7}{2}} \sqrt{1 + 2 \alpha} - 11250 V_2 \alpha^{\frac{7}{2}} \sqrt{1 + 2 \alpha} - 3000 \alpha^{\frac{9}{2}} \sqrt{1 + 2 \alpha} - 9000 V_2 \alpha^{\frac{9}{2}} \sqrt{1 + 2 \alpha}  \notag\\
&\qquad {}+ \sqrt{6} V_2 (1 + 2 \alpha) (1 + 3 \alpha) (-18 + 25 \alpha (7 + 6 \alpha)) + \sqrt{6} (6 + \alpha (-613 + 12 \alpha (-89 + 50 \alpha (1 + 3 \alpha))))) \notag\\
&\qquad {}+ \crr^{2} V_2 (570 \sqrt{\alpha} \sqrt{1 + 2 \alpha} + 330 V_2 \sqrt{\alpha} \sqrt{1 + 2 \alpha}  
- 16095 \alpha^{\frac{3}{2}} \sqrt{1 + 2 \alpha} - 865 V_2 \alpha^{\frac{3}{2}} \sqrt{1 + 2 \alpha} - 28290 \alpha^{\frac{5}{2}} \sqrt{1 + 2 \alpha}  \notag\\
&\qquad {}- 1565 V_2 \alpha^{\frac{5}{2}} \sqrt{1 + 2 \alpha}  - 9000 \alpha^{\frac{7}{2}} \sqrt{1 + 2 \alpha} + 16500 V_2 \alpha^{\frac{7}{2}} \sqrt{1 + 2 \alpha} + 13500 V_2 \alpha^{\frac{9}{2}} \sqrt{1 + 2 \alpha} \notag\\
&\qquad {}+ \sqrt{6} V_2 (1 + 3 \alpha) (-90 + \alpha (727 + 725 \alpha (7 + 6 \alpha)))  + \sqrt{6} (18 + \alpha (-1139 + \alpha (-379 + 75 \alpha (35 + 18 \alpha)))))  \notag\\
&\qquad {}+ \crr V_2^{2} (-420 \sqrt{\alpha} \sqrt{1 + 2 \alpha} - 480 V_2 \sqrt{\alpha} \sqrt{1 + 2 \alpha} + 9330 \alpha^{\frac{3}{2}} \sqrt{1 + 2 \alpha} - 1895 V_2 \alpha^{\frac{3}{2}} \sqrt{1 + 2 \alpha} + 16770 \alpha^{\frac{5}{2}} \sqrt{1 + 2 \alpha} \notag\\
&\qquad {}+ 2385 V_2 \alpha^{\frac{5}{2}} \sqrt{1 + 2 \alpha} + 750 \alpha^{\frac{7}{2}} \sqrt{1 + 2 \alpha}  + 16250 V_2 \alpha^{\frac{7}{2}} \sqrt{1 + 2 \alpha} - 4500 \alpha^{\frac{9}{2}} \sqrt{1 + 2 \alpha} \notag\\
&\qquad {}+ 15000 V_2 \alpha^{\frac{9}{2}} \sqrt{1 + 2 \alpha} - \sqrt{6} V_2 (1 + 3 \alpha) (-66 + \alpha (637 + 50 \alpha (71 + 66 \alpha))) \notag\\
&\qquad {}+ \sqrt{6} (-18 + \alpha (1589 + \alpha (2329 - 75 \alpha (19 + 42 \alpha))))) + \crr^{3} (-240 \sqrt{\alpha} \sqrt{1 + 2 \alpha}  - 60 V_2 \sqrt{\alpha} \sqrt{1 + 2 \alpha} + 6870 \alpha^{\frac{3}{2}} \sqrt{1 + 2 \alpha} \notag\\
&\qquad {}+ 5985 V_2 \alpha^{\frac{3}{2}} \sqrt{1 + 2 \alpha} + 9270 \alpha^{\frac{5}{2}} \sqrt{1 + 2 \alpha} 
+ 17745 V_2 \alpha^{\frac{5}{2}} \sqrt{1 + 2 \alpha} - 2250 V_2 \alpha^{\frac{7}{2}} \sqrt{1 + 2 \alpha} \notag\\
&\qquad {}- \sqrt{6} (6 + \alpha (-163 + 18 \alpha (49 + 100 \alpha)) + V_2 (1 + 3 \alpha) (-54 + \alpha (299 + 25 \alpha (103 + 54 \alpha)))))\Bigr), \\[0.5ex]
\BB_5(\alpha)&
= \tfrac{1}{25 \alpha (1 + 2 \alpha)^{\frac{3}{2}} (1 + 3 \alpha)} \Bigl(\crr^{4} \sqrt{\alpha} (1 + 3 \alpha) (5 \sqrt{6} + 30 \sqrt{6} \alpha - 168 \sqrt{\alpha} \sqrt{1 + 2 \alpha}) + \crr^{2} V_2 (15 \sqrt{6} (-4 + 5 V_2) \sqrt{\alpha}  \notag\\
&\qquad {}+ 5 \sqrt{6} (-1 + 123 V_2) \alpha^{\frac{3}{2}} + 10 \sqrt{6} (14 + 153 V_2) \alpha^{\frac{5}{2}} + 60 \sqrt{6} (1 + 18 V_2) \alpha^{\frac{7}{2}} \notag\\
&\qquad {} - 36 (-1 + V_2) \sqrt{1 + 2 \alpha} - 2 (408 + 191 V_2) \alpha \sqrt{1 + 2 \alpha}   - (1497 + 572 V_2) \alpha^{2} \sqrt{1 + 2 \alpha}   \notag\\
&\qquad {}+ 150 (-3 + 7 V_2) \alpha^{3} \sqrt{1 + 2 \alpha} + 900 V_2 \alpha^{4} \sqrt{1 + 2 \alpha}) + \crr V_2^{2} (5 \sqrt{6} (18- 19 V_2) \sqrt{\alpha}  + 5 \sqrt{6} (25 - 149 V_2) \alpha^{\frac{3}{2}}  \notag\\
&\qquad {}- 10 \sqrt{6} (11 + 188 V_2) \alpha^{\frac{5}{2}} - 60 \sqrt{6} (4 + 25 V_2) \alpha^{\frac{7}{2}} + 36 (-1 + V_2) \sqrt{1 + 2 \alpha} + 7 (71 + 45 V_2) \alpha \sqrt{1 + 2 \alpha} \notag\\
&\qquad {}+ (1040 + 1021 V_2) \alpha^{2} \sqrt{1 + 2 \alpha} + 50 (1 + 34 V_2) \alpha^{3} \sqrt{1 + 2 \alpha} + 300 (-1+ 5 V_2) \alpha^{4} \sqrt{1 + 2 \alpha})  \notag\\
&\qquad {}+ V_2^{3} (40 \sqrt{6} (-1 + V_2) \sqrt{\alpha} + 75 \sqrt{6} (-1  + 4 V_2) \alpha^{\frac{3}{2}} + 20 \sqrt{6} (3 + 37 V_2) \alpha^{\frac{5}{2}} + 60 \sqrt{6} (3 + 10 V_2) \alpha^{\frac{7}{2}} \notag\\
&\qquad {}- 12 (-1 + V_2) \sqrt{1 + 2 \alpha} - 2 (13 + 86 V_2) \alpha \sqrt{1 + 2 \alpha} - (111 + 908 V_2) \alpha^{2} \sqrt{1 + 2 \alpha} \notag\\
&\qquad {}- 100 (1 + 21 V_2) \alpha^{3} \sqrt{1 + 2 \alpha} - 300 (1 + 6 V_2) \alpha^{4} \sqrt{1 + 2 \alpha}) + \crr^{3} (-12 \sqrt{1 + 2 \alpha} \notag\\
&\qquad {}+ 12 V_2 \sqrt{1 + 2 \alpha} + 345 \alpha \sqrt{1 + 2 \alpha} + 407 V_2 \alpha \sqrt{1 + 2 \alpha} + 468 \alpha^{2} \sqrt{1 + 2 \alpha} \notag\\
&\qquad {}+ 1063 V_2 \alpha^{2} \sqrt{1 + 2 \alpha} - 150 V_2 \alpha^{3} \sqrt{1 + 2 \alpha} - 5 \sqrt{6} \sqrt{\alpha} (-2 + 9 \alpha + 18 \alpha^{2} + V_2 (1 + 3 \alpha) (5 + 4 \alpha (7 + 3 \alpha))))\Bigr).
\end{align}
\end{subequations}
\endgroup
We also introduce the rescaled coefficients (which will be used in Appendix~\ref{app:CA}):
\begin{subequations}\label{Bernstein:coeff:exp:rescaled}
\begin{align}
\tilde{\BB}_0(\alpha)
&:=-\tfrac{25\alpha^{\frac{1}{2}}(1+2\alpha)(1+3\alpha)}
{\crr}\BB_0(\alpha),\\
\tilde{\BB}_1(\alpha)
&:=-625\alpha(1+2\alpha)^{\frac{3}{2}}(1+3\alpha)
\BB_1(\alpha),\\
\tilde{\BB}_2(\alpha)
&:=-6250\alpha^{\frac{3}{2}}(1+2\alpha)^{\frac{5}{2}}(1+3\alpha)
\BB_2(\alpha),\\
\tilde{\BB}_3(\alpha)
&:=-6250\alpha^2(1+2\alpha)^{\frac{3}{2}}(1+3\alpha)
\BB_3(\alpha),\\
\tilde{\BB}_4(\alpha)
&:=-625\alpha^{\frac{3}{2}}(1+2\alpha)^{\frac{3}{2}}(1+3\alpha)
\BB_4(\alpha),\\
\tilde{\BB}_5(\alpha)
&:=-25\alpha(1+2\alpha)^{\frac{3}{2}}(1+3\alpha)
\BB_5(\alpha).
\end{align}
\end{subequations}

\subsection{Elementary proof of Lemma~\ref{lemma:Bernstein:q} for monatomic and diatomic gases} \label{app:elementary}
\begin{proof}
We now prove Lemma~\ref{lemma:Bernstein:q} by elementary, though somewhat lengthy, exact computations in two physically important cases: monatomic gases
$(\gamma=\frac{5}{3}$, equivalently $\alpha=\frac{1}{3}$) and diatomic gases
$(\gamma=\frac{7}{5}$, equivalently $\alpha=\frac{1}{5}$). The idea is to substitute the exact values of $\crr$ and $V_2$ into \eqref{coeff:bb:i}, then specialize to $\alpha=\frac{1}{5}$ and $\alpha=\frac{1}{3}$. This gives explicit algebraic expressions for the coefficients $\BB_j$. We then obtain rational bounds for the radicals and check, exactly, that each $\BB_j$ is negative.

\textit{Monatomic gas, $\alpha=\frac13$.} From
\eqref{formula:NNN:1} and \eqref{def:V2}, for $\alpha=\frac{1}{3}$ we have
\[
\crr=\tfrac{6+\sqrt{141}}{12},\qquad
V_2=\tfrac{40+3\sqrt{141}-\sqrt{4229}}{80}.
\]
The resulting Bernstein coefficients are
{\footnotesize
\allowdisplaybreaks
\begin{align*}
\BB_0(\tfrac13)
&=\begin{aligned}[t]
&-\tfrac{6+\sqrt{141}}{28800000}
\bigl(-1952991+155100\sqrt{30}+291992\sqrt{141}
+6600\sqrt{30}\sqrt{141}\\
&\qquad\qquad\qquad \quad
{}+9816\sqrt{4229}+1800\sqrt{30}\sqrt{4229}
-6823\sqrt{141}\sqrt{4229}
+300\sqrt{30}\sqrt{141}\sqrt{4229}\bigr),
\end{aligned}\\
\BB_1(\tfrac13)
&=\begin{aligned}[t]
&-\tfrac{1}{72000000000}
\bigl(120605148000-501746055\sqrt{30}
+1877373000\sqrt{141}\\
&\qquad\qquad\qquad \qquad
{}+7194282\sqrt{30}\sqrt{141}-2290671000\sqrt{4229}
+116329986\sqrt{30}\sqrt{4229}\\
&\qquad\qquad\qquad \qquad
{}-54756000\sqrt{141}\sqrt{4229}
+8412585\sqrt{30}\sqrt{141}\sqrt{4229}\bigr),
\end{aligned}\\
\BB_2(\tfrac13)
&=\begin{aligned}[t]
&-\tfrac{1}{3600000000000}
\bigl(6021730894257-103611339600\sqrt{30}
+50632645000\sqrt{141}\\
&\qquad\qquad\qquad \qquad
{}-14354046150\sqrt{30}\sqrt{141}
-85715415000\sqrt{4229}
+5811076050\sqrt{30}\sqrt{4229}\\
&\qquad\qquad\qquad \qquad
{}+457824721\sqrt{141}\sqrt{4229}
+551521200\sqrt{30}\sqrt{141}\sqrt{4229}\bigr),
\end{aligned}\\
\BB_3(\tfrac13)
&=\begin{aligned}[t]
&-\tfrac{1}{1440000000000}
\bigl(1092494981118-38352798063\sqrt{30}
-168051614000\sqrt{141}\\
&\qquad\qquad\qquad \qquad
{}+2401202400\sqrt{30}\sqrt{141}
-3947622000\sqrt{4229}
+2234455200\sqrt{30}\sqrt{4229}\\
&\qquad\qquad\qquad \qquad
{}+4052182654\sqrt{141}\sqrt{4229}
+76915761\sqrt{30}\sqrt{141}\sqrt{4229}\bigr),
\end{aligned}\\
\BB_4(\tfrac13)
&=\begin{aligned}[t]
&-\tfrac{1}{28800000000}
\bigl(-22169903823+457134285\sqrt{30}
-9666775000\sqrt{141}\\
&\qquad\qquad\qquad \qquad
{}+473814000\sqrt{30}\sqrt{141}
+662925000\sqrt{4229}
+3222000\sqrt{30}\sqrt{4229}\\
&\qquad\qquad\qquad \qquad
{}+184222481\sqrt{141}\sqrt{4229}
-7837395\sqrt{30}\sqrt{141}\sqrt{4229}\bigr),
\end{aligned}\\
\BB_5(\tfrac13)
&=\begin{aligned}[t]
&-\tfrac{1}{2880000000}
\bigl(21083761647-2187904305\sqrt{30}
-667805000\sqrt{141}\\
&\qquad\qquad\qquad \qquad
{}+134044800\sqrt{30}\sqrt{141}
-263265000\sqrt{4229}
+26630400\sqrt{30}\sqrt{4229}\\
&\qquad\qquad\qquad \qquad
{}+11553391\sqrt{141}\sqrt{4229}
-2449665\sqrt{30}\sqrt{141}\sqrt{4229}\bigr).
\end{aligned}
\end{align*}}
The bounds
\[
\tfrac{109}{20}<\sqrt{30}<\tfrac{11}{2},\qquad
\tfrac{237}{20}<\sqrt{141}<\tfrac{119}{10},\qquad
65<\sqrt{4229}<\tfrac{1301}{20}
\]
are easily checked by squaring. Substituting lower bounds for positive terms and upper bounds for negative terms gives lower bounds for the numerators; since the prefactors in the displayed $\BB_j$ formulas are negative, this yields
\[
\begin{alignedat}{2}
\BB_0(\tfrac13)&<-\tfrac{6410453(6+\sqrt{141})}{5760000000}<0,\qquad&
\BB_1(\tfrac13)&<-\tfrac{3424975697977}{9600000000000}<0,\\
\BB_2(\tfrac13)&<-\tfrac{5684038198583}{4800000000000}<0,\qquad&
\BB_3(\tfrac13)&<-\tfrac{241252356853909}{115200000000000}<0,\\
\BB_4(\tfrac13)&<-\tfrac{3891835769729}{2304000000000}<0,\qquad&
\BB_5(\tfrac13)&<-\tfrac{43062662263}{230400000000}<0.
\end{alignedat}
\]

\textit{Diatomic gas, $\alpha=\frac{1}{5}$.}
From \eqref{formula:NNN:1} and \eqref{def:V2}, for $\alpha=\frac15$ we have
\[
\crr=\tfrac{25}{16},\qquad
V_2=\tfrac{473-\sqrt{139729}}{448}.
\]
The resulting Bernstein coefficients are
{\footnotesize
\begin{align*}
\BB_0(\tfrac15)
&=-\tfrac{5(15083317+3178000\sqrt{6}\sqrt{7}-172229\sqrt{139729}
+14000\sqrt{6}\sqrt{7}\sqrt{139729})}{89915392},\\
\BB_1(\tfrac15)
&=-\tfrac{156129187200\sqrt{7}
+324273959\sqrt{6}\sqrt{139729}
-7129546607\sqrt{6}-503126400\sqrt{7}\sqrt{139729}}
{62940774400\sqrt{7}},\\
\BB_2(\tfrac15)
&=-\tfrac{17882225403157
+5805032478\sqrt{6}\sqrt{7}\sqrt{139729}
-441600050094\sqrt{6}\sqrt{7}-24647891309\sqrt{139729}}
{6168195891200},\\
\BB_3(\tfrac15)
&=-\tfrac{817820980253\sqrt{6}+1533157541838\sqrt{7}
+9200520139\sqrt{6}\sqrt{139729}
+10951568994\sqrt{7}\sqrt{139729}}
{1762341683200\sqrt{7}},\\
\BB_4(\tfrac15)
&=-\tfrac{432907974373\sqrt{6}-661583007905\sqrt{7}
-714882301\sqrt{6}\sqrt{139729}
+3207984985\sqrt{7}\sqrt{139729}}
{176234168320\sqrt{7}},\\
\BB_5(\tfrac15)
&=-\tfrac{3(169719216703\sqrt{7}-59859899911\sqrt{6}
-14989793\sqrt{6}\sqrt{139729}
-275553511\sqrt{7}\sqrt{139729})}
{88117084160\sqrt{7}}.
\end{align*}}
The bounds
\[
\tfrac{61}{25}<\sqrt6<\tfrac{49}{20},\qquad
\tfrac{66}{25}<\sqrt7<\tfrac{53}{20},\qquad
\tfrac{1869}{5}<\sqrt{139729}<\tfrac{37381}{100}
\]
are easily checked by squaring.  Substituting lower bounds for positive terms and upper bounds for negative terms gives lower bounds for the numerators; since the prefactors in the displayed $\BB_j$ formulas are negative, this yields
\[
\begin{alignedat}{2}
\BB_0(\tfrac15)&<-\tfrac{488398043}{1798307840}<0,\qquad&
\BB_1(\tfrac15)&<-\tfrac{96039806228149}{31470387200000\sqrt7}<0,\\
\BB_2(\tfrac15)&<-\tfrac{494481821122721831}{154204897280000000}<0,\qquad&
\BB_3(\tfrac15)&<-\tfrac{3155239329996832}{220292710400000\sqrt7}<0,\\
\BB_4(\tfrac15)&<-\tfrac{3628257701088411}{352468336640000\sqrt7}<0,\qquad&
\BB_5(\tfrac15)&<-\tfrac{220671962661}{440585420800\sqrt7}<0.
\end{alignedat}
\]
The bound for $\BB_3(\frac15)$ is immediate from the displayed formula, since
all terms in its numerator have the same sign.
\end{proof}


\section{The sign of Bernstein coefficients for the explosion barrier: rational-cover proof}
\label{app:rational-cover-bernstein-q}

This appendix gives an exact rational-cover proof of the bound \eqref{bound:BB} in Lemma~\ref{lemma:Bernstein:q}.  The proof is implemented in the accompanying file \path{verify_bernstein_q_rational_cover.py}, available in the folder \path{rational_cover} at \url{https://github.com/giorgiocialdea/new_class_explosions_ca}.  The script starts from the vector field \eqref{P:exp}, the barrier \eqref{def:F}, and the polynomial $\QQ$ defined by \eqref{def:b:ex}--\eqref{def:q}; it then computes the coefficients in \eqref{coeff:QQ} and \eqref{Bernstein:coeff:exp}.

We introduce the following auxiliary quantities:
\begin{equation} \label{eq:appD:sf-def}
    \mathsf c=(1+3\alpha)\crr,
    \qquad
    \mathsf v=(1+3\alpha)V_2,
    \qquad
    \rr=-\tfrac15\sqrt{\tfrac{6\alpha}{1+2\alpha}}<0.
\end{equation}
Thus, by \eqref{Q2:def} and \eqref{eq:appD:sf-def}, we have
\begin{equation}
    Q_2=\tfrac{\mathsf c-\mathsf v}{\alpha(1+3\alpha)}.
    \label{eq:appD:Q2-sf}
\end{equation}

The three quantities in \eqref{eq:appD:sf-def} are solutions of the following quadratic polynomials
\begin{subequations}
\label{eq:appD:sf-relations}
\begin{align}
4\mathsf c^2+(6\alpha-10)\mathsf c-60\alpha^2-33\alpha+6&=0,
    \label{eq:appD:sf-c-rel}\\
(4\alpha+2)\mathsf v^2
  +(-6\alpha\mathsf c-\mathsf c+6\alpha^2-\alpha-4)\mathsf v
  -2\mathsf c^2+6\alpha\mathsf c+5\mathsf c&=0,
    \label{eq:appD:sf-v-rel}\\
25(1+2\alpha)\rr^2-6\alpha&=0.
    \label{eq:appD:rr-rel}
\end{align}
\end{subequations}
Here \eqref{eq:appD:sf-c-rel} follows from the equation from $\crr$ in \eqref{eq:cr:1}; \eqref{eq:appD:sf-v-rel} is \eqref{triple:eq}, rewritten using \eqref{Vinfty:def} and \eqref{eq:appD:sf-def}; and \eqref{eq:appD:rr-rel} follows from the definition of $\rr$ in \eqref{upp:bound:imp}. We also recall that we have $\mathsf c>0$, $\mathsf v\ge0$, and $\rr<0$.  The case $\mathsf v=0$ occurs only at $\alpha=1+\sqrt2$.

Let $\BB_j$, $0\le j\le5$, be the Bernstein coefficients of $\QQ$ defined in \eqref{Bernstein:coeff:exp}.  Starting from \eqref{P:exp}, \eqref{def:F}, and \eqref{def:b:ex}--\eqref{def:q}, the script substitutes \eqref{eq:appD:sf-def} and reduces all powers $\mathsf c^2$, $\mathsf v^2$, and $\rr^2$ by \eqref{eq:appD:sf-relations}.  It obtains polynomials $\mathsf L_j=\mathsf L_j(\alpha,\mathsf c,\mathsf v,\rr)$, such that
\begin{equation}
    D_j(\alpha)(-\BB_j(\alpha))
    =\mathsf L_j(\alpha,\mathsf c(\alpha),\mathsf v(\alpha),\rr(\alpha)),
    \qquad 0\le j\le5,
    \label{eq:appD:L-def}
\end{equation}
where, recalling that $\gamma=1+2\alpha$,
\begin{equation}
\label{eq:appD:Dj}
\begin{alignedat}{3}
D_0(\alpha)&=200\alpha\gamma^2(1+3\alpha)^5,&\qquad
D_1(\alpha)&=20000\alpha^2\gamma^4(1+3\alpha)^5,&\qquad
D_2(\alpha)&=200000\alpha^3\gamma^7(1+3\alpha)^5,\\
D_3(\alpha)&=400000\alpha^4\gamma^4(1+3\alpha)^5,&
D_4(\alpha)&=40000\alpha^3\gamma^6(1+3\alpha)^5,&
D_5(\alpha)&=800\alpha^2\gamma^6(1+3\alpha)^5.
\end{alignedat}
\end{equation}
Each $D_j$ is positive for $\alpha>0$.  Therefore it remains to prove
\begin{equation}
    \mathsf L_j(\alpha,\mathsf c(\alpha),\mathsf v(\alpha),\rr(\alpha))>0,
    \qquad 0\le j\le5.
    \label{eq:appD:L-positive-goal}
\end{equation}
Each $ \mathsf L_j$ has an affine structure in the variables $\mathsf c, \mathsf v, \rr$
\begin{equation}
    \mathsf L_j(\alpha,\mathsf c,\mathsf v,\rr)
    =\sum_{i,k,l\in\{0,1\}}
    p^{(j)}_{i k l}(\alpha)
    \mathsf c^i \mathsf v^k \rr^l,
    \label{eq:appD:multi-affine}
\end{equation}
where each $p^{(j)}_{i k l}$  is a polynomial in  $\alpha$ with rational coefficients.
Consequently, once $\alpha$ is restricted to an interval and $(\mathsf c,\mathsf v,\rr)$ is enclosed in a box $[\mathsf c_{I,-}, \mathsf c_{I,+}] \times [\mathsf v_{I,-}, \mathsf v_{I,+}] \times[\mathsf \rr_{I,-}, \mathsf \rr_{I,+}]  $, the minimum over such a box is attained at one of its eight corners  $(\mathsf c,\mathsf v,\rr) \in \{ \mathsf c_{I,-}, \mathsf c_{I,+}\} \times \{\mathsf v_{I,-}, \mathsf v_{I,+}\} \times \{\mathsf \rr_{I,-}, \mathsf \rr_{I,+}\} $.  
\subsection{The rational cover of $\alpha$: $0<\alpha \le\frac{1}{100}$}
  For very small $\alpha$, set $x=\sqrt\alpha$.  On $0<x\le\frac{1}{10}$,
\begin{equation}
    \tfrac32<\mathsf c<\tfrac32+13x^2,
    \qquad
    \tfrac34-7x^2<\mathsf v<\tfrac34,
    \qquad
    -\tfrac12x<\rr<-\tfrac{12}{25}x.
    \label{eq:appD:small-box}
\end{equation}
The bounds for $\mathsf c$ follow from \eqref{eq:cr:1}, namely
\begin{equation*}
    \mathsf c=\tfrac{5-3x^2+\sqrt{1+102x^2+249x^4}}4,
\end{equation*}
and the inequalities $1+3x^2<\sqrt{1+102x^2+249x^4}<1+55x^2$.  The bounds for $\rr$ follow directly from \eqref{eq:appD:rr-rel}.  The bounds for $\mathsf v$ follow by substituting $\mathsf v=\frac{3}{4}$ and $\mathsf v=\frac{3}{4}-7x^2$ into \eqref{eq:appD:sf-v-rel}, using the preceding bounds for $\mathsf c$, and checking the resulting one-variable signs on $0<x\le \frac{1}{10}$.

For each corner of the box \eqref{eq:appD:small-box}, the script writes
\begin{equation*}
    \mathsf L_j(x^2,\mathsf c,\mathsf v,\rr)=x^{m_j}\mathsf S_j(x).
\end{equation*}
It then converts each polynomial $\mathsf S_j$ to Bernstein form on $[0,\frac{1}{10}]$.  The next table gives, for each $j$, the common exponent $m_j$ and the smallest Bernstein coefficient among all eight corners.
\begin{center}
\begin{tabular}{c|c|c}
$j$ & $m_j$ & smallest Bernstein coefficient of $\mathsf S_j$ \\ \hline
0 & 1 & $\frac{46617633}{291200}$ \\ 
1 & 2 & $2835$ \\ 
2 & 3 & $3888$ \\ 
3 & 4 & $729$ \\ 
4 & 3 & $\frac{3402}{5}$ \\ 
5 & 2 & $\frac{2809156631}{70840000}$ \\ 
\end{tabular}
\end{center}
All entries in the last column are positive.  Since $x>0$, this proves \eqref{eq:appD:L-positive-goal} for $0<\alpha\le \frac{1}{100}$.
\subsection{The rational cover of $\alpha$: $\frac{1}{100} \le \alpha \le 1 +\sqrt{2}$.}
For the remaining range we use a rational cover of the full interval
$[10^{-2},1+\sqrt2]$.   If $I=[a_-,a_+]$ is one of the intervals in the table below, we define the six-decimal rational enclosures by
\begin{align*}
\mathsf c_{I,-}&:=10^{-6}\lfloor10^6\mathsf c(a_-)\rfloor,&
\mathsf c_{I,+}&:=10^{-6}\lceil10^6\mathsf c(a_+)\rceil,\\
\mathsf v_{I,-}&:=10^{-6}\lfloor10^6\mathsf v(a_+)\rfloor,&
\mathsf v_{I,+}&:=10^{-6}\lceil10^6\mathsf v(a_-)\rceil,\\
\rr_{I,-}&:=-10^{-6}\lceil10^6|\rr(a_+)|\rceil,&
\rr_{I,+}&:=-10^{-6}\lfloor10^6|\rr(a_-)|\rfloor.
\end{align*}
On the last interval, whose right endpoint is $2.4143>1+\sqrt2$, the lower bound for $\mathsf v$ is taken to be $0$, the value at the true endpoint. The definitions show that $\mathsf c$ is increasing, while $\mathsf v$ and $\rr$ are decreasing.  For example, $\mathsf c' >0$ is equivalent after squaring to
\begin{equation*}
(102+498\alpha)^2-36(1+102\alpha+249\alpha^2)
=576(415\alpha^2+170\alpha+18)>0,
\end{equation*}
and the monotonicity of $\rr$ is immediate from \eqref{eq:appD:rr-rel}; the monotonicity of $\mathsf v$ follows by differentiating \eqref{eq:appD:sf-v-rel}.  Therefore, for $\alpha\in I$,
\begin{equation*}
\mathsf c_{I,-}\le\mathsf c(\alpha)\le\mathsf c_{I,+},
\qquad
\mathsf v_{I,-}\le\mathsf v(\alpha)\le\mathsf v_{I,+},
\qquad
\rr_{I,-}\le\rr(\alpha)\le\rr_{I,+}<0.
\end{equation*}

For every interval $I$, every $j$, and every corner of the corresponding box, the corner substitution in $\mathsf L_j$ gives a polynomial in $\alpha$.  After the affine change of variables $\alpha=a_-+(a_+-a_-)s$, the script converts this polynomial to the
Bernstein basis on $0\le  s \le1$.  The number $\mu_I$ in the last column is the minimum of all Bernstein coefficients obtained from the six values of $j$ and the eight corners, rounded downward.  Thus each displayed $\mu_I$ is an exact rational lower bound.

\begingroup
\scriptsize
\setlength{\tabcolsep}{3pt}
\begin{center}
\begin{tabular}{r@{ }c@{ }r@{\hspace{0.8em}}|@{\hspace{0.8em}}r@{ }c@{ }r@{\hspace{0.8em}}|@{\hspace{0.8em}}r@{ }c@{ }r}
$\#$ & $I$ & $\mu_I$ & $\#$ & $I$ & $\mu_I$ & $\#$ & $I$ & $\mu_I$ \\ \hline
1 & $[0.0100,0.0476]$ & $0.0683$ & 27 & $[0.5637,0.5832]$ & $1.2020$ & 53 & $[1.5981,1.6358]$ & $1202750.1$\\
2 & $[0.0476,0.0997]$ & $0.0009$ & 28 & $[0.5832,0.6042]$ & $3.6619$ & 54 & $[1.6358,1.6728]$ & $1430234.2$\\
3 & $[0.0997,0.1233]$ & $0.1098$ & 29 & $[0.6042,0.6269]$ & $8.4783$ & 55 & $[1.6728,1.7091]$ & $1687089.5$\\
4 & $[0.1233,0.1436]$ & $0.1149$ & 30 & $[0.6269,0.6516]$ & $2.6067$ & 56 & $[1.7091,1.7447]$ & $1975261.2$\\
5 & $[0.1436,0.1638]$ & $0.0770$ & 31 & $[0.6516,0.6785]$ & $3.1022$ & 57 & $[1.7447,1.7797]$ & $2293900.2$\\
6 & $[0.1638,0.1858]$ & $0.4428$ & 32 & $[0.6785,0.7078]$ & $15.0907$ & 58 & $[1.7797,1.8141]$ & $2648053.4$\\
7 & $[0.1858,0.2116]$ & $0.2924$ & 33 & $[0.7078,0.7398]$ & $21.0283$ & 59 & $[1.8141,1.8480]$ & $3036110.0$\\
8 & $[0.2116,0.2435]$ & $1.4567$ & 34 & $[0.7398,0.7748]$ & $20.0246$ & 60 & $[1.8480,1.8814]$ & $3464231.3$\\
9 & $[0.2435,0.2765]$ & $0.2080$ & 35 & $[0.7748,0.8131]$ & $9.4483$ & 61 & $[1.8814,1.9143]$ & $3934438.4$\\
10 & $[0.2765,0.3047]$ & $0.2478$ & 36 & $[0.8131,0.8549]$ & $44.1293$ & 62 & $[1.9143,1.9467]$ & $4449091.1$\\
11 & $[0.3047,0.3294]$ & $0.6852$ & 37 & $[0.8549,0.9006]$ & $18.7492$ & 63 & $[1.9467,1.9787]$ & $5005197.2$\\
12 & $[0.3294,0.3517]$ & $1.1168$ & 38 & $[0.9006,0.9504]$ & $73.9324$ & 64 & $[1.9787,2.0103]$ & $4818939.9$\\
13 & $[0.3517,0.3726]$ & $0.4054$ & 39 & $[0.9504,1.0000]$ & $5655.3$ & 65 & $[2.0103,2.0415]$ & $6268539.8$\\
14 & $[0.3726,0.3919]$ & $0.9711$ & 40 & $[1.0000,1.0569]$ & $2756.4$ & 66 & $[2.0415,2.0723]$ & $6980652.3$\\
15 & $[0.3919,0.4091]$ & $0.4949$ & 41 & $[1.0569,1.1111]$ & $19195.5$ & 67 & $[2.0723,2.1027]$ & $7749602.0$\\
16 & $[0.4091,0.4247]$ & $0.8540$ & 42 & $[1.1111,1.1630]$ & $42932.6$ & 68 & $[2.1027,2.1328]$ & $8570043.5$\\
17 & $[0.4247,0.4394]$ & $1.9542$ & 43 & $[1.1630,1.2129]$ & $75303.0$ & 69 & $[2.1328,2.1625]$ & $9461887.4$\\
18 & $[0.4394,0.4537]$ & $2.1787$ & 44 & $[1.2129,1.2610]$ & $117843.7$ & 70 & $[2.1625,2.1919]$ & $10408735.5$\\
19 & $[0.4537,0.4679]$ & $3.0942$ & 45 & $[1.2610,1.3075]$ & $116202.2$ & 71 & $[2.1919,2.2210]$ & $11424047.2$\\
20 & $[0.4679,0.4823]$ & $1.9909$ & 46 & $[1.3075,1.3525]$ & $239094.2$ & 72 & $[2.2210,2.2498]$ & $12510434.0$\\
21 & $[0.4823,0.4971]$ & $1.3652$ & 47 & $[1.3525,1.3962]$ & $320425.6$ & 73 & $[2.2498,2.2783]$ & $13670815.6$\\
22 & $[0.4971,0.5124]$ & $5.1017$ & 48 & $[1.3962,1.4386]$ & $418745.8$ & 74 & $[2.2783,2.3066]$ & $14895849.3$\\
23 & $[0.5124,0.5285]$ & $2.8464$ & 49 & $[1.4386,1.4799]$ & $533725.0$ & 75 & $[2.3066,2.3346]$ & $16215917.4$\\
24 & $[0.5285,0.5455]$ & $5.5932$ & 50 & $[1.4799,1.5202]$ & $667726.0$ & 76 & $[2.3346,2.3624]$ & $17604877.0$\\
25 & $[0.5455,0.5546]$ & $736.6$ & 51 & $[1.5202,1.5596]$ & $822350.3$ & 77 & $[2.3624,2.3899]$ & $19096815.1$\\
26 & $[0.5546,0.5637]$ & $852.0$ & 52 & $[1.5596,1.5981]$ & $1000677.6$ & 78 & $[2.3899,2.4143]$ & $21108988.8$\\
\end{tabular}
\end{center}
\endgroup
Since every $\mu_I$ is positive, \eqref{eq:appD:L-positive-goal} holds for $\frac{1}{100}\le\alpha\le1+\sqrt2$.  Together with the small-$\alpha$ check, it holds for all $0<\alpha\le1+\sqrt2$.

The tables are regenerated from the repository by running:
\begin{center}
\begin{BVerbatim}
python verify_bernstein_q_rational_cover.py
python verify_bernstein_q_rational_cover.py 
--tex-tables rational_cover_tables.tex
\end{BVerbatim}
\end{center}

All endpoint boxes, corner substitutions, Bernstein conversions, and sign tests are performed with exact rational arithmetic; SymPy is used to take exact floors and ceilings of the algebraic endpoint values.

Finally, since each $D_j(\alpha)$ in \eqref{eq:appD:Dj} is positive,
\eqref{eq:appD:L-positive-goal} gives
\begin{equation*}
    \BB_j(\alpha)<0,
    \qquad 0\le j\le5,
    \qquad 0<\alpha\le1+\sqrt2,
\end{equation*}
which is \eqref{bound:BB}.


\section{The sign of Bernstein coefficients for the explosion barrier: interval arithmetic proof}
\label{app:CA}

This appendix describes the computer-assisted verification of the bounds \eqref{bound:BB} using interval arithmetic.  The purpose of the appendix is to describe the computational procedure, the data read by the program, and the routines used to verify the inequalities.  All computer-assisted steps are carried out using rigorous interval arithmetic: every interval produced by the program contains the exact value represented by the corresponding formula.
The programs are written in \textsc{Python} and are run in a \sage{}~\cite{sagemath} environment.  The code uses \texttt{sage.all} for exact rational arithmetic and for the \texttt{RealBallField}/\texttt{RealIntervalField} interval arithmetic. All source code and data mentioned in this appendix are available in the folder \path{computer_assisted_bounds} at \url{https://github.com/giorgiocialdea/new_class_explosions_ca}.  The main script is
\begin{center}
\path{computer_assisted_bounds/script.py}.
\end{center}
The file \texttt{lemmas.py} is the main verification file.  It defines the six functions $\tilde{\BB}_0,\ldots,\tilde{\BB}_5$ (see their definitions in \eqref{Bernstein:coeff:exp:rescaled}), reads the target bounds from \texttt{supplementary\_data/bounds.txt}, and proves these bounds on
$b:=\sqrt{\alpha} \in[0,14/9]$ by rigorous interval arithmetic with adaptive subdivision.
For the structure of this Appendix, we have taken inspiration from~\cite{CGP26,CGS25,BCG25}.
\subsection{Supplementary data}

The only numerical data file needed for this verification is
\begin{center}
\path{computer_assisted_bounds/supplementary_data/bounds.txt}.
\end{center}
It contains the rounded constants displayed in Table~\ref{table:bb}.  The file is organized as a sequence of entries, one per line.  Each nonempty line has the form
\begin{center}
\texttt{label|value|side|description}.
\end{center}
Here \texttt{label} is the identifier used in the code, \texttt{value} is the numerical bound, \texttt{side} is either \texttt{lower} or \texttt{upper}, and \texttt{description} records the corresponding row of Table~\ref{table:bb}.
The routine \texttt{load\_bounds()} converts the values into elements of \RBF{} and stores them in dictionaries indexed by their labels.

The bounds read by the program are the following.
\begin{longtable}{lll}
\caption{Rounded bounds for the $\tilde{\BB}_j$ coefficients.}\label{table:bb}\\
\toprule
Problem & coefficient & rounded bounds \\
\midrule
\endhead
\texttt{b5b0} & $\tilde{\BB}_0$ & $0.009 < \tilde{\BB}_0 < 127.2$ \\
\texttt{b5b1} & $\tilde{\BB}_1$ & $4.1 < \tilde{\BB}_1 < 12508.1$ \\
\texttt{b5b2} & $\tilde{\BB}_2$ & $60.3 < \tilde{\BB}_2 < 1625863.7$ \\
\texttt{b5b3} & $\tilde{\BB}_3$ & $4.6 < \tilde{\BB}_3 < 571287.3$ \\
\texttt{b5b4} & $\tilde{\BB}_4$ & $4.9 < \tilde{\BB}_4 < 44572.8$ \\
\texttt{b5b5} & $\tilde{\BB}_5$ & $0.006 < \tilde{\BB}_5 < 1304.1$ \\
\bottomrule
\end{longtable}

\subsection{Objects, methods and routines}

We briefly describe the main objects and routines used in the implementation.

\textbf{Formulae.}  The functions \texttt{b5b0()}, \ldots, \texttt{b5b5()} in \path{lemmas.py} build the six rescaled coefficients $\tilde{\BB}_j$ from
\eqref{Bernstein:coeff:exp:rescaled}.  The variable used by the program is $b=\sqrt{\alpha}$, with
\[
    b\in[0,14/9].
\]
This rational interval covers the parameter range used in the paper after the change of variables $\alpha=b^2$.  The common quantities $\crr$ and $V_2$ are constructed from \eqref{formula:NNN:1} and \eqref{def:V2}.  They are wrapped as
shared subexpressions, so that each is evaluated only once on each interval.

\textbf{Expression objects.}  The file \path{lemmas.py} defines a small one-variable expression class.  It supports algebraic operations, division,
square roots, integer powers, and automatic differentiation.  Coefficients and interval endpoints are parsed as exact rational numbers.  Square roots are evaluated only when their radicands are verified nonnegative, and divisions are evaluated only when the denominator is verified away from zero.

\textbf{Methods.}  For an interval $I=[x_0-r,x_0+r]$, the program first tries
the direct interval evaluation
\[
    f(I)\subset J.
\]
It also uses the first-order Taylor enclosure
\begin{equation}
    f(I)\subset f(x_0)+[-r,r] f'(I),
    \label{eq:taylor1-enclosure}
\end{equation}
where the derivative enclosure is obtained by automatic differentiation of the same expression tree.

\textbf{Routine.}  The routine \texttt{prove\_table\_row()} verifies one row of
Table~\ref{table:bb}.  Let $[a,b]\subset\mathbb R$, let $L,U\in\mathbb R$, and
let $f:[a,b]\to\mathbb R$.  The routine verifies whether
\begin{equation}
    L < f(x) < U,\qquad x\in[a,b].
    \label{eq:row-bound-check}
\end{equation}
The algorithm starts with a list of intervals covering $[a,b]$, after inserting
the fixed rational cuts
\[
    2^{-10},\quad 2^{-6},\quad 2^{-3},\quad 2^{-1}.
\]
While the list is nonempty, an interval $I$ is removed from the list and the program tries to enclose $f(I)$ by direct interval arithmetic and by
\eqref{eq:taylor1-enclosure}.  If one of these enclosures lies strictly inside $(L,U)$, the interval is accepted.  Otherwise the interval is bisected and the two subintervals are returned to the list, unless the prescribed maximum depth has already been reached (for our computation, with the parameters in \path{parameters.py}, the maximum depth is never reached).  If all intervals are accepted, then \eqref{eq:row-bound-check} holds on the whole domain.

\subsection{Details of the proof}

We can now give a proof of Lemma~\ref{lemma:Bernstein:q}.

\textbf{Proof of Lemma~\ref{lemma:Bernstein:q}.}  The proof is carried out by the function \texttt{bernstein\_bounds()} in \path{lemmas.py}.  This function
constructs the six expressions $\tilde{\BB}_j$ by calling
\texttt{b5b0()}, \ldots, \texttt{b5b5()}.  For each $j=0,\ldots,5$, it reads from \path{bounds.txt} the constants $L_j$ and $U_j$ displayed in Table~\ref{table:bb}, and calls \texttt{prove\_table\_row()} to verify
\[
    L_j < \tilde{\BB}_j(b) < U_j,\qquad b\in[0,14/9].
\]
All operations in this computation are performed with interval-valued objects, so each accepted interval gives a rigorous enclosure of the exact expression on
that interval.  Once the adaptive subdivision has accepted every interval in the covering of $[0,14/9]$, the six inequalities in Table~\ref{table:bb} are proved.  In particular, each lower bound is positive, and therefore $\tilde{\BB}_j>0$ for $j=0,\ldots,5$.  This proves \eqref{bound:BB}.

The computation is reproduced from the repository root by running
\begin{center}
\begin{BVerbatim}
cd computer_assisted_bounds
python script.py
\end{BVerbatim}
\end{center}
inside a \sage{} environment.  The default command prints only the lemma-level verification status and the final substitution check.  For diagnostic output, one may run
\begin{center}
\begin{BVerbatim}
python script.py --verbose
\end{BVerbatim}
\end{center}
to print row-level progress and the bounds read from \path{bounds.txt}.  The
command
\begin{center}
\begin{BVerbatim}
python script.py --print-subintervals
\end{BVerbatim}
\end{center}
prints every accepted subinterval together with its local enclosure, its lower and upper margins relative to the displayed bounds, and the corresponding relative margins.

\section*{Acknowledgements}
The work of JC was partially supported by NSF Grant DMS--2622139.
The work of GC was partially supported by the Collaborative NSF grant DMS--2307681 and a Simons Dissertation Fellowship.
The work of SS was partially supported by the Collaborative NSF grant DMS--2307680.
The work of VV was partially supported by the Collaborative NSF grant DMS--2307681 and a Simons Investigator Award.

The authors acknowledge the use of \textsc{GPT~5.5 Pro} to assist in \emph{drafting and refining the \textsc{Python} and~\textsc{SageMath} codes} used in the computer-assisted verifications from Appendices~\ref{app:rational-cover-bernstein-q} and~\ref{app:CA}; these codes were verified by the authors.
The authors acknowledge the usage of~\textsc{Mathematica~14} for performing, and the usage of the~\textsc{SymPy} library in~\textsc{Python} for verifying, the \emph{symbolic calculations} which produced the coefficients~$\{\cp_i\}_{i=0}^{6}$ in equation~\eqref{coeff:PP} and~$\{ \beta_i \}_{i=0}^8$ in equation~\eqref{bernstein:12} of Appendix~\ref{app:coef:p}, respectively $\{\cq_i\}_{i=0}^{5}$ in equation~\eqref{coeff:qq:i} and~$\{\BB_i\}_{i=0}^5$ in equation~\eqref{coeff:bb:i} of Appendix~\ref{app:coef:b}.


\end{document}